\documentclass{amsart}
\usepackage{amsmath} 
\usepackage{amssymb} 
\usepackage{amsthm} 
	\theoremstyle{definition}
	\newtheorem{thm}{Theorem}[section]
	\newtheorem{prop}[thm]{Proposition}
	
	\newtheorem{lemma}[thm]{Lemma}
	\newtheorem{defn}[thm]{Definition}
	\newtheorem{exmp}[thm]{Example}
	\newtheorem{exmps}[thm]{Examples}
	\newtheorem{const}[thm]{Construction}
	\newtheorem{rmk}[thm]{Remark}
	\newtheorem{nota}[thm]{Notation}

\usepackage{chngcntr} 
	\counterwithin{figure}{section} 
\usepackage{cite} 
\usepackage{enumerate} 
\usepackage{etex} 
\usepackage{fancyhdr} 
\usepackage{float} 
\usepackage[utf8]{inputenc} 
\usepackage[margin=1.2in]{geometry} 
\usepackage{graphicx} 
\usepackage{multicol} 
\usepackage{pgfplots} 
\usepackage{subcaption} 
\usepackage{tikz} 
\usepackage{tikz-cd} 
\usepackage{xcolor} 

\newcommand{\bC}{\mathbb{C}}

\newcommand{\bN}{\mathbb{N}}

\newcommand{\bR}{\mathbb{R}}

\newcommand{\bT}{\mathbb{T}}

\newcommand{\bZ}{\mathbb{Z}}

\newcommand{\sH}{\mathcal{H}}

\newcommand{\sK}{\mathcal{K}}

\newcommand{\sP}{\mathcal{P}}

\newcommand{\sR}{\mathcal{R}}
\newcommand{\sS}{\mathcal{S}}
\newcommand{\sT}{\mathcal{T}}

\newcommand{\eps}{\varepsilon}
\newcommand{\vphi}{\varphi}
\newcommand{\spec}{\text{sp}}

\newcommand{\setdiv}{\,\big|\,}
\newcommand{\ran}{\text{ran}}
\newcommand{\id}{\text{id}}
\newcommand{\orb}{\text{orb}}

\newcommand{\dprime}{^{\prime\prime}}

\newcommand{\systemfull}{\sS=( T,( X_t )_{ t = 1 , \ldots , T }, ( K_t )_{ t = 1 , \ldots , T }, ( Y_{ t , k } )_{ t = 1 , \ldots , T ; k = 1 , \ldots , K_t } , ( J_{ t , k } )_{ t = 1 , \ldots , T ; k = 1 , \ldots , K_t } ) }
\newcommand{\systembasic}{\sS=( T, ( X_t ) , ( K_t ), ( Y_{ t , k } ) , ( J_{ t , k } ) ) }
\newcommand{\systemarg}[1]{ \sS^{ #1 } = ( T^{ #1 } , ( X_t^{ #1 } ) , ( K_t^{ #1 } ) , ( Y_{ t , k }^{ #1 } ) , ( J_{ t , k }^{ #1 } ) ) }

\newenvironment{psmallmatrix}
  {\left(\begin{smallmatrix}}
  {\end{smallmatrix}\right)}
  
\address{Department of Mathematics, Ben Gurion University of the Negev,
P.O.B. 653, Be'er Sheva 84105, Israel}
\email{PHerstedt@gmail.com}

\allowdisplaybreaks

\begin{document}

\title{AT-algebras from fiberwise essentially minimal zero-dimensional dynamical systems}

\author{Paul Herstedt}
\thanks{This research was supported by the Israel Science Foundation grant no.~476/16.}  

\maketitle

\begin{abstract}
We introduce a type of zero-dimensional dynamical system (a pair consisting of a totally disconnected compact metrizable space along with a homeomorphism of that space), which we call ``fiberwise essentially minimal", and we prove that the associated crossed product C*-algebra of such a system is an AT-algebra. Under the additional assumption that the system has no periodic points, we prove that the associated crossed product C*-algebra has real rank zero, which tells us that such C*-algebras are classifiable by K-theory. The associated crossed product C*-algebras to these nontrivial examples are of particular interest because they are non-simple (unlike in the minimal case).
\end{abstract}

\section*{Introduction}

The underlying goal of much research in the field of $ C^* $-algebras in the last few decades is $ K $-theoretic classification, which is the goal of showing that, for various nice classes of $ C^* $-algebras, isomorphism at the level of $ K $-theory implies isomorphism of the $ C^* $-algebras. This work was pioneered by George Elliott in the 70's (see \cite{Elliott76} for his groundbreaking work classifying the simple AF-algebras introduced by Bratteli in \cite{Bratteli72}). Classification of simple $ C^* $-algebras has been separately widely studied over the last few decades but as of recently is complete due to the work of many hands; for the most definitive results, see \cite{Phillips99} for the purely infinite case and see \cite{ElliotGongLinZiu16} for the stably finite case. There is still much to be discovered in the non-simple case. Some early results include the work of Elliott, Marius Dadarlat, and Guihua Gong in the 90's (see \cite{Elliott93}, \cite{ElliottGong96}, and \cite{DadarlatGong97}). This work classifies (not necessarily simple) real rank zero AH-algebras of slow dimension growth. In particular, this classification result includes A$\bT$-algebras of real rank zero, and examples of such $ C^* $-algebras arise from constructions introduced in this paper.

This paper is an expansion of work done in the 90's on classifying the crossed product $ C^* $-algebras associated to minimal zero-dimensional dynamical systems. In 1989, Ian Putnam (in \cite{Putnam89}) showed that these crossed products have computable ``large AF-subalgebras" whose $ K_0 $ group matches that of the crossed product. The following year, Putnam showed (in \cite{Putnam90}) that these crossed products are A$\bT$-algebras of real rank zero and are hence classified by the previously computed $ K $-theory (see \cite{Elliott93}, \cite{ElliottGong96}, and \cite{DadarlatGong97}). In 1992, Putnam, together with Richard Herman and Christian Skau, expanded this result to ``essentially minimal" zero-dimensional systems with no periodic points (see Definition \ref{defnEssentiallyMinimal} of this paper) in their paper \cite{HermanPutnamSkau92}. We further expand this result in Theorems \ref{thmMainTheorem} and \ref{thmRR0}, proving that a larger class of zero-dimensional dynamical systems (fiberwise essentially minimal with no periodic points; see Definition \ref{defnFiberwiseEssentiallyMinimal}) have associated crossed products that are A$\bT$-algebras of real rank zero. Theorem \ref{thmMainTheorem} in fact shows that if we allow periodic points, we still get that the crossed products are A$\bT$-algebras; to our knowledge, this result is not written down anywhere even in the essentially minimal case, as the primary concern of previous papers have been $ C^* $-algebras of real rank zero, as this was one of the elements of a sufficient condition for classification at the time.

\textbf{Acknowledgements:} This work was done at the University of Oregon during the author's time as a Ph.D. student and at the Ben Gurion University of the Negev, funded by the Israel Science Foundation (grant no.476/16). The author would like to thank his Ph.D. advisor Chris Phillips for proposing the idea for this project and for many helpful conversations and edits along the way, as well as Ian Putnam for pointing out a simplification to the proof of Theorem \ref{thmRR0}, and finally his current post doctoral supervisor Ilan Hirshberg for his help with editing the final pieces.

\section{Preliminaries}\label{sectionPrelim}

This section introduces the terms that will be used to prove the main results (Section \ref{sectionTheorems}) of this paper. Some basic facts will also be proved.

Let $ X $ be a totally disconnected compact metrizable space and let $ h : X \to X $ be a homeomorphism of $ X $. We call $ ( X , h ) $ a \emph{zero-dimensional system}. Let $ \alpha $ be the automorphism of $ C ( X ) $ induced by $ h $; that is, $ \alpha $ is defined by $ \alpha ( f ) ( x ) = f ( h^{ -1 } ( x ) ) $ for all $ f \in C ( X ) $ and all $ x \in X $. Then we denote the crossed product of $ C ( X ) $ by $ \alpha $ by $ C^* ( \bZ , X , h ) $ (or, less commonly, $ C^* ( \bZ , C ( X ) , \alpha ) $). We denote the ``standard unitary" of $ C^* ( \bZ , X , h ) $ by $ u $, which is a unitary element of $ C^* ( \bZ , X , h ) $ that satisfies $ u f u^* = \alpha ( f ) $ for all $ f \in C ( X ) $.

We will use the disjoint union symbol $ \bigsqcup $ to denote unions of disjoint sets. We will not always say explicitly that the sets in this union are disjoint, as this will be implied by the notation.

\begin{defn}\label{defnPartition}
Let $ X $ be a totally disconnected compact metrizable space. We define a \emph{partition} $ \sP $ of $ X $ to be a finite set of mutually disjoint compact open subsets of $ X $ whose union is $ X $. We denote by $ C ( \sP ) $ the subset of $ C ( X ) $ consisting of functions that are constant on elements of $ \sP $. 
\end{defn}

For ease of notation, we will often denote a sequence $ ( x_n )_{ n = 1 }^\infty $ just by $ ( x_n ) $. That such an object is a sequence will be clear from context. 

\begin{defn}\label{defnGeneratingSequenceOfPartitions}
We say that a sequence $ ( \sP_n ) $ of partitions of $ X $ is a \emph{generating sequence of partitions of $ X $} if $ \sP_{ n + 1 } $ is finer than $ \sP_n $ for all $ n \in \bZ_{ > 0 } $, and if for every $ x \in X $, there is a sequence $ ( V_n ) $ such that $ V_n \in \sP_n $ for all $ n \in \bZ_{ > 0 } $ and such that $ \bigcap_{ n = 1 }^\infty V_n = \{ x \} $. 
\end{defn}

The terminology ``generating sequence" comes from the fact that such a sequence of partitions generates the topology of $ X $.

\begin{nota}
Let $ x \in X $. We say that $ x $ is a \emph{periodic point} if there is a nonzero integer $n$ such that $ h^n ( x ) = x $; otherwise, we say that $ x $ is an \emph{aperiodic point}. We denote the orbit of $ x $ by $ \orb ( x ) = \{ h^n ( x ) \setdiv n \in \bZ \} $.
\end{nota}

We say that a nonempty closed subset $ Y $ of $ X $ is a \emph{minimal set} if it is $ h $-invariant and has no nonempty $ h $-invariant proper closed subsets. By Zorn's lemma, minimal sets exist for every zero-dimensional system. We use the following definition from \cite{HermanPutnamSkau92}.

\begin{defn}\label{defnEssentiallyMinimal}
We say that a dynamical system $ ( X , h ) $ is an \emph{essentially minimal system} if it has a unique minimal set. 
\end{defn}

We say that $ ( X , h ) $ is a \emph{minimal system} if the unique minimal set is $ X $ (this is a simple restatement of the classic definition). Note that essentially minimal systems have no nontrivial compact open $ h $-invariant subsets, since such a set and its complement would contain disjoint minimal sets. Also note that we do not assume that all points must be aperiodic, so $ \bZ / n\bZ $ with the shift homeomorphism is an example of a minimal zero-dimensional system.

\begin{exmp}
We give a standard example and non-example of essentially minimal zero-dimensional systems.
\begin{enumerate}[(a)]
\item Let $ X = \bZ \cup \{ \infty \} $ be the one-point compactification of the integers and let $ h $ be the shift homeomorphism, defined by $ h ( x ) = x + 1 $ for $ x \in \bZ $ and $ h ( \infty ) = \infty $. Then $ ( X , h ) $ is an essentially minimal zero-dimensional system, which can be observed by the fact that the only nonempty closed $ h $-invariant proper subset of $ X $ is $ \{ \infty \} $.
\item Let $ X = \bZ \cup \{ \infty \} \cup \{ -\infty \} $ be the two-point compactification of the integers and let $ h $ be the shift homeomorphism, defined by $ h ( x ) = x + 1 $ for $ x \in \bZ $, $ h ( \infty ) = \infty $, and $ h ( -\infty ) = -\infty $. Then $ ( X , h ) $ is not essentially minimal, as there are two minimal sets, $ \{ \infty \} $ and $ \{ -\infty \} $.
\end{enumerate}
\end{exmp}

The following definition is standard in dynamical systems.

\begin{defn}\label{defnLambdaU}
Let $ ( X , h ) $ be a zero-dimensional system and let $ U \subset X $. We define the map $ \lambda_U : U \to \bZ_{ > 0 } \cup \{ \infty \}$ by $ \lambda_U ( x ) = \inf\{ n \in \bZ_{ > 0 } \setdiv h^n ( x ) \in U \}$ if this infimum exists, and $ \lambda_U ( x ) = \infty $ otherwise. We call this map the \emph{first return time map} of $ U $.
\end{defn}

We now provide a few facts about the first return time map. The following proposition is standard and often not proved, so we provide the proof for the reader.

\begin{prop}\label{propLambdaUContinuous}
Let $ ( X , h ) $ be a zero-dimensional system, let $ U $ be a compact open subset of $ X $, and let $ \lambda_U $ be as in Definition \ref{defnLambdaU}. Then $ \lambda_U $ is continuous.
\end{prop}

\begin{proof}
Let $ x_0 \in U $ satisfy $ \lambda_U ( x_0 ) = m $ for some $ m \in \bZ_{ > 0 } $. We show that $ \lambda_U $ is continuous at $ x_0 $. We claim that since $U$ is open, $\lambda_U$ is upper semi-continuous at $ x_0 $. Since $ h^m ( x_0 ) \in U $, we have $ x_0 \in h^{ -m } ( U ) $. Then, since $ h^{ -m } $ is a homeomorphism, $ h^{ -m } ( U ) $ is an open set in $ X $, and so is $ h^{ -m } ( U ) \cap U $. For all $ x \in h^{ -m } ( U ) \cap U $, we have $ h^m ( x ) \in U $, and so $ \lambda_U ( x ) \leq m = \lambda_U ( x_0 ) $. Thus, $ \lambda_U $ is upper semicontinuous at $ x_0 $. We claim that since $ U $ is closed, $ \lambda_U $ is lower semi-continuous at $ x_0 $. Suppose that $ \lambda_U $ is not lower semi-continuous at $ x_0 $. This means that there is a sequence $ ( x_n ) $ in $ U $ such that $ x_n \to x_0 $ and $ \lambda_U ( x_n ) < m $ for all $ n \in \bZ_{ > 0 } $. So, since $ \{ \lambda_U ( x_n ) \setdiv n \in \bZ_{ >0 } \} \subset \{ 1 , \ldots, m - 1 \} $, which is a finite set, there is a subsequence $ ( x_{ n_k } ) $ of $ ( x_n ) $ such that $( \lambda_U ( x_{ n_k } ) ) $ is a constant sequence. Say that $ \lambda_U ( x_{ n_k } ) = j $ for all $ k \in \bZ_{ > 0 } $. But then $ ( x_{ n_k } ) $ is a sequence in $ h^{ -j } ( U ) $, which is closed since $ h^{ -j } $ is continuous. Since $ x_{ n_k } \to x_0 $, we conclude that $ x_0 \in h^{ -j } ( U ) $. But then $ m \leq j $, which is a contradiction to $ \lambda_U ( x_n ) < m $ for all $ n \in \bZ_{ > 0 } $. Thus, $ \lambda_U $ is lower semi-continuous at $ x_0 $, and therefore continuous at $ x_0 $.

Now let $ x_0 \in U $ satisfy $ \lambda_U ( x_0 ) = \infty $. We now show that $ \lambda_U $ is continuous at $ x_0 $. Suppose not; that is, suppose there is a sequence $ ( x_n ) $ in $ U $ converging to $ x_0 $ such that that $ \lim_{ n \to \infty } \lambda_U ( x_n ) \neq \infty $. By passing to a subsequence, we may assume that $ ( \lambda_U ( x_n ) ) $ is bounded, and then by passing to another subsequence, we may assume that $ ( \lambda_U ( x_n ) ) $ is constant, say equal to $k$. This means that $ h^k ( x_n ) \in U $ for all $ n \in \bZ_{ > 0 } $. But then since $ h^k $ is continuous, $ \lim_{ n \to \infty } h^k ( x_n ) = h^k ( x_0 ) $. Since $ U $ is closed, $ \lim_{ n \to \infty } h^k ( x_n ) \in U $. Thus, $ h^k ( x_0 ) \in U $, a contradiction to $ \lambda_U ( x_\infty ) = \infty $. Thus, $ \lambda_U $ is continuous at $ x_0 $.

Altogether, we see that $ \lambda_U $ is continuous.
\end{proof}

The following proposition is contained in Theorem 1.1 of \cite{HermanPutnamSkau92}.

\begin{prop}\label{propEssentiallyMinimalReturnTimeMap}
Let $ ( X , h ) $ be an essentially minimal zero-dimensional system, let $ Y $ be its unique minimal set, let $ y \in Y $, let $ U $ be a compact open neighborhood of $ y $, and let $ \lambda_U $ be as in Definition \ref{defnLambdaU}. Then we have the following:
\begin{enumerate}[(a)]
\item \label{propEssentiallyMinimalReturnTimeMap(a)} $ \bigcup_{ j \in \bZ } h^j ( U ) = X $.
\item \label{propEssentiallyMinimalReturnTimeMap(b)} $ \ran( \lambda_U ) $ is a finite subset of $ \bZ_{ > 0 } $.
\item \label{propEssentiallyMinimalReturnTimeMap(c)} $ U = \{ h^{ \lambda_U ( x ) } ( x ) \setdiv x \in U \} $.
\item \label{propEssentiallyMinimalReturnTimeMap(d)} $ \bigcup_{ j \in \bZ_{ > 0 } } h^j ( U ) = \bigcup_{ j \in \bZ_{ < 0 } } h^j ( U ) = X $.
\end{enumerate}
\end{prop}

The following definition is a key definition of this paper. It was constructed by extrapolating the properties of what a zero-dimensional dynamical system must have for the technique of the proof of Theorem 2.1 of \cite{Putnam90} to apply. This theorem proves that minimal zero-dimensional dynamical systems have A$\bT$-algebra crossed products, and is the main inspiration for this project (Theorem \ref{thmMainTheorem} of this paper is a generalization of this).

\begin{defn}\label{defnSystemOfFiniteReturnTimeMaps}
Let $ ( X , h ) $ be a zero-dimensional system and let $ \sP $ be a partition of $ X $. We define a 
\emph{system of finite first return time maps subordinate to $ \sP $} to be a tuple
\[
\systemfull
\]
such that:
\begin{enumerate}[(a)]
\item \label{defnSystemOfFiniteReturnTimeMaps(a)} We have $ T \in \bZ_{ > 0 } $. 
\item \label{defnSystemOfFiniteReturnTimeMaps(b)} For each $ t \in \{ 1 , \ldots , T \} $, $ X_t $ is a compact open subset of $ X $ contain in an element of $ \sP $.
\item \label{defnSystemOfFiniteReturnTimeMaps(c)} For each $ t \in \{ 1 , \ldots , T \} $, $ K_t \in \bZ_{ > 0 } $. 
\item \label{defnSystemOfFiniteReturnTimeMaps(d)} For each $ t \in \{ 1 , \ldots , T \} $ and each $ k \in \{ 1 , \ldots , K_t \} $, $ Y_{ t , k } $ is a compact open subset of $X_t$. Moreover, for each $ t \in \{ 1 , \ldots , T\}$, $ \{ Y_{ t , 1 } , \ldots , Y_{ t , K_t } \} $ is a partition of $X_t$; that is, 
\[
\bigsqcup_{ k = 1 }^{ K_t } Y_{ t , k } = X_t .
\]
\item \label{defnSystemOfFiniteReturnTimeMaps(e)} For each $ t \in \{ 1 , \ldots , T \} $ and each $ k \in \{ 1 , \ldots , K_t \} $, $ J_{ t , k } \in \bZ_{ > 0 } $. Using Definition \ref{defnLambdaU}, we can write $ \{ J_{ t , k } \} = \lambda_{ X_t } ( Y_{ t , k } ) $. Moreover, for each $ t \in \{ 1 , \ldots , T \} $, $ \{ h^{ J_{ t , 1 } } ( Y_{ t , 1 } ) , \ldots , h^{ J_{ t , K_t } } ( Y_{ t , K_t } ) \} $ is a partition of $ X_t $; that is, 
\[
\bigsqcup_{ k = 1 }^{ K_t } h^{ J_{ t , k } } ( Y_{ t , k } ) = X_t .
\]
\item \label{defnSystemOfFiniteReturnTimeMaps(f)} The set
\[
\sP_1 ( \sS ) = \big\{ h^j ( Y_{ t , k } ) \setdiv \mbox{$ t \in \{ 1 , \ldots , T \} $, $ k \in \{ 1 , \ldots , K_t \} $, and $ j \in \{ 0 , \ldots , J_{ t , k } - 1 \} $} \big\}
\]
is a partition of $X$. Note that this combined with condition (\ref{defnSystemOfFiniteReturnTimeMaps(e)}) also implies 
\[
\sP_2 ( \sS ) = \big\{ h^j ( Y_{ t , k } ) \setdiv \mbox{$ t \in \{ 1 , \ldots , T \} $, $ k \in \{ 1 , \ldots , K_t \} $, and $ j \in \{ 1 , \ldots , J_{ t , k } \} $} \big\}
\]
is a partition of $ X $.
\end{enumerate}
\end{defn}

We will often suppress the notation of a system of finite first return time maps subordinate to $ \sP $ by writing $ \systembasic $. When writing multiple systems in the same proof or statement (which will often be done), we will use superscripts on all elements of the system; for example, we may write $ \systemarg{ \prime } $.

\begin{rmk}\label{rmkSystemOfFiniteReutrnTimeMaps}
We make some comments about what the objects in Definition \ref{defnSystemOfFiniteReturnTimeMaps} mean and where the name of the system comes from. Adopt the notation of Definition \ref{defnSystemOfFiniteReturnTimeMaps}. Then for each $ t \in \{ 1 , \ldots , T \} $, each $ k \in \{ 1 , \ldots , K_t \} $, and each $ x \in Y_{ t , k } $, we have $ \lambda_{ X_t } ( x ) = J_{ t , k } $. So $ J_{ t , k } $ is the ``first return time" of each element of $ Y_{ t , k } $ to $ X_t $. We enumerate some more points below.
\begin{enumerate}[(a)] 
\item The number of ``bases" (see (\ref{rmkSystemOfFiniteReutrnTimeMaps(b)}) below) of $ \sS $ is $ T $. 
\item \label{rmkSystemOfFiniteReutrnTimeMaps(b)} The ``bases" of $ \sS $ is $ ( X_t )_{ t = 1 }^T $. These are the domains of the first return time maps above. That $ \sS $ is subordinate to $ \sP $ means that, for each $ t \in \{ 1 , \ldots , T \}$, $ X_t $ is contained in an element of $ \sP $.
\item For each $ t \in \{ 1 , \ldots , T \} $, $ K_t $ is the size of the partition of $ X_t $ into sets with constant first return time to $ X_t $.
\item For each $ t \in \{ 1 , \ldots , T \} $ and each $ k \in \{ 1 , \ldots , K_t \} $, $ Y_{ t , k } $ is a piece of $ X_t $ that has constant first return time to $ X_t $.
\item For each $ t \in \{ 1 , \ldots , T \} $ and each $ k \in \{ 1 , \ldots , K_t \} $, $ J_{ t , k } $ is the first return time of $ Y_{ t , k } $ to $ X_t $.
\end{enumerate}
\end{rmk}

\begin{exmps}\label{exmpsSystemsOfFiniteFirstReturnTimeMaps}
We provide some examples that illustrate Definition \ref{defnSystemOfFiniteReturnTimeMaps}.
\begin{enumerate}[(a)]
\item \label{exmpsSystemsOfFiniteFirstReturnTimeMaps(a)} If $ ( X , h ) $ is a minimal zero-dimensional system, then for any partition $ \sP $ of $ X $, $ ( X , h ) $ admits a system of finite first return time maps subordinate to $ \sP $. It is shown in the proof of Theorem 2.1 of \cite{Putnam90} that we can take $ T = 1 $, and $ X_1 $ can be any compact open subset of $ X $ that is contained in an element of $ \sP $. This can be explained by using a couple of basic properties of minimal systems. Every point in a compact open set will come back to that compact open set in finite time; this gives us $ K_1 $ and $ J_{ 1 , k } $, and we set $ Y_{ 1 , k } = \lambda_{ X_t }^{ -1 } ( J_{ 1 , k } ) $. Then, we see $ \bigsqcup_{ j = 0 }^{ J_{ t , 1 } - 1 } \bigsqcup_{ k = 1 }^{ K_t } h^j ( Y_{ 1 , k } ) $ is an invariant compact open subset of $ X $ (by Proposition \ref{propPropertiesOfX_t}(\ref{propPropertiesOfX_t(a)})), and therefore by minimality is all of $ X $.

\item \label{exmpsSystemsOfFiniteFirstReturnTimeMaps(b)} In the comments preceding Theorem 8.3 of \cite{HermanPutnamSkau92}, it is implicitly stated that if $ ( X , h ) $ is an essentially minimal zero-dimensional system with no periodic points, then for any partition $ \sP $ of $ X $, $ ( X , h ) $ admits a system of finite first return time maps subordinate to $ \sP $. This can be shown using the same technique of that of the proof of Theorem 2.1 of \cite{Putnam90} by taking $ T = 1 $ and taking $ X_1 $ to be any compact open subset of $ X $ that intersects the unique minimal set of $ ( X , h ) $ and is contained in an element of $ \sP $. Compare Proposition \ref{propPropertiesOfX_t} with Proposition \ref{propEssentiallyMinimalReturnTimeMap}(\ref{propEssentiallyMinimalReturnTimeMap(d)}).

\item \label{exmpsSystemsOfFiniteFirstReturnTimeMaps(c)} For an explicit example, let $ X = \bZ \cup \{ \infty \} $ be the one point compactification of the integers and let $ h $ be the shift homeomorphism, defined by $ h ( x ) = x + 1 $ for all $ x \in \bZ $ and $ h ( \infty ) = \infty $. Let $ \sP $ be a partition of $ \bZ $ and let $ U  \in \sP $ be the element that contains $ \infty $. By the topology of $ X $, there is a compact open set $ X_1 \subset U $ such that $ X_1 = \left( \left( ( - \infty , a ] \sqcup [ b , \infty ) \right) \cap \bZ \right) \sqcup \{ \infty \} $. Set $ T = 1 $, set $ K_1 = 2 $, set $ Y_{ 1 , 1 } = \{ \infty \} \cup \left( \left( ( - \infty , a - 1 ] \cup [ b , \infty ) \right) \cap \bZ \right ) $, set $ Y_{ 1 , 2 } = \{ a \} $, set $ J_{ 1 , 1 } = 1 $, and set $ J_{ 1 , 2 } =  b - a $. Then $ \sS $ is a system of finite first return time maps subordinate to $ \sP $. See this illustrated below, where $ X $ is represented as a subset of a circle, where the north pole is $ \infty $.
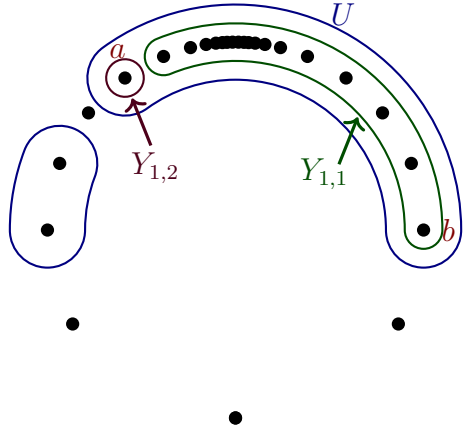
\begin{center}
\begin{tikzpicture}[scale=2.5]
\foreach \i in {1,1.5,2,2.6,3.5,5,8,13,20,30,45,70,200}
{
\fill ({cos((180/(\i))+90)},{sin((180/(\i))+90)}) circle(1pt);
}
\foreach \i in {1,1.5,2,2.6,3.5,5,8,13,20,30,45,70,200}
{
\fill ({cos((-180/(\i))+90)},{sin((-180/(\i))+90)}) circle(1pt);
}
\draw[thick, color=blue!50!black] ( { 1.2*cos ( 126 ) } , { 1.2*sin ( 126 ) } ) arc(126:0:1.2) arc(0:-180:0.2) arc (0:126:0.8) arc(306:126:0.2);
\draw[thick, color=blue!50!black] ( { 1.2*cos ( 180 ) } , { 1.2*sin ( 180 ) } ) arc(180:{180/2.6+90}:1.2) arc({180/2.6+90}:{180/2.6-90}:0.2) arc({180/2.6+90}:180:0.8) arc(0:-180:0.2);
\node[color=blue!50!black] at ( 0.6 , 1.15 ) { {\Large $ U $ } };
\draw[thick, color=green!30!black] ( { 1.1*cos ( 180/8+90 ) } , { 1.1*sin ( 180/8+90 ) } ) arc({180/8+90}:0:1.1) arc(0:-180:0.1) arc (0:{180/8+90}:0.9) arc({180/8+270}:{180/8+90}:0.1);
\draw[thick, color=purple!40!black] ( { 1.1*cos (126) } , { 1.1*sin(126) } ) arc( 126:486:0.1);
\node[color=red!50!black] at ( -0.6 , 0.95 ) { {\Large $ a $ } };
\node[color=red!50!black] at ( 1.16 , 0 ) { {\Large $ b $ } };
\node[color=green!30!black] at ( 0.5 , 0.3 ) { {\Large $ Y_{ 1 , 1 } $ } };
\node[color=purple!40!black] at ( -0.4 , 0.33 ) { {\Large $ Y_{ 1 , 2 } $ } };
\draw[->, very thick, color=green!30!black] ( 0.55 , 0.35 ) -- ( 0.65 , 0.6 );
\draw[->, very thick, color=purple!40!black] ( -0.45 , 0.45 ) -- ( -0.55 , 0.7 );
\end{tikzpicture}
\end{center}
The partition $ \sP_1 ( \sS ) $ is shown below. Note that in this specific picture, $ J_{ t , 2 } = b - a = 7 $.
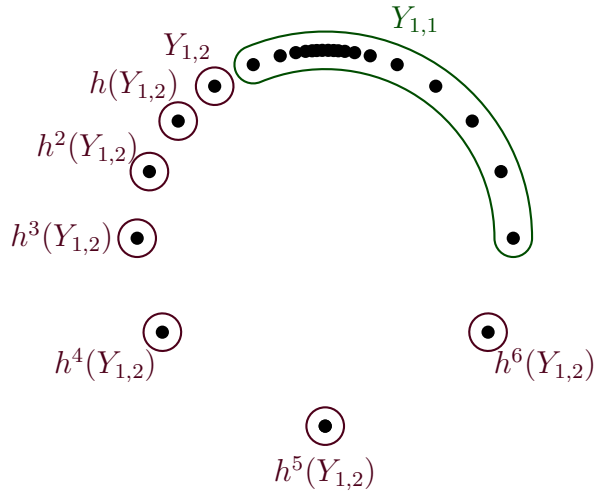
\begin{center}
\begin{tikzpicture}[scale=2.5]
\foreach \i in {1,1.5,2,2.6,3.5,5,8,13,20,30,45,70,200}
{
\fill ({cos((180/(\i))+90)},{sin((180/(\i))+90)}) circle(1pt);
}
\foreach \i in {1,1.5,2,2.6,3.5,5,8,13,20,30,45,70,200}
{
\fill ({cos((-180/(\i))+90)},{sin((-180/(\i))+90)}) circle(1pt);
}
\foreach \i in {-1.5 , 1, 1.5, 2, 2.6 , 3.5 , 5 }
{
\draw[thick, color=purple!40!black] ( { 1.1*cos (180/\i+90) } , { 1.1*sin(180/\i+90) } ) arc( {180/\i+90}:{180/\i+450}:0.1);
}
\draw[thick, color=green!30!black] ( { 1.1*cos ( 180/8+90 ) } , { 1.1*sin ( 180/8+90 ) } ) arc({180/8+90}:0:1.1) arc(0:-180:0.1) arc (0:{180/8+90}:0.9) arc({180/8+270}:{180/8+90}:0.1);
\node[color=purple!40!black] at ( { 1.25*cos( 180 / 5 + 90 ) }, { 1.25*sin( 180 / 5 + 90 ) } ) {\Large $ Y_{ 1 , 2 } $};
\node[color=purple!40!black] at ( { 1.3*cos( 180 / 3.5 + 90 ) }, { 1.3*sin( 180 / 3.5 + 90 ) } ) {\Large $ h ( Y_{ 1 , 2 } ) $};
\node[color=purple!40!black] at ( { 1.35*cos( 180 / 2.6 + 90 ) }, { 1.35*sin( 180 / 2.6 + 90 ) } ) {\Large $ h^2 ( Y_{ 1 , 2 } ) $};
\node[color=purple!40!black] at ( { 1.4*cos( 180 / 2 + 90 ) }, { 1.4*sin( 180 / 2 + 90 ) } ) {\Large $ h^3 ( Y_{ 1 , 2 } ) $};
\node[color=purple!40!black] at ( { 1.35*cos( 180 / 1.5 + 90 ) }, { 1.35*sin( 180 / 1.5 + 90 ) } ) {\Large $ h^4 ( Y_{ 1 , 2 } ) $};
\node[color=purple!40!black] at ( { 1.25*cos( 180 / 1 + 90 ) }, { 1.25*sin( 180 / 1 + 90 ) } ) {\Large $ h^5 ( Y_{ 1 , 2 } ) $};
\node[color=purple!40!black] at ( { 1.35*cos( 180 / -1.5 + 90 ) }, { 1.35*sin( 180 / -1.5 + 90 ) } ) {\Large $ h^6 ( Y_{ 1 , 2 } ) $};
\node[color=green!30!black] at ( { 1.25*cos( 180 /- 8 + 90 ) }, { 1.25*sin( 180 / -8 + 90 ) } ) {\Large $ Y_{ 1 , 1 } $};
\end{tikzpicture}
\end{center}
\end{enumerate}

Example \ref{exmpsSystemsOfFiniteFirstReturnTimeMaps}(\ref{exmpsSystemsOfFiniteFirstReturnTimeMaps(c)}) is that of an essentially minimal zero-dimensional system (with a periodic point). In fact, using Proposition \ref{propEssentiallyMinimalReturnTimeMap}, it is not hard to show that, given an essentially minimal zero-dimensional system $ ( X , h ) $ and a partition $ \sP $ of $ X $, there is a system of finite first return time maps subordinate to $ \sP $. It may be difficult to think of less trivial examples of zero-dimensional systems $ ( X , h ) $ and partitions $ \sP $ of $ X $ such that $ ( X , h ) $ admits a system of finite first return time maps subordinate to $ \sP $; the definition is a bit technical and doesn't lend itself to example generation. It turns out that there are other examples of zero-dimensional systems that admit systems of finite first return time maps; see Definition \ref{defnFiberwiseEssentiallyMinimal}, Examples \ref{exmpsFiberwiseEssentiallyMinimal}, and Theorem \ref{thmFiberwiseIffAdmitsPartitions}.
\end{exmps}

\begin{prop}\label{propPropertiesOfX_t}
Let $(X,h)$ be a zero-dimensional system, let $\sP$ be a partition of $X$, and let $\systembasic$ be a system of finite first return time maps subordinate to $\sP$. Then:
\begin{enumerate}[(a)]
\item \label{propPropertiesOfX_t(a)} For every $ t \in \{ 1 , \ldots , T \} $, $ \bigcup_{ j \in \bZ } h^j ( X_t ) $ is a compact open subset of $ X $.
\item \label{propPropertiesOfX_t(b)} We have $\bigsqcup_{ t = 1 }^T \bigcup_{ j \in \bZ } h^j ( X_t ) = X $.
\end{enumerate}
\end{prop}

\begin{proof}
We claim that 
\[
\bigsqcup_{ k = 1 }^{ K_t } \bigsqcup_{ j = 0 }^{ J_{ t , k } - 1 } h^j ( Y_{ t , k } ) = \bigcup_{ j \in \bZ } h^j ( X_t ) .
\]
Conclusion (\ref{propPropertiesOfX_t(a)}) follows from this claim since the left-hand side of the above equation is clearly compact and open, as it is the disjoint union of finitely many compact open sets. Conclusion (\ref{propPropertiesOfX_t(b)}) follows from this claim since $ \bigsqcup_{ t = 1 }^T  \bigsqcup_{ k = 1 }^{ K_t } \bigsqcup_{ j = 0 }^{ J_{ t , k } - 1 } h^j ( Y_{ t , k } ) = X $. To prove the claim, fix $ t \in \{ 1 , \ldots , T \} $ and set 
\[
W_t = \bigsqcup_{ k = 1 }^{ K_t } \bigsqcup_{ j = 0 }^{ J_{ t , k } - 1 } h^j ( Y_{ t , k } ).
\]
Note that condition (\ref{defnSystemOfFiniteReturnTimeMaps(e)}) of Definition \ref{defnSystemOfFiniteReturnTimeMaps} tells us that $ \bigsqcup_{ k = 1 }^{ K_t } h^{ J_{ t , k } } ( Y_{ t , k } ) = X_t $. Thus, $ W_t $ is an $ h $-invariant set that contains $ X_t $, meaning that it must contain $ \bigcup_{ j \in \bZ } h^j ( X_t ) $. To see that $ W_t $ contains nothing more, note that $ W_t \subset \bigcup_{ j \in \bZ } \bigsqcup_{ k = 1 }^{ K_t } h^j ( Y_{ t , k } ) = \bigcup_{ j \in \bZ } h^j ( X_t ) $. This proves the claim, and by the argument above, proves the proposition.
\end{proof}

The following proposition shows that we can ``refine" a system $ \systembasic $ of finite first return time maps without changing the bases, so that the associated partitions (Definition \ref{defnSystemOfFiniteReturnTimeMaps}(\ref{defnSystemOfFiniteReturnTimeMaps(f)})) are as fine as we want. We essentially do this by taking $ t \in \{ 1 , \ldots , T \} $ and cutting up the sets $ Y_{ t , k } $ small enough so that their iterates are always contained in elements of a given partition. This increases the value of $ K_t $ and makes more of $ J_{ t , k } $ the same number, but this is no problem, as there is nothing in Definition \ref{defnSystemOfFiniteReturnTimeMaps} that says $ J_{ t , k } $ must be distinct for each $ k \in \{ 1 , \ldots , K_t \} $ (even though a standard way of constructing systems of finite first return time maps by defining $ Y_{ t , k } = \lambda_{ X_t }^{ -1 } ( j ) $ for $ j \in \ran ( \lambda_{ X_t } ) $ would lead to one where $ J_{ t , k } $ is distinct for each $ k  \in \{ 1 , \ldots , K_t \} $). This proposition is used a number of times in the lemmas that lead to Theorem \ref{thmMainTheorem}.

\begin{prop}\label{propReturnTimeMapsGiveFinerPartitions}
Let $ ( X , h ) $ be a zero-dimensional system, let $\sP$ and $\sP'$ be partitions of $X$, and let $ \systembasic $ be a system of finite first return time maps subordinate to $\sP$. Then there is a system of finite first return time maps subordinate to $\sP$, denoted by $ \systemarg{\prime} $, such that
\begin{enumerate}[(a)]
\item \label{propReturnTimeMapsGiveFinerPartitions(a)} $ T' = T $ and, for all $ t \in \{ 1 , \ldots , T \} $, $ X_t = X_t' $.
\item \label{propReturnTimeMapsGiveFinerPartitions(b)} Using the notation of Definition \ref{defnSystemOfFiniteReturnTimeMaps}, $ \sP_1 ( \sS' ) $ and $ \sP_2 ( \sS' ) $ are finer than $ \sP' $.
\end{enumerate}
\end{prop}

\begin{proof}
Write $ \sP' = \{ U_1 , \ldots , U_R \} $. Let $ t \in \{ 1 , \ldots , T \} $ and let $ k \in \{ 1 , \ldots , K_t \} $. Define the following set: 
\[
A_{ t , k } = \big\{ Y_{ t , k }  \cap h^{ -j } ( U_r ) \setdiv \mbox{$r \in \{ 1 , \ldots , R \}$, $j \in \{ 0 , \ldots , J_{ t , k } \}$, and $Y_{ t , k } \cap h^{ -j } ( U_r ) \neq \varnothing$} \big\} . 
\]
First, notice that every element of $ A_{ t , k } $ is a subset of $ Y_{ t , k } $. Next, notice that since $ \sP' $ is a partition of $ X $, $\bigsqcup_{ r = 1 }^R ( Y_{ t , k } \cap U_r ) = Y_{ t , k }$. Thus, there is a partition $ \sP_{ t , k } $ of $ Y_{ t , k } $ such that for every $ U \in \sP_{ t , k } $ and every $ V \in A_{ t , k } $, we have $ U \subset V $ or $ U \cap V = \varnothing $. Write $ \sP_{ t , k } = \{ Y_{ t , k } ( 1 ) , \ldots , Y_{ t , k } ( M_{ t , k } ) \} $.

Set $ T' = T $. For each $ t \in \{ 1 , \ldots , T' \} $, define $ X_t' = X_t $ and $ K_t' = \sum_{ k = 1 }^{ K_t } M_{ t , k } $. For each $ t \in \{ 1 , \ldots , T' \} $ and each $ k' \in \{ 1 , \ldots , K_t' \} $, let $ k \in \{ 1 , \ldots , K_t \} $ and $ m \in \{ 1 , \ldots , M_{ t , k } \} $ satisfy $ k' = \sum_{ l = 1 }^{ k - 1 } M_{ t , l } + m $, and define a compact open set $ Y_{ t , k' }' \subset X_t $ by $ Y_{ t , k' }' = Y_{ t , k } ( m ) $ and define $ J_{ t , k' } = J_{ t , k } $.

We now show that $\systemarg{\prime}$ is a system of finite first return time maps subordinate to $\sP$ by checking the conditions of Definition \ref{defnSystemOfFiniteReturnTimeMaps}. It is clear that (\ref{defnSystemOfFiniteReturnTimeMaps(a)}), (\ref{defnSystemOfFiniteReturnTimeMaps(b)}), and (\ref{defnSystemOfFiniteReturnTimeMaps(c)}) are satisfied. Observe that for each $ t \in \{ 1 , \ldots , T' \} $, we have
\begin{align*}
\bigsqcup_{ k = 1 }^{ K_t' } Y_{ t , k }' &= \bigsqcup_{ k = 1 }^{ K_t } \bigsqcup_{ m = 1 }^{ M_{ t , k } } Y_{ t , k } ( m ) \\
&= \bigsqcup_{ k = 1 }^{ K_t } Y_{ t , k } \\
&= X_t \\
&= X_t' .
\end{align*}
Thus, condition (\ref{defnSystemOfFiniteReturnTimeMaps(d)}) is satisfied. Similarly, for each $ t \in \{ 1 , \ldots , T' \} $, we have 
\begin{align*}
\bigsqcup_{ k = 1 }^{ K_t' } h^{ J_{ t , k }' } ( Y_{ t , k }' ) &= X_t' .
\end{align*}
Thus, condition (\ref{defnSystemOfFiniteReturnTimeMaps(e)}) is satisfied. Let $ x \in X $. There are precisely one $ t \in \{ 1 , \ldots , T \} $, one $ k \in \{ 1 , \ldots , K_t \} $, and one $ j \in \{ 0 , \ldots , J_{ t , k } - 1 \} $ such that $ x \in h^j ( Y_{ t , k } ) $. Since $ \sP_{ t , k } $ is a partition of $ Y_{ t , k } $, there is precisely one $ m \in \{ 1 , \ldots , M_{ t , k } \} $ such that $ h^{ -j } ( x ) \in Y_{ t , k } ( m ) $. Set $ k' = \sum_{ l = 1 }^{ k - 1 } M_{ l , k } + m $, so that $ Y_{ t , k } ( m ) = Y_{ t , k }' $. Then since $ J_{ t , k' }' = J_{ t , k } $, we have precisely one $ t \in \{ 1 , \ldots , T' \}$, one $k' \in \{ 1 , \ldots , K_t' \}$, and one $ j \in \{ 0 , \ldots , J_{ t , k' }' - 1 \} $ such that $ x \in h^j ( Y_{ t , k' }' ) $. Thus, condition (\ref{defnSystemOfFiniteReturnTimeMaps(f)}) is satisfied. 

We now verify the conclusions of the proposition. Clearly (\ref{propReturnTimeMapsGiveFinerPartitions(a)}) is satisfied. For (\ref{propReturnTimeMapsGiveFinerPartitions(b)}), let $ t \in \{ 1 , \ldots T' \} $, let $ k' \in \{ 1 , \ldots , K_t' \} $, and let $ j \in \{ 0 , \ldots , J_{ t , k }' - 1 \} $. By definition, there is some $ k \in \{ 1 , \ldots , K_t \} $ and some $ m \in \{ 1 , \ldots , M_{ t , k } \} $ such that $ Y_{ t , k' }' = Y_{ t , k } ( m ) $. Since $ \sP' $ is a partition of $X$, it is also clear that 
\[
\sP_j' = \big\{ h^{ -j } ( U_r ) \setdiv \mbox{$r \in \{ 1 , \ldots , R \}$} \big\} 
\]
is a partition of $X$. Thus, there is some $ r \in \{ 1 , \ldots , R \}$ such that $ h^{ -j } ( U_r ) \cap Y_{ t , k } $ intersects $ Y_{ t , k } ( m ) $. By the definition of $ \sP_{ t , k } $, this means that $ Y_{ t , k } ( m ) \subset h^{ -j } ( U_r ) \cap Y_{ t , k } $. But then $ h^j ( Y_{ t , k' }' ) = h^j ( Y_{ t , k } ( m ) ) \subset U_r $. This proves that $ \sP_1 ( \sS' ) $ and $ \sP_2 ( \sS' ) $ are finer than $ \sP' $. This proves the proposition.
\end{proof}

The next two lemmas give us conditions for when a system $ \sS' $ has associated partitions (Definition \ref{defnSystemOfFiniteReturnTimeMaps}(\ref{defnSystemOfFiniteReturnTimeMaps(f)})) that are finer than those of another system $ \sS $. These are useful in proving other lemmas in Sections \ref{sectionProof1} and \ref{sectionProof2}.

\begin{lemma}\label{lemmaConditionForFinerPartitions}
Let $ ( X , h ) $ be a zero-dimensional system, let $ \sP $ be a partition of $ X $, and let $ \systembasic $ and $ \systemarg{ \prime } $ be systems of finite first return time maps subordinate to $ \sP $ such that $ \bigsqcup_{ t = 1 }^T X_t = \bigsqcup_{ t = 1 }^{ T' } X_t' $. Then $ \sP_1 ( \sS' ) $ is finer than $ \sP_1 ( \sS ) $ if and only if for each $ s \in \{ 1 , \ldots , T' \} $ and each $ l \in \{ 1 , \ldots , K_s' \} $, there is a $ t \in \{ 1 , \ldots , T \} $ and a $ k \in \{ 1 , \ldots , K_t \} $ such that $ Y_{ s , l }' \subset Y_{ t , k } $.
\end{lemma}

\begin{proof}
($ \Rightarrow $). Let $ s \in \{ 1 , \ldots , T' \} $ and let $ l \in \{ 1 , \ldots , K_s' \} $. Since $ \sP_1 ( \sS' ) $ is finer than $ \sP_1 ( \sS ) $, there is a $ t \in \{ 1 , \ldots , T \} $, a $ k \in \{ 1 , \ldots , K_t \} $, and a $ j \in \{ 0 , \ldots , J_{ t , k } - 1 \} $ such that $ Y_{ s , l }' \subset h^j ( Y_{ t , k } ) $. But by assumption, there is $ t' \in \{ 1 , \ldots , T \} $ such that $ Y_{ s , l }' \subset X_{ t' } $. Then since $ \sP_1 ( \sS )$ is a partition of $ X $, we must have $ t' = t $, and hence we must have $ j = 0 $.

($ \Leftarrow $). Let $ t \in \{ 1 , \ldots , T \} $, let $ k \in \{ 1 , \ldots , K_t \} $, and let $ j \in \{ 0 , \ldots , J_{ t , k } - 1 \} $. Let $ s \in \{ 1 , \ldots , T' \} $, $ l  \in \{ 1 , \ldots , K_{ s }' \} $, and $ i \in \{ 0 , \ldots , J_{ s  , l  } - 1 \} $ satisfy $ h^i ( Y_{ s , l }' ) \cap h^j ( Y_{ t , k } ) \neq \varnothing $. Now, note that, for all $ x \in h^i ( Y_{ s , l }' ) \cap h^j ( Y_{ t , k } ) $, we have $ h^{ - 1 } ( x ) , \ldots , h^{ - j + 1 } ( x ) \notin  \bigsqcup_{ t = 1 }^T X_t $ and $ h^{ -j } ( x ) \in  \bigsqcup_{ t = 1 }^T X_t $. Similarly, we have $ h^{ - 1 } ( x ) , \ldots , h^{ - i + 1 } ( x ) \notin  \bigsqcup_{ t = 1 }^T X_t $ and $ h^{ -i } ( x ) \in  \bigsqcup_{ t = 1 }^T X_t' $. Since $ \bigsqcup_{ t = 1 }^T X_t' =  \bigsqcup_{ t = 1 }^T X_t $, this means that $ i = j $, so $ Y_{ s , l }' \cap Y_{ t , k } \neq \varnothing $. But by assumption this means we must have $ Y_{ s , l }' \subset Y_{ t , k } $, and therefore $ h^i ( Y_{ s , l }' ) \subset h^j ( Y_{ t , k } ) $. This proves that $ \sP_1 ( \sS' ) $ is finer than $ \sP_1 ( \sS ) $.
\end{proof}

\begin{lemma}\label{lemmaPartition1FinerIffPartition2Finer}
Let $ ( X , h ) $ be a zero-dimensional system, let $ \sP $ be a partition of $ X $, and let $ \systembasic $ and $ \systemarg{ \prime } $ be systems of finite first return time maps subordinate to $ \sP $ such that $ \bigsqcup_{ t = 1 }^T X_t = \bigsqcup_{ t = 1 }^{ T' } X_t' $. Then $ \sP_1 ( \sS' ) $ is finer than $ \sP_1 ( \sS ) $ if and only if $ \sP_2 ( \sS' ) $ is finer than $ \sP_2 ( \sS ) $.
\end{lemma}

\begin{proof}
($ \Rightarrow $). Set $ \widehat{ X } = \bigsqcup_{ t = 1 }^T X_t $ and set $ \widehat{ X }' = \bigsqcup_{ t = 1 }^{ T' } X_t' $ (note that $ \widehat{ X } = \widehat{ X }' $, but the distinction will be important in our reasoning later). Let $ t \in \{ 1 , \ldots , T \} $ and let $ k \in \{ 1 , \ldots , K_t \} $. By assumption, there are $ s ( 1 ) , \ldots , s ( M ) \in \{ 1 , \ldots , T' \} $ and there are $ l ( m , 1 ) , \ldots , l ( m , N_m ) \in \{ 1 , \ldots , K_{ s ( m ) }' \} $ for each $ m \in \{ 1 , \ldots , M \} $ such that $ \bigsqcup_{ m = 1 }^N \bigsqcup_{ n = 1 }^{ N_m } Y_{ s ( m ) , l ( m , n ) }' = Y_{ t , k } $. Then clearly $ \bigsqcup_{ m = 1 }^N \bigsqcup_{ n = 1 }^{ N_m } h^{ J_{ t , k } } ( Y_{ s ( m ) , l ( m , n ) }' ) = h^{ J_{ t , k } } ( Y_{ t , k } ) $. Now, for each $ x \in Y_{ t , k } $, we have $ \lambda_{ \widehat{ X } } ( x ) = \lambda_{ X_t } ( x ) = J_{ t , k } $. But then for all $ m \in \{ 1 , \ldots , M \} $, we have $ \lambda_{ X_{ s ( m ) }' } ( x ) = \lambda_{ \widehat{ X }' } ( x ) = \lambda_{ \widehat{ X } } ( x ) = J_{ t , k } $. Thus, for each $ m \in \{ 1 , \ldots , M \} $ and each $ n \in \{ 1 , \ldots , N_m \} $, we have $ J_{ t , k } = J_{ s ( m ) , l ( m , n ) }' $, and so $ h^{ J_{ t , k } } ( Y_{ s ( m ) , l ( m , n ) } ) $ is an element of $ \sP_2 ( \sS' ) $. Thus, $ \sP_2 ( \sS' ) $ is finer than $ \sP_2 ( \sS ) $. The proof of ($ \Leftarrow $) is analogous to this.
\end{proof}

The following proposition is not used later in the paper, but it gives intuition about how systems of finite first return time maps can be refined in the absence of periodic points. 

\begin{prop}\label{propAperiodicGivesLargeReturnTime}
Let $(X,h)$ be a zero-dimensional system such that, for any partition $\sP$ of $X$, $(X,h)$ admits a system of finite first return time maps subordinate to $\sP$. Then $(X,h)$ has no periodic points if and only if for every partition $\sP$ and every $ N \in \bZ_{ > 0 } $, there is a system $\systembasic$ of finite first return time maps subordinate to $\sP$ such that $J_{t,k} \geq N$ for all $t \in \{1,\ldots,T\}$ and all $k \in \{1,\ldots,K_t\}$.
\end{prop}

\begin{proof}
($\Rightarrow$). Let $ N \in \bZ_{ > 0 } $ and let $\sP$ be a partition of $X$. Since all points are aperiodic, for each $x \in X$, there is a compact open neighborhood $U_x$ such that $U_x, h(U_x),\ldots, h^{N-1}(U_x)$ are pairwise disjoint. Then $(U_x)_{x \in X}$ is a compact open cover of $X$, and hence has a finite compact open refinement. By taking appropriate intersections, this refinement can be taken to be a partition $\sP'$ of $X$. Let $\sP''$ be a partition of $X$ finer than both $\sP$ and $\sP'$ and let $ \systembasic $ be a system of finite first return time maps subordinate to $\sP''$. Since $\sP''$ is finer than $\sP$, this $ \sS $ is also subordinate to $\sP$. Since the first $ N - 1 $ iterates of any element of this partition are pairwise disjoint, we must have $J_{t,k} \geq N$ for all $t \in \{1,\ldots,T\}$ and all $k \in \{1,\ldots,K_t\}$.	

($\Leftarrow$). Suppose that $x \in X$ is a periodic point of $(X,h)$, let $ M \in \bZ_{ > 0 } $ satisfy $h^M(x)=x$, let $N \in \bZ_{ > 0 } $ be larger than $M$, and let $\sP$ be a partition of $X$.  Let $\systembasic$ be a system of finite first return time maps subordinate to $\sP$ and let $\sP_1(\sS)$ be as in Definition \ref{defnSystemOfFiniteReturnTimeMaps}. Since $\sP_1(\sS)$ is a partition of $X$, there are $t \in \{1,\ldots,T\}$, $k \in \{1,\ldots,K_t\}$, and $j \in \{0,\ldots,J_{t,k}-1\}$ such that $x \in h^j(Y_{t,k})$. It is clear that $h^{-j}(x) \in Y_{t,k}$ and  $h^{-j}(x) \in h^M(Y_{t,k})$, and so $J_{t,k} \leq M < N$.
\end{proof}

We now introduce the second main definition of this paper, whose connection to Definition \ref{defnSystemOfFiniteReturnTimeMaps} is Theorem \ref{thmFiberwiseIffAdmitsPartitions}.

\begin{defn}\label{defnFiberwiseEssentiallyMinimal}
Let $ ( X , h ) $ be a zero-dimensional system. We say that $ ( X , h ) $ is \emph{fiberwise essentially minimal} if there is a closed subset $ Z \subset X $ and a quotient map $ \psi: X \to Z $ such that
\begin{enumerate}[(a)]
\item \label{defnFiberwiseEssentiallyMinimal(a)} $ \psi|_Z : Z \to Z $ is the identity map.
\item \label{defnFiberwiseEssentiallyMinimal(b)} $ \psi \circ h = \psi $.
\item \label{defnFiberwiseEssentiallyMinimal(c)} For each $ z \in Z $, $ ( \psi^{ -1 } ( z ) , h|_{ \psi^{ -1 } ( z ) } ) $ is an essentially minimal system and $ z $ is in its minimal set.
\end{enumerate}
\end{defn}

\begin{exmps}\label{exmpsFiberwiseEssentiallyMinimal}
We provide some examples of fiberwise essentially minimal zero-dimensional systems.
\begin{enumerate}[(a)]
\item \label{exmpsFiberwiseEssentiallyMinimal(a)} For an zero-dimensional system $ ( X , h ) $, essentially minimality implies fiberwise essentially minimality. We can take $ Z $ to be $ \{ z \} $ for any $ z $ in the minimal set of $ X $ and then we can take $ \psi : X \to Z $ to be the (only possible) map $ x \mapsto z $.
\item \label{exmpsFiberwiseEssentiallyMinimal(b)}Let $ Z $ be a compact metrizable totally disconnected space and let $ ( Y , h' ) $ be an essentially minimal zero-dimensional system. Let $ X = Y \times Z $ and let $ h = h' \times \id $. Then $ ( X , h ) $ is an essentially minimal system, where we take $ \psi : X \to Z $ to be the map $ ( y , z ) \mapsto z $.
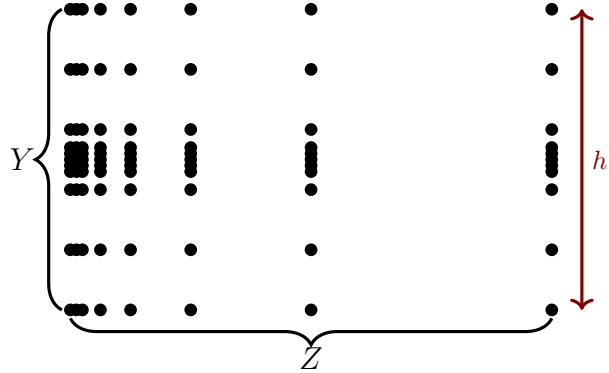
\begin{center}
\begin{tikzpicture}[scale=0.8]
\foreach \i in { 0 , 0.1 , 0.2 , 0.5 , 1 , 2 , 4 , 8 }
\foreach \j in { 0 , 1 , 2 , 2.3 , 2.4 , 2.5 , 2.6 , 2.7 , 3 , 4 , 5 }
{
\fill ( \i , \j ) circle(3pt);
}
\draw [very thick, decorate,decoration={brace,amplitude=10pt},xshift=-4pt,yshift=0pt]
(0,0) -- (0,5);
\node at ( -0.8 , 2.5 ) {\Large $ Y $};
\draw [very thick, decorate,decoration={brace,amplitude=10pt,mirror},xshift=0pt,yshift=-4pt]
(0,0) -- (8,0);
\node at ( 4 , -0.8 ) {\Large $ Z $};
\draw [<-> , very thick, color=red!50!black] ( 8.5 , 0 ) -- node[right] { $ h $ }( 8.5 , 5 );
\end{tikzpicture}
\end{center}

\item \label{exmpsFiberwiseEssentiallyMinimal(c)} Let $ Z = \bZ \cup \{ \infty \} $, let $ ( Y , h' ) $ be an essentially minimal zero-dimensional system. Let $ X = ( Y \times Z ) / ( Y \times  \{ \infty \} ) $ and let $ \pi : Y \times Z \to X $ be the quotient map. Let $ \widetilde{ h } = \pi ( h' \times \id ) : Y \times Z \to X $ and let $ h : X \to X $ be the continuous map satisfying $ h \circ \pi = \widetilde{ h } $, which is obtained from the universal property of the quotient map. One checks that $ h $ is a homeomorphism. Define $ \widetilde{ \psi } : Y \times Z \to Z $ by $ \widetilde{ \psi } ( ( y , z ) ) = z $ and then let $ \psi : X \to Z $ be the continuous map satisfying $ \psi \circ \pi = \widetilde{ \psi } $, which is obtained from the universal property of the quotient map. One checks that $ \psi $ itself is a quotient map, and then one checks that $ ( X , h ) $ is a fiberwise essentially minimal zero-dimensional system (using $ Z $ and $ \psi $ as defined above).
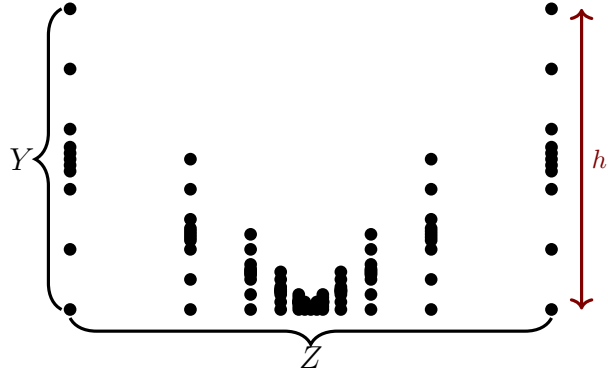
\begin{center}
\begin{tikzpicture}[scale=0.8]
\foreach \i in { 0.1 , 0.2 , 0.5 , 1 , 2 , 4 }
\foreach \j in { 0 , 1 , 2 , 2.3 , 2.4 , 2.5 , 2.6 , 2.7 , 3 , 4 , 5 }
{
\fill ( \i , \i*\j/4 ) circle(3pt);
}
\foreach \i in { -4 , -2 , -1 , -0.5 , -0.2 , -0.1 }
\foreach \j in { 0 , 1 , 2 , 2.3 , 2.4 , 2.5 , 2.6 , 2.7 , 3 , 4 , 5 }
{
\fill ( \i , \i*\j/-4 ) circle(3pt);
}
\fill ( 0 , 0 ) circle(3pt);
\draw [very thick, decorate,decoration={brace,amplitude=10pt},xshift=-4pt,yshift=0pt]
(-4,0) -- (-4,5);
\node at ( -4.8 , 2.5 ) {\Large $ Y $};
\draw [very thick, decorate,decoration={brace,amplitude=10pt,mirror},xshift=0pt,yshift=-4pt]
(-4,0) -- (4,0);
\node at ( 0 , -0.8 ) {\Large $ Z $};
\draw [<-> , very thick, color=red!50!black] ( 4.5 , 0 ) -- node[right] { $ h $ }( 4.5 , 5 );
\end{tikzpicture}
\end{center}

\end{enumerate}
\end{exmps}

\begin{rmk}
Let $ ( X , h ) $ be a fiberwise essentially minimal zero-dimensional system and let $ \psi $ and $ Z $ be as in Definition \ref{defnFiberwiseEssentiallyMinimal}. The term ``fiberwise" comes from the fact that such a system looks like a fiber bundle over $ Z $ where the homeomorphism acts fiberwise, except for the fact that the fibers need not be pairwise homeomorphic, and the bundle need not be locally trivial. In the context of $ C^* $-algebras, this terminology is somewhat standard, as the definition of $ C ( Z ) $-algebras do not require local triviality or pairwise fiber isomorphism. It turns out that $ C^* ( \bZ , X , h ) $ is indeed a $ C ( Z ) $-algebra (see Proposition \ref{propFiberwiseGivesContinuousField}). 
\end{rmk}

Given a $ C^* $-algebra $ A $, we denote its multiplier algebra by $ M ( A ) $ and we denote its center by $ Z ( A ) $. The following definition is standard (for example, it is Definition 3.2 of \cite{GaHiSa17}). 

\begin{defn}\label{defnC(Y)Algebras}
Let $ Y $ be a compact Hausdorff space. We say that a pair $ ( A , \theta ) $ is a \emph{$ C ( Y ) $-algebra} if $ A $ is a $ C^* $-algebra and $  \theta: C ( Y ) \to Z ( M ( A ) ) $ is a $ * $-homomorphism such that $ \theta ( C ( Y ) ) A = A $. We often suppress $ \theta $ from the notation, and just say that $ A $ is a $ C ( Y ) $-algebra.

Given $ y \in Y $, we define the \emph{fiber over $ y $} to be the quotient $ A_y = A / \theta( C_0 ( Y  \setminus \{ y \} ) ) A $. Given $ a \in A $, we denote the image of $ a $ in the quotient $ A_y $ by $ a_y $. 

We additionally say that $ A $ is a \emph{continuous $ C ( Y ) $-algebra} if for every $ a \in A $, the map $ N_a : Y \to \bR $ given by $ N_a ( y ) = \| a_y \| $ is continuous.
\end{defn}

\begin{lemma}\label{lemmaClosedMapLemma}
Let $ X , Y $ be topological spaces, let $ f : X \to Y $ be a closed map, let $ S $ be a subset of $ Y $, and let $ U $ be an open set of $ X $ containing $ f^{ -1 } ( S ) $. Then there is an open subset $ V $ of $ Y $ containing $ S $ such that $ f^{ - 1 } ( V ) \subset U $. 
\end{lemma}

\begin{proof}
Since $ X \setminus U $ is closed and since $ f $ is a closed map, $ f ( X \setminus U ) $ is a closed set. Since $ X \setminus U $ is disjoint from $ f^{ -1 } ( S ) $, $ f ( X \setminus U ) $ is disjoint from $ S $. Set $ V = Y \setminus f ( X \setminus U ) $. Then $ V $ is an open subset of $ Y $ that contains $ S $. Let $ x \in f^{ -1 } ( V ) $. If $ x \notin U $, then $ f ( x ) \in f ( X \setminus U ) $, a contradiction. Thus, $ f^{ -1 } ( V ) \subset U $.
\end{proof}

\begin{prop}\label{propFiberwiseGivesContinuousField}
Let $ ( X , h ) $ be a fiberwise essentially minimal zero-dimensional system and let $ Z $ be as in Definition \ref{defnFiberwiseEssentiallyMinimal}. Then $ C^* ( \bZ , X , h ) $ is a continuous $ C ( Z ) $-algebra. Moreover, for each $ z \in Z $, we have $ C^* ( \bZ , X , h )_z \cong C^* ( \bZ , \psi^{ -1 } ( z ) , h|_{ \psi^{ -1 } ( z ) } ) $.
\end{prop}

\begin{proof}
Let $ \iota : Z \to X $ be the inclusion of $ Z $ in $ X $. Let $ \alpha $ be the automorphism of $ C ( X ) $ induced by $ h $, so that $ \alpha ( f ) ( x ) = f ( h^{ - 1 } ( x ) ) $ for all $ f \in  C ( X ) $. For convenience, denote $ C^* ( \bZ , X , h ) $ by $ A $. Since $ \psi \circ h = \psi $, note that we have $ \iota = \alpha \circ \iota $. We have a natural embedding $ C ( X ) \hookrightarrow A $, which, when composed with $ \iota $, gives us an embedding $ \theta : C ( Z ) \to A $. To see that the image of $ \theta $ is central, it suffices to show that it commutes with the standard unitary $ u $ in $ A $. So let $ f \in C ( Z ) $, and see that $ \theta ( f ) u = u \alpha( \theta ( f ) ) = u \theta ( f ) $ by the $ \alpha $-invariance of $ \iota $. Thus, $ A $ is a $ C ( Z ) $-algebra.

Now, let $ a \in A $, let $ \eps > 0 $, and let $ z \in Z $. There is an $ N \in \bZ_{ \geq 0 } $, a partition $ \sP $ of $ X $, and an element $ b = \sum_{ k = -N }^N f_k u^k $ where $ f_k \in C ( \sP ) $ for all $ k \in \{ - N , \ldots , N \} $ such that $ \| a - b \| < \eps/2 $. Then for all $ z' \in Z $, we have
\[
\big| \| a ( z' ) \| - \| b ( z' ) \| \big| \leq \| a ( z' ) - b ( z' ) \| \leq \| a - b \| < \eps/2 .
\]
Let $ U_1 , \ldots U_L $ be the elements of $ \sP $ that have nonempty intersection with $ \psi^{ -1 } ( z ) $, and define $ U = \bigsqcup_{ l = 1 }^L U_l $. Since $ \psi $ is a continuous map from a compact space to a Hausdorff space, $ \psi $ is a closed map, and so we can apply Lemma \ref{lemmaClosedMapLemma} with $ Z $ in place of $ Y $, $ \psi $ in place of $ f $, and $ \{ z \} $ in place of $ S $ to get an open set $ V $ of $ Z $ containing $ z $ such that $ \psi^{ -1 } ( V ) \subset U $. Then notice that since $ f_k $ is constant on $ \psi^{ -1 } ( V ) $ for all $ k $, we have $ \| b ( z ) \| = \| b ( z' ) \| $ for all $ z' \in V $. Thus, for all $ z' \in V $, we have
\[
\big| \| a ( z' ) \| - \| b ( z ) \| \big| =  \big| \| a ( z' ) \| - \| b ( z' ) \| \big| < \eps/2 ,
\]
and so 
\[
\big| \| a ( z ) \| - \| a ( z' ) \| \big| < \eps.
\]
Thus, $ A $ is a continuous $ C ( Z ) $-algebra.

Finally, let $ z \in Z $. By Definition \ref{defnFiberwiseEssentiallyMinimal}, we see $ ( \psi^{ -1 } ( z ) , h|_{ \psi^{ -1 } ( z ) } ) $ is a zero-dimensional system, and so $ C^* ( \bZ , \psi^{ -1 } ( z ) , h|_{ \psi^{ -1 } ( z ) } ) $ makes sense. Define a map $ \vphi : C^* ( \bZ , X , h )_z \to C^* ( \bZ , \psi^{ -1 } ( z ) , h|_{ \psi^{ -1 } ( z ) } ) $ by sending elements of the form $ \left( \sum_{ k = -N }^N f_k u^k \right)_z $ to $ \sum_{ k = -N }^N f_k|_{ \psi^{ -1 } ( z ) } u^k $. Clearly $ \vphi $ respects pointwise operations, is bounded, and is injective. By the Tietze extension theorem, $ \vphi $ is surjective. We now show that $ \vphi $ is well defined. Let $ f , g \in C ( X ) $ and let $ x \in \psi^{ -1 } ( z ) $ and suppose that $ f ( x ) \neq g ( x ) $. Then there is no $ h \in C_0 ( X \setminus \psi^{ -1 } ( z ) ) $ such that $ f + h = g $. This proves that $ \vphi $ is well-defined. This proves that $ \vphi $ is an isomorphism of $ C^* $-algebras and proves the lemma.
\end{proof}

\section{Theorems}\label{sectionTheorems}

We now introduce the main theorems of the paper. Theorem \ref{thmFiberwiseIffAdmitsPartitions} and Theorem \ref{thmMainTheorem} will take significant work to prove, and their proofs will be located in Section \ref{sectionProof1} and Section \ref{sectionProof2}, respectively. Examples illustrating the proof of Theorem \ref{thmMainTheorem} are provided at the end of Section \ref{sectionProof2}.

\begin{thm}\label{thmFiberwiseIffAdmitsPartitions}
Let $ ( X , h ) $ be a zero-dimensional system. Then $ ( X , h ) $ is fiberwise essentially minimal if and only if for any partition $ \sP $ of $ X $, $ ( X , h ) $ admits a system of finite first return time maps subordinate to $ \sP $.
\end{thm}

The following theorem is a generalization of Theorem 2.1 of \cite{Putnam90}, which states that if $ ( X , h ) $ is a minimal zero-dimensional system, then $ C^* ( \bZ , X , h ) $ is an A$ \bT $-algebra.

\begin{thm}\label{thmMainTheorem}
Let $ ( X , h ) $ be a fiberwise essentially minimal zero-dimensional system. Then \linebreak $ C^*( \bZ , X , h ) $ is an A$ \bT $-algebra.
\end{thm}

\begin{thm}\label{thmRR0}
Let $ ( X , h ) $ be a fiberwise essentially minimal zero-dimensional system with no periodic points. Then $ C^* ( \bZ , X , h ) $ has real rank zero.
\end{thm}

\begin{proof}
 Let $ Z $ be as in Definition \ref{defnFiberwiseEssentiallyMinimal}. By Proposition \ref{propFiberwiseGivesContinuousField}, $ C ( \bZ , X , h ) $ is a continuous $ C ( Z ) $-algebra such that for each $ z \in Z $, the fiber $ C^* ( \bZ , X , h )_z $ is isomorphic to $ C^* ( \bZ , \psi^{ -1 } ( z ) , h_{ \psi^{ -1 } ( z ) } ) $. By Definition \ref{defnFiberwiseEssentiallyMinimal}, $ ( \psi^{ -1 } ( z ) , h|_{ \psi^{ - 1 } ( z ) } ) $ is an essentially minimal zero-dimensional system, and since $ ( X , h ) $ has no periodic points, neither does $ ( \psi^{ -1 } ( z ) , h|_{ \psi^{ - 1 } ( z ) } ) $. Thus, by the comments following Theorem 8.3 of \cite{HermanPutnamSkau92}, the fibers of $ C^* ( \bZ ,  X , h ) $ all have real rank zero. By Theorem 2.1 of \cite{Pasnicu05}, $ C^* ( \bZ , X , h ) $ has real rank zero.
\end{proof}

Thus, crossed products $ C^* $-algebras associated to fiberwise essentially minimal zero-dimensional $ C^* $-algebras with no periodic points are classifiable by \cite{DadarlatGong97}.

\section{Proof of Theorem \ref{thmFiberwiseIffAdmitsPartitions}}\label{sectionProof1}

The following lemma tells us that we can ``refine" a system $ \sS $ (as in Proposition \ref{propReturnTimeMapsGiveFinerPartitions}) to a system $ \sS' $ such that we can create a new system where each base in the system is equal to $ Y_{ t , k }' $ for some $ t $ and $ k $. This is lemma is used in the proof of Lemma \ref{lemmaNewBasesUnderOldBases}, which tells us that we can choose systems whose bases are ``under" a given system; this fact is crucial to Construction \ref{constSequenceOfPartitions}.
 
\begin{lemma}\label{lemmaReturnTimeMapsNewBasesFromOldPartition}
Let $ ( X , h ) $ be a zero-dimensional system, let $ \sP $ be a partition of $ X $, and let $ \systembasic $ be a system of finite first return time maps subordinate to $ \sP $. Let $ \sP_1 ( \sS ) $ and $ \sP_2 ( \sS ) $ be as in Definition \ref{defnSystemOfFiniteReturnTimeMaps} and let $ \systemarg{ ( 1 ) } $ be a system of finite first return time maps subordinate to $ \sP_1 ( \sS ) $. Then there is a system 
\[
\systemarg{ \prime } 
\]
of finite first return time maps subordinate to $\sP$ and a system 
\[
\systemarg{ ( 1 ) \prime } 
\]
of finite first return time maps subordinate to $ \sP_1 ( \sS ) $ such that:
\begin{enumerate}[(a)]
\item \label{lemmaReturnTimeMapsNewBasesFromOldPartition(a)} We have $ T' = T $, and for all $ t \in \{ 1 , \ldots , T' \} $, we have $ X_t = X_t' $.
\item \label{lemmaReturnTimeMapsNewBasesFromOldPartition(b)} The partition $ \sP_1 ( \sS' ) $ is finer than $ \sP_1 ( \sS ) $ and the partition $ \sP_2 ( \sS' ) $ is finer than $ \sP_2 ( \sS ) $.
\item \label{lemmaReturnTimeMapsNewBasesFromOldPartition(c)} For each $ s \in \{ 1 , \ldots , T^{ ( 1 ) \prime } \} $, there is a $ t_s \in \{ 1 , \ldots , T \} $ and a $ k_s \in \{ 1 , \ldots , K_{ t_s } \} $ such that $ X_t^{ ( 1 ) \prime } = Y_{ t_s , k_s }' $.
\end{enumerate}
\end{lemma}

\begin{proof}
Since $ \sS^{ ( 1 ) } $ is subordinate to $ \sP_1 ( \sS )$, for each $ s \in \{ 1 , \ldots , T^{ ( 1 ) } \} $, there is some $ t_s \in \{ 1 , \ldots , T \} $, some $ k_s \in \{ 1 , \ldots , K_{ t_s } \} $, and some $ j_s \in \{ 0 , \ldots , J_{ t_s , k_s } - 1 \} $ such that $ X_s^{ ( 1 ) } \subset h^{ j_s } ( Y_{ t_s , k_s } ) $. Set $ T^{ ( 1 ) \prime } = T^{ ( 1 ) } $, set $ X_s^{ ( 1 ) \prime } = h^{ -j_s } ( X_s^{ ( 1 ) } ) $ and set $ K_s^{ ( 1 ) \prime } = K_s^{ ( 1 ) } $ for all $ s \in \{ 1 , \ldots , T^{ ( 1 ) \prime } \} $, and set $ Y_{ s , k }^{ ( 1 ) \prime } = h^{ -j_s } ( Y_{ s , k }^{ ( 1 ) } ) $ and set $ J_{ s , k }^{ ( 1 ) \prime } = J_{ s , k }^{ ( 1 ) } $ for all $ s \in \{ 1 , \ldots , T^{ ( 1 ) \prime } \} $ and $ k \in \{ 1 , \ldots , K^{ ( 1 ) \prime }_t \} $. 

We now check that $ \systemarg{(1)\prime} $ is a system of finite first return time maps subordinate to $ \sP_1 ( \sS ) $ by checking each of the conditions of Definition \ref{defnSystemOfFiniteReturnTimeMaps}. Conditions (\ref{defnSystemOfFiniteReturnTimeMaps(a)}) and (\ref{defnSystemOfFiniteReturnTimeMaps(c)}) are clearly met. By construction, for each $ s \in \{ 1 , \ldots , T^{ ( 1 ) \prime } \}$, we have $ X_s^{ ( 1 ) \prime } \subset Y_{ t_s , k_s } \in \sP_1 ( \sS ) $. Thus, condition (\ref{defnSystemOfFiniteReturnTimeMaps(b)}) is met. Since for all $ s \in \{ 1 , \ldots , T^{ ( 1 ) } \} $ we have $\bigsqcup_{ k = 1 }^{ K_s } Y_{ s , k }^{ ( 1 ) } = X_s^{ ( 1 ) } $, for all $ s \in \{ 1 , \ldots , T^{ ( 1 ) \prime } \} $ we have 
\begin{align*}
\bigsqcup_{ k = 1 }^{ K_s^{ ( 1 ) \prime } } Y_{ s , k }^{ ( 1 ) \prime } &= \bigsqcup_{ k = 1 }^{ K_s^{ ( 1 ) } } h^{ - j_s } ( Y_{ t , k }^{ ( 1 ) } ) \\
&= h^{ - j_s } ( X_s^{ ( 1 ) } ) \\
&= X_s^{ ( 1 ) \prime } .
\end{align*}
Thus, condition (\ref{defnSystemOfFiniteReturnTimeMaps(d)}) is satisfied. For condition (\ref{defnSystemOfFiniteReturnTimeMaps(e)}), clearly $ J_{ s , k } \in \bZ_{ > 0 }$ for all $ s \in \{ 1 , \ldots , T^{ ( 1 ) \prime } \} $ and all $ k \in \{ 1 , \ldots , K_s^{ ( 1 ) \prime } \} $. For each $ s \in \{ 1 , \ldots , T^{ ( 1 ) \prime } \} $, we also have 
\begin{align*}
\bigsqcup_{ k = 1 }^{ K_s^{ ( 1 ) \prime } } h^{ J_{ s , k }^{ ( 1 ) \prime } } ( Y_{ s , k }^{ ( 1 ) \prime } ) &= \bigsqcup_{ k = 1 }^{ K_s^{ ( 1 ) } } h^{ J_{ s , k }^{ ( 1 ) } - j_s } ( Y_{ s , k }^{ ( 1 ) } ) \\
&= h^{ -j_s } \left( \bigsqcup_{ k = 1 }^{ K_s^{ ( 1 ) } } h^{ J_{ s , k }^{ ( 1 ) \prime } } ( Y_{ s , k }^{ ( 1 ) } ) \right) \\
&= h^{ -j_s } ( X_s^{ ( 1 ) } ) \\
&= X_s^{ ( 1 ) \prime }.
\end{align*}
Thus, condition (\ref{defnSystemOfFiniteReturnTimeMaps(e)}) holds. Finally, let $ x \in X $. Since $ \sS^{ ( 1 ) } $ is a system of finite first return time maps subordinate to $ \sP_1 ( \sS ) $, there is precisely one $ s \in \{ 1 , \ldots , T^{ ( 1 ) } \} $, one $ k \in \{ 1 , \ldots , K_s^{ ( 1 ) } \} $, and one $ j \in  \{ 0 , \ldots , J_{ s , k }^{ ( 1 ) } - 1 \} $ such that $ x \in h^{ j - j_s } ( Y_{ s , k }^{ ( 1 ) } ) = h^j ( Y_{ s , k }^{ ( 1 ) \prime } ) $. This is all that was needed to show that $ \sP_1 ( \sS^{ ( 1 ) \prime } )$ is a partition of $ X $. Thus, condition (\ref{defnSystemOfFiniteReturnTimeMaps(f)}) is satisfied, proving that $ \sS^{ ( 1 ) \prime } $ is a system of finite first return time maps subordinate to $ \sP_1 ( \sS ) $.

Now, for each $ t \in \{ 1 , \ldots , T \} $, let $ A_t = \{ a ( t , 1 ) , \ldots , a ( t , M_t ) \} $ denote the set of all $ s \in \{ 1 , \ldots , T^{ ( 1 ) \prime } \} $ such that $ t_s = t $ and $ X_s^{ ( 1 ) \prime } \neq Y_{ t_s , k_s } $, and let $ B_t = \{ b ( t , 1 ) , \ldots , b ( t , N_t ) \} $ be the set of all $ k \in \{ 1 , \ldots , K_t \} $ such that $ \left( \bigsqcup_{ s = 1 }^{ T^{ ( 1 ) \prime } } X_s^{ ( 1 ) \prime } \right) \cap Y_{ t , k } = \varnothing $ or such that $ X_s^{ ( 1 ) \prime } = Y_{ t_s , k_s } $. Set $ T' = T $, $ X_t' = X_t $ and $ K_t' = K_t + M_t $ for all $ t \in \{ 1 , \ldots , T' \} $, and 
\[
Y_{ t , k }' = 
\begin{cases} 
X_s^{ ( 1 ) \prime } & \text{if } s \in A_t \text{ and } k = k_s , \\ 
Y_{ t , k_s } \setminus X_s^{ ( 1 ) \prime } & \text{if } k = K_t + m \text{ for some } m \in \{ 1 , \ldots , M_t \} , \text{ and } s = a ( t , m ) , \\ 
Y_{ t , k } & \text{otherwise} 
\end{cases}
\]
and 
\[
J_{ t , k }' = 
\begin{cases} 
J_{ t , k_s } & \text{if } k = K_t + m \text{ for some } m \in \{ 1 , \ldots , M_t \} , \text{ and } s = a ( t , m ) , \\ 
J_{ t , k } & \text{otherwise} 
\end{cases}
\]
for all $ t \in \{ 1 , \ldots , T' \} $ and $ k \in \{ 1 , \ldots , K_t' \} $.
We now check that $\systemarg{\prime}$ is a system of finite first return time maps subordinate to $\sP$ by checking the conditions of Definition \ref{defnSystemOfFiniteReturnTimeMaps}. Conditions (\ref{defnSystemOfFiniteReturnTimeMaps(a)}), (\ref{defnSystemOfFiniteReturnTimeMaps(b)}), and (\ref{defnSystemOfFiniteReturnTimeMaps(c)}) are clearly met. For each $ t \in \{ 1 , \ldots , T' \} $, we have the following, where we shorten $ k_{ a ( t , m ) } $ to $ k ( t , m ) $:
\begin{align*}
\bigsqcup_{ k = 1 }^{ K_t' } Y_{ t , k }' &= \left( \bigsqcup_{ m = 1 }^{ M_t } X_{ a ( t , m ) }^{ ( 1 ) \prime } \right ) \sqcup \left( \bigsqcup_{ m = 1 }^{ M_t } Y_{ t , k ( t , m ) } \setminus X_{ a ( t , m ) }^{ ( 1 ) \prime } \right)  \sqcup \left( \bigsqcup_{ n = 1 }^{ N_t } Y_{ t , b ( t , n ) } \right) \\
&= \left( \bigsqcup_{ m = 1 }^{ M_t } Y_{ t , k ( t , m ) } \right) \sqcup\left( \bigsqcup_{ n = 1 }^{ N_t } Y_{ t , b ( t , n ) } \right) \\
&= X_t \\
&= X_t ' .
\end{align*}
Thus, condition (\ref{defnSystemOfFiniteReturnTimeMaps(d)}) is met. Similarly, for each $ t \in \{ 1 , \ldots , T' \} $, we have the following, where we again shorten $ k_{ a ( t , m ) } $ to $ k ( t , m ) $:
\begin{align*}
\bigsqcup_{ k = 1 }^{ K_t' } h^{ J_{ t , k }' } ( Y_{ t , k }' ) &= \left( \bigsqcup_{ m = 1 }^{ M_t } h^{ J_{ t , k ( t , m ) }  } ( X_{ a ( t , m ) }^{ ( 1 ) \prime } ) \right ) \sqcup \left( \bigsqcup_{ m = 1 }^{ M_t } h^{ J_{ t , k ( t , m ) } } ( Y_{ t , k ( t , m ) } \setminus X_{ a ( t , m ) }^{ ( 1 ) \prime } ) \right) \\
& \hspace{0.5cm} \sqcup \left( \bigsqcup_{ n = 1 }^{ N_t } h^{ J_{ t , b ( t , n ) } } ( Y_{ t , b ( t , n ) } ) \right) \\
&= \left( \bigsqcup_{ m = 1 }^{ M_t } h^{ J_{ t , k ( t , m ) } } ( Y_{ t , k ( t , m ) } ) \right) \sqcup\left( \bigsqcup_{ n = 1 }^{ N_t } h^{ J_{ t , b ( t , n ) } } ( Y_{ t , b ( t , n ) } ) \right) \\
&= X_t \\
&= X_t' .
\end{align*}
Thus, condition (\ref{defnSystemOfFiniteReturnTimeMaps(e)}) is met. Finally, for each $ x \in X $, there are precisely one $ t \in \{ 1 , \ldots , T \} $, one $ k \in \{ 1 , \ldots , K_t \}$, and one $ j \in \{ 0 , \ldots , J_{ t , k } - 1 \} $ such that $ x \in h^j ( Y_{ t , k } )$. If $ Y_{ t , k } = Y_{ t , k }' $, then $ J_{ t , k } = J_{ t , k }' $, and so $ x \in h^j ( Y_{ t , k }' ) $ for precisely one $ t \in \{ 1 , \ldots , T' \} $, one $ k \in \{ 1 , \ldots , K_t' \}$ and one $ j \in \{ 0 , \ldots , J_{ t , k }' - 1 \} $. Otherwise, there is some $ s \in A_t $ such that $ X_s^{ ( 1 ) \prime } \subset Y_{ t , k }$. There are now two possible cases. First, if $ x \in h^j ( X_s^{ ( 1 ) \prime } ) $, then $ J_{ t , k } = J_{ t , k }' $, and so $ x \in h^j ( Y_{ t , k }' ) $ for precisely one $ t \in \{ 1 , \ldots , T' \} $, one $ k \in \{ 1 , \ldots , K_t' \}$ and one $ j \in \{ 0 , \ldots , J_{ t , k }' - 1 \} $. Otherwise, if $ x \in h^j ( Y_{ t , k } \setminus X_s^{ ( 1 ) \prime } ) $, then $ x \in h^j ( Y_{ t , K_t + m }' ) $ where $ m $ is such that $ s = a ( t , m ) $. In this case, we also have $ J_{ t , k } = J_{ t , K_t + m }' $, and so $ x \in h^j ( Y_{ t , k' }' ) $ for precisely one $ t \in \{ 1 , \ldots , T' \} $, one $ k' \in \{ 1 , \ldots , K_t' \}$ and one $ j \in \{ 0 , \ldots , J_{ t , k' }' - 1 \} $ (specifically, $ k'  = K_t + m $) . Thus, condition (\ref{defnSystemOfFiniteReturnTimeMaps(f)}) holds, and so $ \sS' $ is indeed a system of finite first return time maps subordinate to $ \sP $.

We now check that the conclusions of the lemma are satisfied. Clearly conclusion (\ref{lemmaReturnTimeMapsNewBasesFromOldPartition(a)}) is satisfied. Conclusion (\ref{lemmaReturnTimeMapsNewBasesFromOldPartition(b)}) is satisfied by Lemma \ref{lemmaConditionForFinerPartitions} and Lemma \ref{lemmaPartition1FinerIffPartition2Finer}. Finally, conclusion (\ref{lemmaReturnTimeMapsNewBasesFromOldPartition(c)}) is also clearly met by the way we defined the elements of $ \sS' $.
\end{proof}

\begin{lemma}\label{lemmaNewBasesUnderOldBases}
Let $(X,h)$ be a zero-dimensional system such that, for any partition $\sR$ of $X$, $(X,h)$ admits a system of finite first return time maps subordinate to $\sR$. Let $\sP$ and $\sP'$ be partitions of $X$, and let $ \systembasic $ be a system of finite first return time maps subordinate to $\sP$. Then there is a system $ \systemarg{\prime} $ of finite first return time maps subordinate to $\sP'$ such that, for each $t' \in \{1,\ldots,T'\}$, there is a $t \in \{1,\ldots,T\}$ such that $X_{t'}' \subset X_t$.
\end{lemma}

\begin{proof}
By Proposition \ref{propReturnTimeMapsGiveFinerPartitions}, there is a system 
\[
\systemarg{(0)}
\]
of finite first return time maps subordinate to $ \sP $ such that $ T^{ ( 0 ) } = T $, $ \sP_1( \sS^{ ( 0 ) } ) $ is finer than $ \sP' $, and for all $ t \in \{ 1 , \ldots , T \} $, $ X_t^{ ( 0 ) } = X_t $.

Now, by Lemma \ref{lemmaReturnTimeMapsNewBasesFromOldPartition}, there is some system
\[
\systemarg{\prime}
\]
of finite first return time maps subordinate to $ \sP_1( \sS^{ ( 0 ) } ) $ such that for all $ t' \in \{ 1 , \ldots , T' \} $, there is some $t \in \{1,\ldots,T\}$ such that $ X_{ t' }' \subset X_t $. Since $ \sS' $ is subordinate to $ \sP_1( \sS^{ ( 0 ) } ) $ and since $ \sP_1 ( \sS^{ ( 0 ) } ) $ is finer than $ \sP' $, the conclusion follows.
\end{proof}

\begin{lemma}\label{lemmaFinerGeneratingSequences}
Let $ ( X , h ) $ be a zero-dimensional system and let $ ( \sP_n ) $ be a generating sequence of partitions of $ X $ (Definition \ref{defnGeneratingSequenceOfPartitions}). Let $ ( \sP_n' ) $ be a sequence of partitions such that, for every $ n \in \bZ_{ > 0 } $, $ \sP_{ n + 1 }' $ is finer than $ \sP_n' $, and for every $ n \in \bZ_{ > 0 } $, there is some $ m_n \in \bZ_{ > 0 } $ such that $ \sP_{ m_n }' $ is finer than $ \sP_n $. Then $ ( \sP_n' ) $ is a generating sequence of partitions of $ X $.
\end{lemma}

\begin{proof}
Let $ x \in X $ and let $ ( V_n ) $ be a sequence such that $ V_n \in \sP_n $ for all $ n \in \bZ_{ > 0 } $ and $ \bigcap_{ n = 1 }^\infty V_n = \{ x \} $. We inductively construct a sequence $ ( U_m ) $ such that $ U_m \in \sP_m' $ for all $ m \in \bZ_{ > 0 } $ and $ \bigcap_{ m = 1 }^\infty U_m = \{ x \} $. First, by assumption, there is an $ m_1 \in \bZ_{ > 0 } $ such that $ \sP_{ m_1 }' $ is finer than $ \sP_1 $. We can therefore choose $ U_1 , \ldots , U_{ m_1 } $ such that $ U_1 \supset \cdots \supset U_{ m_1 } $, $ U_{ m_1 } \subset V_1 $, and $ x \in U_m \in \sP_m $ for all $ m \in \{ 1 , \ldots , m_1 \} $. Next, there is an $ m_2 \in \bZ_{ > 0 } $ such that $ \sP_{ m_2 }' $ is finer than $ \sP_2 $. Since $ \sP_{ m + 1 }' $ is finer than $ \sP_m' $ for all $ m \in \bZ_{ > 0 } $, we are free to assume that $ m_2 > m_1 $. We can therefore choose $ U_{ m_1 + 1 } , \ldots , U_{ m_2 } $ such that $ U_{ m_1 + 1 } \supset \cdots \supset U_{ m_2 } $, $ U_{ m_2 } \subset V_2 $, and $ x \in U_m \in \sP_m $ for all $ m \in \{ m_1 + 1 , \ldots , m_2 \} $. Repeating this process yields $ ( U_m ) $, proving the lemma.
\end{proof}

The following construction is a key ingredient in the proof of Theorem \ref{thmMainTheorem}. It takes a zero-dimensional zero-dimensional system such that, for any partition $ \sP $ of $ X $, $ ( X , h ) $ admits a system of finite first return time maps subordinate to $ \sP $, and constructs a closed set $ Z $ and an equivalence relation $ \sim $ whose quotient map will be $ \psi $ in Definition \ref{defnFiberwiseEssentiallyMinimal}.

\begin{const}\label{constSequenceOfPartitions}
Let $ ( X , h ) $ be a zero-dimensional system such that, for any partition $ \sP $ of $ X $, $ ( X , h ) $ admits a system of finite first return time maps subordinate to $ \sP $. Let $ ( \sP^{ ( n ) } ) $ be a generating sequence of partitions of $ X $. Using Proposition \ref{propReturnTimeMapsGiveFinerPartitions}, we choose a system 
\[
\systemarg{ ( 1 ) } 
\]
of finite first return time maps subordinate to $ \sP^{ ( 1 ) } $ such that $ \sP_1 ( \sS^{ ( 1 ) } ) $  is finer than $ \sP^{ ( 2 ) } $.

We construct a sequence of systems of finite first return time maps inductively. Let $ n $ be an integer such that $ n \geq 2 $ and use Lemma \ref{lemmaNewBasesUnderOldBases} with $ \sS^{ ( n - 1 ) } $ in place of $ \sS $,  $ \sP^{ ( n - 1 ) } $ in place of $ \sP $, and $ \sP^{ ( n ) } $ in place of $ \sP' $ to get a system $ \sS^{ ( n ) \prime } $ of finite first return time maps subordinate to $ \sP^{ ( n ) } $ such that, for every $ t' \in \{ 1 , \ldots , T^{ ( n ) \prime } \} $, there is a $ t \in \{ 1 , \ldots , T^{ ( n - 1 ) } \} $ such that $ X_{ t' }^{ ( n ) \prime } \subset X_t^{ ( n ) } $. Then apply Proposition \ref{propReturnTimeMapsGiveFinerPartitions} with $ \sP^{ ( n ) } $ in place of both $ \sP $ and $ \sP' $ and with  $ \sS^{ ( n ) \prime } $ in place of $ \sS $ to get a system $ \systemarg{ ( n ) } $ such that $T^{ ( n ) } = T^{ ( n ) \prime }$, $ X_t^{ ( n ) } = X_t^{ ( n ) \prime } $ for all $ t \in \{ 1 , \ldots , T^{ ( n ) } \} $, and $ \sP_1( \sS^{ ( n ) } ) $ is finer than $ \sP^{ ( n ) } $. By Lemma \ref{lemmaFinerGeneratingSequences}, $ ( \sP_1 ( \sS^{ ( n ) } ) ) $ is a generating sequence of partitions, since, for all $ n \in \bZ_{ > 0 } $, $ \sP_1 ( \sS^{ ( n ) } ) $ is finer than $ \sP^{ ( n + 1 ) } $.

Let $ x_1 , x_2 \in X $. We say that $ x_1 \sim x_2 $ if and only if there exists a sequence $ ( t_n ) $ where $ t_n \in \{ 1 , \ldots , T^{ ( n ) } \} $ for all $ n \in \bZ_{ > 0 } $ such that $ x_1 , x_2 \in \bigcap_{ n = 1 }^\infty \bigcup_{ j \in \bZ } h^j ( X_{ t_n }^{ ( n ) } ) $. Define a set $ Z \subset X $ by $ Z = \bigcap_{ n = 1 }^\infty \bigsqcup_{ t = 1 }^{ T^{ ( n ) } } X_t^{ ( n ) } $. 
\end{const}

\begin{rmk}
Adopt the notation of Construction \ref{constSequenceOfPartitions}. We remark that, for most choices of $ ( t_n ) $, the set $ \bigcap_{ n = 1 }^\infty \bigcup_{ j \in \bZ } h^j ( X_{ t_n }^{ ( n ) } ) $ will be empty. In fact, it is nonempty if and only if, for every positive integer $ n $ with $ n \geq 2 $, we have $ X_{ t_n }^{ ( n ) } \subset X_{ t_{ n - 1 } }^{ ( n - 1 ) } $; if $ X_{ t_n }^{ ( n ) } \not\subset X_{ t_{ n - 1 } }^{ ( n - 1 ) } $, then by construction, we have $ X_{ t_n }^{ ( n ) } \cap X_{ t_{ n - 1 } }^{ ( n - 1 ) } = \varnothing $. Another thing to notice is that since for every $ n \in \bZ_{ > 0 } $ the sets $ X_1^{ ( n ) } , \ldots , X_{ T^{ ( n ) } }^{ ( n ) } $ are pairwise disjoint, the sequence $ ( t_n ) $ corresponding to an equivalence class is unique. Finally, we can see that $ z $ is in $ Z $ if and only if there is a sequence $ ( t_n ) $ such that $ z \in \bigcap_n X^{ ( n ) }_{ t_n } $.
\end{rmk}

\begin{lemma}\label{lemmaSimIsAnEquivalenceRelation}
The relation $ \sim $ from Construction \ref{constSequenceOfPartitions} is an equivalence relation.
\end{lemma}

\begin{proof}
The only thing that is nonobvious about whether or not this is an equivalence relation is whether or not all elements of $ X $ have an equivalence class. But by Proposition \ref{propPropertiesOfX_t}(\ref{propPropertiesOfX_t(b)}), for every $ n \in \bZ_{ > 0 } $, there is some $ t \in \{ 1 , \ldots , T^{ ( n ) } \} $ such that $ x \in \bigcup_{ j \in \bZ } h^j ( X_t^{ ( n ) } ) $. Thus, $ \sim $ is indeed an equivalence relation on $ X $.
\end{proof} 

\begin{lemma}\label{lemmaZIsClosed}
The set $ Z $ in Construction \ref{constSequenceOfPartitions} is a closed subset of $ X $ that contains exactly one element from each equivalence class of $ \sim $.
\end{lemma}

\begin{proof}
It is clear that $ Z $ is a closed subset of $ X $, as it is defined to be the intersection of closed subsets of $ X $.

We now show that $ Z $ contains precisely one element from each equivalence class. To see this, first let $ ( t_n ) $ be a sequence such that, for all $ n \in \bZ_{ > 0 } $, we have $ t_n \in \{ 1 , \ldots , T^{ ( n ) } \} $ and $ X_{ t_{ n + 1 } }^{ ( n + 1 ) } \subset X_{ t_n }^{ ( n ) } $. Then $ ( X_{ t_n }^{ ( n ) } ) $ is a decreasing sequence of nonempty compact open subsets of $ X $, and since the union of $ ( \sP^{ ( n ) } ) $ generates the topology of $ X $, $ \bigcap_{ n = 1 }^\infty X_{ t_n }^{ ( n ) } $ contains exactly one element, which is certainly in $ Z $. If $ x' \in X $ is another element in the same equivalence class as $ x $, then $ x' \in \bigcap_{ n = 1 }^\infty \bigcup_{ j \in \bZ } h^j ( X_{ t_n }^{ ( n ) } ) $. If we also have $ x' \in Z $, then $ x' \in \bigcap_{ n = 1 }^\infty X_{ t_n }^{ ( n ) } $, so $ x' = x $. Thus, $ Z $ indeed contains precisely one element from each equivalence class.
\end{proof}

\begin{proof}[Proof of Theorem \ref{thmFiberwiseIffAdmitsPartitions}]
($\Rightarrow$). Let $ Z $ and $ \psi $ be as in Definition \ref{defnFiberwiseEssentiallyMinimal}. Let $ \sP $ be a partition of $ X $. Let $ X_1' , \ldots , X_T' $ be the elements of $ \sP $ with nontrivial intersection with $ Z $. For each $ t \in \{ 1 , \ldots , T \} $, set $ X_t = \psi^{ -1 } ( X_t' \cap Z ) \cap X_t' $. Since $ X_1' , \ldots , X_T' $ are pairwise disjoint, it follows that $ X_1 , \ldots , X_T $ are also pairwise disjoint.

Fix $ t \in \{ 1 , \ldots , T \} $. Let $ \lambda_{ X_t } : X_t \to \bZ_{ > 0 } $ be as in Definition \ref{defnLambdaU} and set $ Z_t = X_t \cap Z $. Let $ z \in Z_t $ and set $ V_z = X_t \cap \psi^{ -1 } ( z ) $. Note that $ V_z $ is a compact open neighborhood of $ z $ in $ \psi^{ -1 } ( z ) $. Since $ h|_{ \psi^{ -1 } ( z ) } $ is essentially minimal and $ V_z $ contains an element from the minimal set of $ ( \psi^{ -1 } ( z ) , h|_{ \psi^{ -1 } ( z ) } ) $, $ \lambda_{ X_t }|_{ V_z } $ is a finite subset of $ \bZ_{ > 0 }$ by Proposition \ref{propEssentiallyMinimalReturnTimeMap}(\ref{propEssentiallyMinimalReturnTimeMap(b)}). Since this holds for all $ z \in Z_t $, we see that $\ran ( \lambda_{ X_t } )$ is a subset of $ \bZ_{ > 0 } $. Because of this, because $ X_t $ is a compact open subset of $X$, and because $ \lambda_{ X_t } $ is continuous by Proposition \ref{propLambdaUContinuous}, it follows from Proposition \ref{propEssentiallyMinimalReturnTimeMap}(\ref{propEssentiallyMinimalReturnTimeMap(b)}) that $ \ran( \lambda_{ X_t } ) $ is a finite set; thus, we can write $ \ran( \lambda_{ X_t } ) = \{ J_{ t , 1 } , \ldots , J_{ t , K_t } \} $. For each $ k \in \{ 1 , \ldots , K_t \} $, define $ Y_{ t , k } = \lambda_{ X_t }^{ -1 } ( J_{ t , K_t } ) $. 

We now check that what was defined above satisfies the conditions of Definition \ref{defnSystemOfFiniteReturnTimeMaps}. Conditions (\ref{defnSystemOfFiniteReturnTimeMaps(a)}) and (\ref{defnSystemOfFiniteReturnTimeMaps(c)}) are clearly met.  For each $ t \in \{ 1 , \ldots , T \} $, note that since $ X_t' \cap Z $ is compact and open in $ Z $ and since $ \psi $ is continuous, $ \psi^{ -1 } ( X_t' \cap Z ) $ is compact and open, and hence $ X_t $ is compact and open. Furthermore, since $ X_t \subset X_t' $, and $ X_t' $ is an element of $ \sP $, it follows that $ X_t $ is contained in an element of $ \sP $. Thus, condition (\ref{defnSystemOfFiniteReturnTimeMaps(b)}) holds. For each $ t \in \{ 1 , \ldots , T \} $, since $ \ran( \lambda_{ X_t } ) = \{ J_{ t , 1 } , \ldots , J_{ t , K_t } \} $, we also have 
\begin{align*}
\bigsqcup_{ k = 1 }^{ K_t } Y_{ t , k } &=\bigsqcup_{ k = 1 }^{ K_t } \lambda_{ X_t }^{ -1 } ( J_{ t , k } ) \\
& = X_t .
\end{align*}
Thus, condition (\ref{defnSystemOfFiniteReturnTimeMaps(d)}) holds. Recall that for each $ z \in Z $, $ ( \psi^{ -1 } ( z ) , h |_{ \psi^{ - 1 } ( z ) } ) $ is an essentially minimal system and $z$ is in its minimal set, and so since $ V_z $ is a compact open neighborhood of $z$, Proposition \ref{propEssentiallyMinimalReturnTimeMap}(\ref{propEssentiallyMinimalReturnTimeMap(c)}) tells us that $ \{ h^{ \lambda_{ V_z } } ( x )  \setdiv x \in V_z \} = V_z$. Thus, for each $ t \in \{ 1 , \ldots , T \} $, we have
\begin{align*}
\bigsqcup_{ k = 1 }^{ K_t } h^{ J_{ t , k } } ( Y_{ t , k } ) &= \{ h^{ \lambda_{ X_t } ( x ) } ( x ) \setdiv x \in X_t \} \\
&= \bigsqcup_{ z \in Z_t } \{ h^{ \lambda_{ V_z } ( x ) } ( x ) \setdiv x \in V_z \} .
\end{align*}
Since $\bigsqcup_{z \in Z_t}V_z=X_t$, this proves that (\ref{defnSystemOfFiniteReturnTimeMaps(e)}) indeed holds. Next, let $ x \in X $. There is precisely one $ z \in Z $ such that $ x \in \psi^{ -1 } ( z ) $, and precisely one $ t \in \{ 1 , \ldots , T \} $ such that $ z \in Z_t $. Let $ j  \in \bZ_{ \geq 0 } $ be the smallest nonnegative integer such that $ h^{ -j } ( x ) \in V_z \subset X_t $. This integer exists by applying Proposition \ref{propEssentiallyMinimalReturnTimeMap}(\ref{propEssentiallyMinimalReturnTimeMap(a)}), which applies since $ ( \psi^{ -1 } ( z ) , h|_{ \psi^{ -1 } ( z ) } ) $ is an essentially minimal system, $ z $ is in its minimal set, and $ V_z $ is a compact open neighborhood of $z$. It is clear that there is precisely one $ k \in \{ 1 , \ldots , K_t \} $ such that $ h^{ -j } ( x ) \in Y_{ t , k } $. Then note that $ j \in \{ 0 , \ldots , J_{ t ,k } - 1 \} $ since either $ j = 0 $ or $ h^k ( x ) \notin X_t $ for all $ k \in \{ - j  , \ldots , - 1 \} $. Thus, this proves that (\ref{defnSystemOfFiniteReturnTimeMaps(f)}) holds as well. Altogether, we see that $ ( X , h ) $ admits a system of finite first return time maps subordinate to $ \sP $.

($\Leftarrow$). Let $ ( \sP^{ ( n ) } ) $ be a generating sequence of partitions of $ X $. Use Construction \ref{constSequenceOfPartitions} to construct a sequence $ ( \sS^{ ( n ) } ) $ of finite first return time maps and adopt the notation of the construction. Define a map $ \psi : X \to Z $ by $ \psi ( x ) =  z $ if $ z \in Z $ and $ x \sim z $. Recall that by Lemma \ref{lemmaSimIsAnEquivalenceRelation}, $ \sim $ is an equivalence relation, and by Lemma \ref{lemmaZIsClosed}, $ \psi $ is a well-defined map.

We claim that $ \psi $ and $ Z $ satisfy the conditions of Definition \ref{defnFiberwiseEssentiallyMinimal}. It is obvious that $ \psi|_Z $ is the identity. To see that $ \psi $ is continuous, let $ x \in X $ and let $ V $ be an open neighborhood of $ \psi ( x ) $ in $ Z $. Since $ x \sim \psi ( x ) $, there is a sequence $ ( t_n ) $ such that $ t_n \in \bZ_{ > 0 } $ for all $ n \in \bZ_{ > 0 } $ and $ x , \psi ( x ) \in \bigcap_{ n = 1 }^\infty \bigcup_{ j \in \bZ } h^j ( X_{ t_n }^{ ( n ) } ) $. Since $ \psi ( x ) \in Z$, we have $ \psi ( x ) \in \bigcap_{ n = 1 }^\infty X_{ t_n }^{ ( n ) } $; because of this and because $ ( \sP^{ ( n ) } ) $ is a generating sequence of partitions, $ \{ X_{ t_n }^{ ( n ) }  \cap Z \setdiv \mbox{$ n \in \bZ_{ > 0 } $} \} $ is a neighborhood basis for $ \psi ( x ) $ in $ Z $, and so there is some $ n \in \bZ_{ > 0 } $ such that $ X_{ t_n }^{ ( n ) } \cap Z \subset V $. Since $ \psi ( X_{ t_n }^{ ( n ) } ) = X_{ t_n }^{ ( n ) } \cap Z $, we see $ \psi ( X_{ t_n }^{ ( n ) } ) \subset V $. Set $ U = \bigcup_{ j \in \bZ } h^j ( X_{ t_n }^{ ( n ) } ) $, which contains $ x $. Notice that since the equivalence classes of elements in $ U $ are the same as the equivalence classes of elements in $ X_{ t_n }^{ ( n ) } $, we have $ \psi ( U ) = \psi ( X_{ t_n }^{ ( n ) } ) \subset V $. Since $ x \in U $, $ U $ is a neighborhood of $ x $ such that $ \psi ( U ) \subset V $, which proves that $\psi$ is continuous.

To show that $ \psi $ is a quotient map, it remains to show that for $ S \subset Z $, if $ \psi^{ -1 } ( S ) $ is open, then $ S $ is open. But since $ \psi $ is the identity on $ Z $, we have $ \psi^{ -1 } ( S ) \cap Z = S $, and then the fact that $ \psi^{ -1 } ( S ) \cap Z $ is open in $ Z $ proves that $ \psi $ is indeed a quotient map.

To see that $ \psi \circ h = \psi $, let $ x \in X $ and let $ ( t_n ) $ be a sequence such that $ t_n \in \bZ_{ > 0 } $ and $ x \in \bigcup_{ j \in \bZ } h^j ( X_{ t_n }^{ ( n ) } ) $ for all $ n \in \bZ_{ > 0 } $. Then clearly $ h ( x ) \in \bigcup_{ j \in \bZ } h^j ( X_{ t_n }^{ ( n ) } ) $ for all $ n \in \bZ_{ > 0 } $, so $ x \sim h ( x ) $. Thus, $ \psi \circ h = \psi $.

Let $ z \in Z $. It is left to show that $ ( \psi^{ -1 } ( z ) , h|_{ \psi^{ -1 } ( z ) } ) $ is an essentially minimal system and $ z $ is in its minimal set. By Theorem 1.1 of \cite{HermanPutnamSkau92}, it suffices to show that, for every neighborhood $ V $ of $ z $ in $ \psi^{ -1 } ( z ) $, we have $\bigcup_{ j \in \bZ } h^j ( V ) = \psi^{ -1 } ( z ) $. So let $ V $ be a neighborhood of $ z $ in $\psi^{ -1 } ( z ) $, let $ V' $ be a neighborhood of $ z $ in $ X $ such that $ V' \cap \psi^{ -1 } ( z ) = V $, and let $ ( t_n ) $ be a sequence such that $ t_n \in \bZ_{ > 0 } $ for all $ n \in \bZ_{ > 0 } $ and $ z \in \bigcap_{ n = 1 }^\infty X_{ t_n }^{ ( n ) } $. Since $ ( \sP^{ ( n ) } ) $ is a generating sequence of partitions, there is some $ n \in \bZ_{ > 0 } $ such that $ X_{ t_n }^{ ( n ) } \subset V' $. Let $ x \in \psi^{ -1 } ( z ) $, so $ x \sim z $. This means that in particular, we have $ x \in X_{ t_n }^{ ( n ) } $. This tells us that $ \bigcup_{ j \in \bZ } h^j ( X_{ t_n }^{ ( n ) } \cap  \psi^{ -1 } ( z ) ) = \psi^{ -1 } ( z ) $, and since $ X_{ t_n }^{ ( n ) } \cap \psi^{ -1 } ( z )  \subset V' \cap \psi^{ -1 } ( z ) = V $, this shows us that $ \bigcup_{ j \in \bZ } h^j ( V ) = \psi^{ -1 } ( z ) $, as desired.
\end{proof}

\section{Proof of Theorem \ref{thmMainTheorem}}\label{sectionProof2}

We first state some facts about $ K $-theory that will be used in the proof of Theorem \ref{thmMainTheorem}. The following two propositions are well-known (see \cite{PimsnerVoiculescu80}).

\begin{prop}\label{propC(X)K0}
Let $ ( X , h ) $ be a zero-dimensional system. Then there is an isomorphism $ \vphi: K_0 ( C ( X ) ) \to C ( X , \bZ ) $ that sends $ [ \chi_E ] $ (where $ E $ is a compact open subset of $ X $) to $ \chi_E \in C ( X , \bZ ) $.
\end{prop}

A particular consequence of the above proposition is that if $ E_1 $ and $ E_2 $ are compact open subset of $ X $ such that $ E_1 \neq E_2 $, then $ [ \chi_{ E_1 } ] \neq [ \chi_{ E_2 } ] $.

We introduce notation important to the following proposition. Let $ \sT $ denote the Toeplitz algebra, the universal $ C^* $-algebra generated by a single isometry $ s $. Let $ A $ be a unital $ C^* $-algebra and let $ \alpha $ be an automorphism of $ A $, and let $ u $ be the standard unitary of $ C^*( \bZ , A , \alpha ) $. We denote by $ \sT ( A , \alpha ) $ the Toeplitz extension of $ A $ by $ \alpha $, which is the subalgebra of $ C^* ( \bZ , A , \alpha ) \otimes \sT $ generated by $ A \otimes 1 $ and $ u \otimes s $. The ideal generated by $ A \otimes ( 1 - s s^* ) $ is isomorphic to $ A \otimes \sK $, and the quotient by this ideal is isomorphic to $ C^* ( \bZ , A , \alpha) $.

\begin{prop}\label{propKTheoryExactSequence}
Let $ ( X , h ) $ be a zero-dimensional system. Let $ \alpha $ be the automorphism of $ C ( X ) $ induced by $ h $; that is, $ \alpha $ is defined by $ \alpha ( f ) ( x ) = f ( h^{ -1 } ( x ) ) $ for all $ f \in C ( X ) $ and all $ x \in X $. Let $ \delta $ be the connecting map obtained from the exact sequence
\begin{center}
\begin{tikzcd}
0 \arrow[r] & C(X) \otimes \sK \arrow[r, "\iota"] & \sT(C(X),\alpha) \arrow[r, "\pi"] & C^*(\bZ,A,\alpha) \arrow[r] & 0,
\end{tikzcd}
\end{center}
where $ K_0 ( C ( X ) \otimes \sK ) $ is identified with $ K_0 ( C ( X ) ) $ in the standard way. Let $ i :   C(X) \to C^* ( \bZ , X , h ) $ be the natural inclusion. Then there is an exact sequence 
\begin{center}
\begin{tikzcd}
0 \arrow[r] & K_1(C^*(\bZ,X,h)) \arrow[r, "\delta"] & K_0(C(X)) \arrow[r, "\id-\alpha_*"] & K_0(C(X)) \arrow[r, "i_*"] & K_0(C^*(\bZ,X,h)) \arrow[r] & 0.
\end{tikzcd}
\end{center}
\end{prop}

\begin{proof}
Since $ K_1 ( C ( X ) ) = 0 $, this follows immediately from Theorem 2.4 of \cite{PimsnerVoiculescu80}.
\end{proof}

\begin{lemma}\label{lemmaK1ElementsFromInvariantProjections}
Let $ ( X , h ) $ be a zero-dimensional system and let $ E \subset X $ be a compact open $ h $-invariant subset of $ X $. Adopt the notation of Proposition \ref{propKTheoryExactSequence} and set $ p = \chi_E $. Then $ \delta ( [ p u p + ( 1 - p ) ] ) = [ p ] $. Moreover, if $ E $ is nonempty, then $ [ p u p + ( 1 - p ) ] \neq 0 $.
\end{lemma}

\begin{proof}
We use the exact sequence in Proposition \ref{propKTheoryExactSequence} and the definition of the connecting map as in Definition 8.1.1 of \cite{WeggeOlsen93Book}. Let $ p_1 $ be the matrix $\begin{psmallmatrix} 1 & 0 \\ 0 & 0 \end{psmallmatrix}$ and let 
\[
w = \begin{pmatrix} u p \otimes s + ( 1 - p ) \otimes 1 & p \otimes ( 1 - s s^* ) \\ 0 & p u^* \otimes s^* + ( 1 - p ) \otimes 1 \end{pmatrix} \in M_2 ( \sT ( A , \alpha ) ) .
\]
It is straightfoward to check that 
\[
\pi ( w ) = \begin{pmatrix} u p + ( 1 - p ) & 0 \\ 0 & ( u p + ( 1 - p ) )^* \end{pmatrix}.
\]
We also have 
\[
w^* p_1 w = \begin{pmatrix} 1 & 0 \\ 0 & p \otimes ( 1 - s s^* ) \end{pmatrix}.
\]
Thus, $ \delta( [ u p + ( 1 - p ) ]  ) = [ p ] $ as desired.

If $ E $ is nonempty, by Proposition \ref{propC(X)K0}, $ [ p ] \neq 0$. Since $ \delta $ is injective, this means that $ [ u p + ( 1 - p ) ] \neq 0 $.
\end{proof}

\begin{lemma}\label{lemmaTorsionFreeK1}
Let $(X,h)$ be a zero-dimensional system and let $E$ be an $h$-invariant compact open subset of $X$. Set $p=\chi_E$. Then $K_1(p C^*(\bZ,X,h) p)$ is torsion-free.
\end{lemma}

\begin{proof}
Set $ A = C^* ( \bZ , X , h ) $ for convenience of notation. By Proposition \ref{propC(X)K0}, $ K_0 ( C ( X ) ) $ is torsion-free. By Proposition \ref{propKTheoryExactSequence}, since $ K_1 ( A ) $ embeds into $ K_0 ( C ( X ) ) $, $ K_1 ( A ) $ must be torsion-free as well. Since $ p $ is a central projection, we have $ A \cong p A p \oplus ( 1 - p ) A ( 1 - p ) $. This means that we have $ K_1 ( A ) \cong K_1 ( p A p ) \oplus K_1 ( ( 1 - p ) A ( 1 - p ) ) $. Since $ K_1 ( p A p ) $ is a direct summand in a torsion-free group, it itself is torsion-free.
\end{proof}

\begin{lemma}\label{lemmaK1Corners}
Let $ A $ be a unital $ C^* $-algebra, let $ p $ be a projection in $ A $, and let $ v $ be a unitary in $ p A p $. Then if $ [ v + ( 1 - p ) ] \neq 0 $ in $ K_1 ( A ) $, we have $ [ v ] \neq 0 $ in $ K_1 ( p A p ) $.
\end{lemma}

\begin{proof}
Suppose that $ [ v ] = 0 $ in $ K_1 ( p A p ) $. This means that there is an $ n \in \bZ_{ > 0 } $ such that $ v \oplus \underbrace{p \oplus \cdots \oplus p}_{ n - 1 \text{ times}}$ is homotopic to $ \underbrace{p \oplus \cdots \oplus p}_{n \text{ times}}$ in the unitary group of $ M_n ( p A p ) $. Let $ ( x_t )_{ t \in [ 0 , 1 ] } $ be this homotopy. Define a homotopy $ ( y_t )_{ t \in [ 0 , 1 ] } $ in $ M_n ( A ) $ by $ y_t = x_t + \underbrace{(1-p) \oplus \cdots \oplus (1-p)}_{n \text{ times}}$ for all $ t \in [ 0 , 1 ] $. Then $ y_0 = ( v + ( 1 - p ) ) \oplus \underbrace{1 \oplus \cdots \oplus 1}_{n-1 \text{ times}}$ and $ y_1 = \underbrace{1 \oplus \cdots \oplus 1}_{n \text{ times}}$, which shows that $ [ v + ( 1 - p ) ] = 0 $ in $ K_1 ( A ) $.
\end{proof}

We now prove some (probably well-known) technical results.

\begin{lemma}\label{lemmaFiniteSpectrumLogarithm}
Let $ A $ be a unital $ C^* $-algebra, let $ N \in \bZ_{ > 0 } $ and let $ v \in A $ be a unitary with finite spectrum. Then there is a unitary $ w \in C^* ( v ) $, such that $ \| w - 1 \| \leq \pi / N $ and $ w^N = v $.
\end{lemma}

\begin{proof}
Write $ \spec ( v ) = \{ \lambda_1 , \ldots , \lambda_K \} \subset S^1 $. Since $ \spec ( v ) \neq S^1 $, by functional calculus there is a self-adjoint element $ b \in A $ such that $ \exp ( b ) = v $ and such that $ \spec ( b ) \subset [ -\pi , \pi ] $. Setting $ c = ( 1 / N ) b $, we have $ \spec ( c ) \subset [ -\pi / N , \pi / N ] $. Set $ w = \exp ( c ) $, a unitary in $ A $. Clearly $ w^N = \exp ( N c ) = \exp ( b ) = v $. We compute
\begin{align*}
\| w - 1 \| &= \| \exp ( c ) - 1 \| \\
&\leq \max_{ \lambda \in \spec ( c ) } | \exp( \lambda ) - 1 | \\
&\leq \max_{ \lambda \in \spec ( c ) } | \lambda - 0 | \\
&\leq \pi / N ,
\end{align*}
finishing the proof.
\end{proof}

\begin{exmp}\label{exmpBergs}
Lemma \ref{lemmaFiniteSpectrumLogarithm} is one of the results that allows us to use Berg's technique to approximate the standard unitary $ u $ in a circle subalgebra of $ C^* ( \bZ , X , h ) $ in the proof of Theorem \ref{thmMainTheorem}. Because of the importance of this lemma, it is nice to have a concrete example in mind that is relevant to its use in the proof. Using the notation of Lemma \ref{lemmaFiniteSpectrumLogarithm}, if $ A = M_2 $ and $ v = \begin{psmallmatrix} 0 & 1 \\ 1 & 0 \end{psmallmatrix} $, then we can take $ w = \begin{psmallmatrix} \cos ( \theta ) e^{ i \theta }  & \sin ( \theta ) e^{ - i \theta }  \\ \sin ( \theta ) e^{ - i \theta }   & \cos ( \theta ) e^{ i \theta }  \end{psmallmatrix} $ where $ \theta = 2\pi/N $. This is precisely how this lemma would be used in the proof of Theorem \ref{thmMainTheorem} if we took $ X = \bZ \cup \{ \infty \} $, $ h $ to be the shift homeomorphism, $ T = 1 $,
\[
X_1 = ( ( - \infty , a ] \cap \bZ ) \sqcup  ( [ b , \infty ) \cap \bZ ) \sqcup \{ \infty \},
\]
\[
Y_1 = ( ( - \infty , a - 1 ] \cap \bZ ) \sqcup  ( [ b , \infty ) \cap \bZ ) \sqcup \{ \infty \},
\] 
and
\[
Y_2 = \{ a \} .
\]
Using the notation of the proof of Theorem \ref{thmMainTheorem}, we would have $ Y = \{ a \} \sqcup \{ b \} $ and $ \chi_{ Y } v_2 v_1^* \chi_{ Y } $ would be the unitary that switches $ \chi_{ \{ a \} } $ with $ \chi_{ \{ b \} } $, which can be realized as $  \begin{psmallmatrix} 0 & 1 \\ 1 & 0 \end{psmallmatrix} $ in $ \chi_{ Y } A_1 \chi_{ Y } \cong M_2 $. (See Example \ref{exmpIntegers} after the proof of Theorem \ref{thmMainTheorem}.)
\end{exmp}

\begin{lemma}\label{lemmaOrthogonalElementsNorm}
Let $ A $ be a $ C^* $-algebra, let $ L  \in \bZ_{ > 0 } $, and let $ a , a_1 , \ldots , a_m $ be positive elements in $ A $ such that $ a = \sum_{ m = 1 }^M a_m $ and $ a_m \perp a_{ m' } $ for $ m , m' \in \{ 1 , \ldots , M \} $ with $ m \neq m' $. Then $ \| a \| = \max_{ 1 \leq l \leq M } \| a_m \| $.
\end{lemma} 

\begin{proof}
Let $ \sH $ be a Hilbert space and let $ \pi: A \to B ( \sH ) $ be a faithful representation. Then from the operator norm on $ B ( \sH ) $, we know $ \| \pi ( a ) \| = \max_{ 1 \leq l \leq L } \| \pi ( a_l ) \| $. Since the representation is faithful, the conclusion follows. 
\end{proof}

\begin{lemma}\label{lemmaChrisProjectionCutdownLemma}
Let $A$ be a unital $C^*$-algebra, let $a \in A$, let $\eps>0$, let $ M \in \bZ_{ > 0 } $, and let $p_1,\ldots,p_M$ and $q_1,\ldots,q_M$ be projections in $A$ such that $\sum_{m=1}^M p_m = \sum_{m=1}^M q_m = 1$. Then $ p_m a q_n = 0 $ for all $ m , n \in \{ 1 , \ldots , M \} $ with $ m \neq n $ implies $ \| a \| \leq \max_m \| p_m a q_m \| $. 
\end{lemma}

\begin{proof}
Set $ \eps = \max_m \| p_m a q_m \|$. The hypotheses imply that 
\[
a = \left( \sum_{ m = 1 }^M p_m \right) a \left( \sum_{ m = 1 }^M q_m \right) = \sum_{ m = 1 }^M p_m a q_m .
\]
Now consider
\begin{align*}
a^*a &= \left( \sum_{ m = 1 }^M q_m a^* p_m \right) \left( \sum_{ m = 1 }^M p_m a q_m \right) \\
&= \sum_{ m = 1 }^M q_m a^* p_m a q_m \\
&= \sum_{ m = 1 }^M ( p_m a q_m )^* ( p_m a q_m ) .
\end{align*}
We can apply Lemma \ref{lemmaOrthogonalElementsNorm} with $ a $ replaced by $ a^*a $ and $ a_m $ replaced by $ ( p_m a q_m )^* ( p_m a q_m ) $ for all $ m \in \{ 1 , \ldots , M \} $. To check the hypotheses of the lemma, note for all $ m , m' \in \{ 1 , \ldots, M \} $ with $ m \neq m' $, we have $ q_m \perp q_{ m' }$, and so $  ( p_m a q_m )^* ( p_m a q_m )  \perp  ( p_{ m' } a q_{ m' } )^* ( p_{ m' } a q_{ m' } )  $. Now, for all $ m \in \{ 1 , \ldots , M \} $, we have $ \| p_m a q_m \| \leq \eps $, and so $ \|  ( p_m a q_m )^* ( p_m a q_m )  \| = \| p_m a q_m \|^* \leq \eps^2 $. Thus, Lemma \ref{lemmaOrthogonalElementsNorm} tells us that $ \| a^* a \| \leq \eps^2 $, and hence $ \| a \| \leq \eps $ as desired.
\end{proof}

\begin{lemma}\label{lemmaMnOfS1Isomorphism}
Let $ A $ be a unital $ C^* $-algebra, let $ n \in \bZ_{ > 0 } $, let $ ( e_{ i , j } )_{ 1 \leq i , j \leq n } $ be matrix units for a unital copy of $ M_n $ inside $ A $ (call this $ B_0 $), and let $ u \in A $ be a unitary. Let $ B $ be the $ C^* $-subalgebra of $ A $ generated by $ B_0 $ and $ u $. Suppose that $ u $ commutes with $ e_{ i , j } $ for all $ i , j \in \{ 1 , \ldots , n \} $ and that $ \spec ( u ) = S^1 $. Then $ B \cong C ( S^1 ) \otimes M_n $.
\end{lemma}

\begin{proof}
Recall that $ C ( S^1 ) \otimes M_n $ is the universal $ C^* $-algebra generated by $ ( f_{ i , j } )_{ 1 \leq i , j \leq n } $ and $ v $ satisfying the relations
\begin{enumerate}[(a)]
\item \label{lemmaMnOfS1Isomorphism(a)} $ f_{ i , j } f_{ i' , j' } = \delta_{ j , i' } f_{ i , j' } $ for all $ i , j , i' , j' \in \{ 1 , \ldots , n \} $,
\item \label{lemmaMnOfS1Isomorphism(b)} $ \sum_{ i = 1 }^n f_{ i , i } = 1 $,
\item \label{lemmaMnOfS1Isomorphism(c)} $ v v^* = v^* v = 1 $,
\item \label{lemmaMnOfS1Isomorphism(d)} $ f_{ i , j } v = v f_{ i , j  } $ for all $ i , j \in \{ 1 , \ldots , n \} $.
\end{enumerate}
Identify $ v $ with $ z \otimes 1 $, where $ z \in C ( S^1 ) $ is the identity map. Let $ ( g_{ i , j } )_{ 1 \leq i , j \leq n } $ be the standard matrix units for $ M_n $ and identify $ f_{ i , j } $ with $ 1 \otimes g_{ i , j } $ for all $ i , j \in \{ 1 , \ldots , n \} $. Since $ ( e_{ i , j } )_{ 1 \leq i , j \leq n } $ and $ u $ satisfy the relations (\ref{lemmaMnOfS1Isomorphism(a)})--(\ref{lemmaMnOfS1Isomorphism(d)}) as well, by the universal property there is a surjective $ * $-homomorphism $ \vphi: C ( S^1 ) \otimes M_n \to B $ such that $ \vphi ( z \otimes 1 ) = u $ and $ \vphi ( 1 \otimes g_{ i , j } ) = e_{ i , j } $ for all $ i , j \in \{ 1 , \ldots , n \} $. 

For each $ i , j \in \{ 1 , \ldots , n \}$, let $ N_{ i , j } $ and $ M_{ i , j } $ be integers such that $ M_{ i , j } \leq N_{ i , j }$. For each $ i , j \in \{ 1 , \ldots , n \}$ and each $ k \in \{ M_{ i , j } , \ldots , N_{ i , j } \} $, let $ \lambda_{ i , j , k } $ be a complex number. Then define
\[
x = \sum_{ i , j = 1 }^n \left( \sum_{ k = M_{ i , j } }^{ N_{ i , j } } \lambda_{ i , j , k } z^k \right) \otimes g_{ i , j } \in C ( S^1 ) \otimes M_n.
\]
Note that elements of the above form are dense in $ C ( S^1 ) \otimes M_n $. It is clear that 
\[
\vphi ( x ) = \sum_{ i , j = 1 }^n \left( \sum_{ k = M_{ i , j } }^{ N_{ i , j } } \lambda_{ i , j , k } u^k \right) e_{ i , j } .
\]
For each $ i , j \in \{ 1 , \ldots , n \} $, since $ \spec ( u ) = S^1 $, we have $ C^* ( u ) \cong C ( S^1 ) $, and so $ \sum_{ k = M_{ i , j } }^{ N_{ i , j } }  \lambda_{ i , j , k } u^k = 0 $ if and only if $ \lambda_{ i , j , k } = 0 $ for all $ k \in \{ N_{ i , j } , \ldots, M_{ i , j } \} $. But this means that $ \vphi ( x ) = 0 $ implies that $ x = 0 $, meaning that $ \ker ( \vphi ) $ is trivial. Thus, $ \vphi $ is an isomorphism.
\end{proof}

The following lemma tells us when a collection of compact open sets can be used as bases of a system of finite first return time maps. This is an important tool for constructing such systems.

\begin{lemma}\label{lemmaWhereTheBasesMustBe}
Let $ ( X , h ) $ be a fiberwise essentially minimal zero-dimensional system, let $ \sP $ be a partition of $ X $, let $ Z $ and $ \psi $ be as in Definition \ref{defnFiberwiseEssentiallyMinimal}, and let $ X_1 , \ldots , X_T $ be compact open subsets of $ X $, each of which is contained in an element of $ \sP $. Then there is a system $ \systembasic $ of finite first return time maps subordinate to $\sP$ (where $ X_1 , \ldots , X_T $ are as in the first sentence) if and only if for all $ z \in Z $, there is precisely one $ t \in \{ 1 , \ldots , T \} $ such that $ X_t $ intersects $ \psi^{ -1 } ( z ) $, and this intersection intersects the minimal set of $ ( \psi^{ -1 } ( z ) , h|_{ \psi^{ -1 } ( z ) } ) $. 
\end{lemma}

\begin{proof}
($\Rightarrow$). Suppose there is some $ z \in Z $ such that, for all $ t \in \{ 1 , \ldots , T \} $, $ X_t $ does not intersect the minimal set of $ ( \psi^{ -1 } ( z ) , h|_{ \psi^{ -1 } ( z ) } ) $. Since $ \bigcup_{ n \in \bZ } h^n ( X_t ) $ is an $ h $-invariant open set that doesn't contain $ z $, it hence doesn't contain $ \overline{ \orb ( z ) } $. Since this is true for all $ t \in \{ 1 , \ldots , T \} $, this contradicts Proposition \ref{propPropertiesOfX_t}(\ref{propPropertiesOfX_t(b)}). 

($\Leftarrow$). Let $ z \in Z $ and let $ t \in \{ 1 , \ldots , T \} $ satisfy $ X_t \cap \psi^{ -1 } ( z ) \neq \varnothing $. By our assumptions, $ X_t \cap \psi^{ -1 } ( z ) $ is a compact open subset of $ \psi^{ - 1 } ( z ) $ intersecting the minimal set of $ ( \psi^{ -1 } ( z ) , h|_{ \psi^{ -1 } ( z ) } ) $. By Proposition \ref{propEssentiallyMinimalReturnTimeMap}, $ \lambda_{ X_t \cap \psi^{ -1 } ( z ) } ( x ) < \infty $ for all $ x \in X_t \cap \psi^{ - 1 } ( z ) $. Since this holds for all $ z \in Z \cap X_t $, it follows that $ \lambda_{ X_t } ( x ) < \infty $ for all $ x \in X_t $. By Proposition \ref{propLambdaUContinuous}, $ \lambda_{ X_t } $ is continuous, and so $ \ran ( \lambda_{ X_t } ) $ is a finite subset of $ \bZ_{ > 0 } $. Write $ \ran ( \lambda_{ X_t } ) = \{ J_{ t , 1 } , \ldots J_{ t , K_t } \} $ and, for each $ k \in \{ 1 , \ldots , K_t \} $,  define $ Y_{ t , k } = \lambda_{ X_t }^{ -1 } ( J_{ t , k } ) $.

We now claim that $ \systembasic $ is a system of finite first return time maps subordinate to $ \sP $ by checking the conditions of Definition \ref{defnSystemOfFiniteReturnTimeMaps}. That (\ref{defnSystemOfFiniteReturnTimeMaps(a)}), (\ref{defnSystemOfFiniteReturnTimeMaps(b)}), and (\ref{defnSystemOfFiniteReturnTimeMaps(c)}) are satisfied is clear. Condition (\ref{defnSystemOfFiniteReturnTimeMaps(d)}) is satisfied due to the continuity of $ \lambda_{ X_t } $ for each $ t \in \{ 1 , \ldots , T \} $. Condition (\ref{defnSystemOfFiniteReturnTimeMaps(e)}) is satisfied due to Proposition \ref{propEssentiallyMinimalReturnTimeMap}(\ref{propEssentiallyMinimalReturnTimeMap(c)}). Now, let $ x \in X $. By assumption, there is precisely one $ t \in \{ 1 , \ldots T \} $ such that $ X_t \cap \psi^{ -1 } ( \psi ( x ) ) \neq \varnothing $. By Proposition \ref{propEssentiallyMinimalReturnTimeMap}(\ref{propEssentiallyMinimalReturnTimeMap(d)}), $ x \in \bigcup_{ n \in \bZ_{ > 0 } } h^n ( X_t \cap \psi^{ -1 } ( \psi ( x ) ) ) $. Let $ j \in \bZ_{ > 0 } $ be the smallest nonzero positive integer such that $ x \in h^j ( X_t ) $. Let $ k \in \{ 1 , \ldots , K_t \} $ satisfy $ x \in h^j ( Y_{ t , k } ) $. Then since $ j $ was chosen to be minimal, $ x \notin h^l ( X_t ) $ for all $ l \in \{ 0 , \ldots , j - 1 \} $, and so we must have $ j \in \{ 0 , \ldots , J_{ t , k } - 1 \} $. Suppose $ k' \in \{ 1 , \ldots , K_t \} $ and $ j' \in \{ 0 , \ldots , J_{ t , k' } - 1 \} $ are such that $ x \in h^{ j' } ( Y_{ t , k' } ) $. We have $ h^{ -j } ( x )  \in X_t $, $ h^{ J_{ t , k } - j } ( x ) \in X_t $, and $ h^{ -j + l } ( x ) \notin X_t $ for all $ l \in \{ 1 , \ldots , J_{ t , k } - 1 \} $, and so this means $ j' = j $ and hence $ k' = k $. Thus, condition (\ref{defnSystemOfFiniteReturnTimeMaps(f)}) is satisfied. This proves the claim and consequently proves the lemma.
\end{proof}

Lemmas \ref{lemmaIteratesAreContainedInThePartition} through \ref{lemmaTheLastLemma} build up the type of system we want to use in the proof of Theorem \ref{thmMainTheorem}. They are split up as separate lemmas to break what would be an enormous proof into multiple smaller steps. 

Note that in the case where $ ( X , h ) $ has no periodic points, the proofs of Lemmas \ref{lemmaOnlyOneYtkIntersectsZ} through \ref{lemmaTheLastLemma} are greatly simplified. Aside from the small amount of work to prove conclusion (\ref{lemmaTheLastLemma(a)}), Lemma \ref{lemmaTheLastLemma} is implied by Proposition  \ref{propAperiodicGivesLargeReturnTime} and Proposition \ref{propReturnTimeMapsGiveFinerPartitions}.

\begin{lemma}\label{lemmaIteratesAreContainedInThePartition}
Let $(X,h)$ be a fiberwise essentially minimal zero-dimensional system, let $N \in \bZ_{>0}$, let $\sP$ be a partition of $X$, and let $ Z $ and $ \psi $ be as in Definition \ref{defnFiberwiseEssentiallyMinimal}. Then there is a system $ \systembasic $ of finite first return time maps subordinate to $\sP$ such that:
\begin{enumerate}[(a)]
\item \label{lemmaIteratesAreContainedInThePartition(a)} For all $ t \in \{ 1 , \ldots , T \} $, we have $ \psi ( X_t ) \subset X_t $.
\item \label{lemmaIteratesAreContainedInThePartition(b)} For all $ t \in \{ 1 , \ldots , T \} $ and all $ n \in \{ 0 , \ldots , N - 1 \} $, $ h^n ( X_t ) $ is contained in an element of $ \sP $.
\item \label{lemmaIteratesAreContainedInThePartition(c)} The partitions $ \sP_1 ( \sS ) $ and $ \sP_2 ( \sS ) $ are finer than $ \sP $.
\end{enumerate}
\end{lemma}

\begin{proof}
Let $ X_1', \ldots, X_{ T' }' $ be the elements of $ \sP $ that have nonempty intersection with $ Z $. Let $ t \in \{ 1 , \ldots , T' \} $ and write $ \sP = \{ U_1 , \ldots , U_R \} $. We claim that 
\[
\widetilde{ \sP }_t = \left\{ \bigcap_{ n = 0 }^{ N - 1 } ( X_t'  \cap h^{ -n } ( U_{ r_n } ) ) \, \Bigg| \, r_n \in \{ 1 , \ldots , R \} \text{ for } n \in \{ 1 , \ldots , N - 1 \} ; \bigcap_{ n = 0 }^{ N - 1 } ( X_t' \cap h^{ -n } ( U_{ r_n } ) ) \neq \varnothing \right\}
\] 
is a partition of $ X_t' $. Clearly $ \widetilde{ \sP }_t$ is a finite set and all elements of $ \widetilde{ \sP }_t $ are compact open subsets of $ X $. Each element of $ \widetilde{ \sP }_t $ is also contained in $ X_t' $ since $ X_t' \cap h^{ - n }  ( U_r ) \subset X_t' $ for all $ n \in \{ 0 , \ldots , N - 1 \} $ and all $ r \in \{ 1 , \ldots , R \} $. 

What is left to show is that each element of $ X_t' $ is in an element of an element $ \widetilde{ \sP }_t $ and that the elements of $ \widetilde{ \sP }_t $ are pairwise disjoint. Let $ x \in X_t' $. For each $ n \in \{ 1 , \ldots , N - 1 \} $, choose $ r_n \in \{ 1 , \ldots , R \} $ such that $ h^n ( x ) \in U_{ r_n } $. Then 
\[
x \in \bigcap_{ n = 1 }^{ N - 1 }  ( X_t' \cap h^{ - n }  ( U_{ r_n } ) ) .
\]
So $ X_t $ is the union of all elements of $ \widetilde{ \sP }_t $. Now, for each $ n \in \{ 1 , \ldots ,  N - 1 \} $, choose $ r_n' \in \{ 1 , \ldots , R \} $. If 
\[
x \in \bigcap_{ n = 1 }^{ N - 1 } ( X_t' \cap h^{ - n }  ( U_{ r_n' } ) ) ,
\]
then it must be the case that $ h ( x ) \in U_{r_1}$ and $ h ( x ) \in U_{ r_1' } $, but since $ \sP $ is a partition of $ X $, this must mean that $ r_1 = r_1' $. We can repeat this process for $ h^2 ( x ) , \ldots , h^{ N - 1 } ( x ) $, showing that $ r_n = r_n' $ for all $ n \in \{ 1 , \ldots , N - 1 \} $. Thus, elements of $ \widetilde{ \sP }_t $ are pairwise disjoint, so $ \widetilde{ \sP }_t $ is indeed a partition of $ X_t $. 

Let $ \widetilde{ \sP } $ be a partition of $ X $ that contains all elements of $ \widetilde{ \sP }_t $ for all $ t \in \{ 1 , \ldots , T \} $ and is finer than $ \sP $. Let $ X_1\dprime , \ldots , X_T\dprime $ be the elements of $ \widetilde{ \sP } $ that have nonempty intersection with $ Z $. For each $ t \in \{ 1 , \ldots , T \} $, define $ X_t = X_t\dprime \cap \psi^{ -1 } ( X_t\dprime \cap Z ) $, so by construction $ \psi ( X_t ) \subset X_t $. Thus, $ X_1 , \ldots , X_T $ satisfy the hypotheses of Lemma \ref{lemmaWhereTheBasesMustBe}, and so there is a system $ \systembasic $ of finite first return time maps subordinate to $ \sP $, which naturally satisfies conclusion (\ref{lemmaIteratesAreContainedInThePartition(a)}) of this lemma. To see that conclusion (\ref{lemmaIteratesAreContainedInThePartition(b)}) of the lemma is satisfied, for each $ t \in \{ 1 , \ldots , T \} $, there exists $ s \in \{ 1 , \ldots , T' \} $ such that $ X_t $ is contained in an element of $ \widetilde{ \sP }_s $, and therefore for every $ n \in \{ 0 , \ldots , N \} $, $ h^n ( X_t ) $ is contained in an element of $ \sP $. By applying Proposition \ref{propReturnTimeMapsGiveFinerPartitions}, we may assume that conclusion (\ref{lemmaIteratesAreContainedInThePartition(c)}) holds.
\end{proof}

We improve upon conclusion (\ref{lemmaIteratesAreContainedInThePartition(a)}) of Lemma \ref{lemmaIteratesAreContainedInThePartition} in the following lemma. We add conclusion (\ref{lemmaOnlyOneYtkIntersectsZ(d)}) for ease of use of this lemma in the proof of Lemma \ref{lemmaAllJtkLargerThanNExceptJt1}, and conclusion (\ref{lemmaOnlyOneYtkIntersectsZ(d)}) is a natural result of our method of proof anyway.

To have some intuition as to what is going on in the following three lemmas, $ Y_{ t , 1 } $ should be thought of as the sole subset of $ X_t $ that intersects the periodic points of $ \psi^{ -1 } ( \psi ( X_t ) ) $ with period less than or equal to $ N $ (if such periodic points exist). Having only one piece of the base that contains periodic points of a small period is of technical relevance in the proof of Theorem \ref{thmMainTheorem}.

\begin{lemma}\label{lemmaOnlyOneYtkIntersectsZ}
Let $ ( X , h ) $ be a fiberwise essentially minimal zero-dimensional system, let $ \sP $ be a partition of $ X $, and let $ Z $ and $ \psi $ be as in Definition \ref{defnFiberwiseEssentiallyMinimal}. Then there is a system 
\[
\systembasic 
\]
such that:
\begin{enumerate}[(a)]
\item \label{lemmaOnlyOneYtkIntersectsZ(a)} For all $ t \in \{ 1 , \ldots , T \} $, we have $ \psi ( X_t ) \subset Y_{ t , 1 } $.
\item \label{lemmaOnlyOneYtkIntersectsZ(b)} For all $ t \in \{ 1 , \ldots , T \} $ and all $ n \in \{ 0 , \ldots , N - 1 \} $, $ h^n ( X_t ) $ is contained in an element of $ \sP $.
\item \label{lemmaOnlyOneYtkIntersectsZ(c)} The partitions $ \sP_1 ( \sS ) $ and $ \sP_2 ( \sS ) $ are finer than $ \sP $.
\item \label{lemmaOnlyOneYtkIntersectsZ(d)} If $ \systemarg { \prime } $ is a system of finite first return time maps subordinate to $ \sP $ that satisfies the conclusions of Lemma \ref{lemmaIteratesAreContainedInThePartition}, then we can choose $ \sS $ so that $ \bigsqcup_{ t = 1 }^T X_t = \bigsqcup_{ t = 1 }^T X_t' $.
\end{enumerate}
\end{lemma}

\begin{proof}
Let $ \systemarg{ \prime } $ be a system of finite first return time maps subordinate to $ \sP $ satisfying the conclusions of Lemma \ref{lemmaIteratesAreContainedInThePartition}. For each $ s \in \{ 1 , \ldots , T' \} $, let $ A_s = \{ a ( s , 1 ) , \ldots , a ( s , N_s ) \} $ be the set of all $ k \in \{ 1 , \ldots , K_s' \} $ such that $ Y_{ s , k }' \cap Z \neq \varnothing $. Set $ T = \sum_{ s = 1 }^{ T' } N_s $. 

Let $ t \in \{ 1 , \ldots , T \} $. There is some $ s \in \{ 1 , \ldots , T' \} $ and some $ n \in \{ 1 , \ldots , N_s \} $ such that $ t = \sum_{ r = 1 }^{ s - 1 } N_r + n $. Define $ X_t = \psi^{ -1 } ( Y_{ s , a ( s , n ) }' \cap Z ) \cap X_s $. Let $ B_t = \{ b ( t , 1 ) , \ldots , b ( t , K_t ) \} $ be the set of all $ k \in \{ 1 , \ldots , K_s' \} $ such that $ Y_{ s , k }' \cap X_t \neq \varnothing $, taking $ b ( t , 1 ) = a ( s , n ) $. For each $ k \in \{ 1 , \ldots , K_t \} $, set $ Y_{ t , k } = Y_{ s , b ( t , k ) }' \cap X_t $ and set $ J_{ t , k } = J_{ s , b ( t , k ) }' $. 

We now show that $ \systembasic $ is a system of finite first return time maps subordinate to $ \sP $ by verifying the conditions of Definition \ref{defnSystemOfFiniteReturnTimeMaps}. It is clear that conditions (\ref{defnSystemOfFiniteReturnTimeMaps(a)}) and (\ref{defnSystemOfFiniteReturnTimeMaps(c)}) are satisfied. Let $ t \in \{ 1 , \ldots , T \} $ and let $ s \in \{ 1 , \ldots , T' \} $ and $ n \in \{ 1 , \ldots , N_s \} $ satisfy $ t = \sum_{ r = 1 }^{ s - 1 } N_r + n $. Notice that $ Y_{ s , a ( s , n ) }' \cap Z $ is compact and open in $ Z $, and so by the continuity of $ \psi $, $ \psi^{ -1 } ( Y_{ s , a ( s , n ) }' \cap Z ) $ is compact and open in $ X $. Thus, $ X_t $ is compact and open. Since $ X_s' $ is contained in an element of $ \sP $, so is $ X_t $. Thus, condition (\ref{defnSystemOfFiniteReturnTimeMaps(b)}) is satisfied. For each $ k \in \{ 1 , \ldots , K_t \} $, it is clear that $ Y_{ t , k } $ is a compact open subset of $ X_t $, since it is the intersection of two compact open subsets of $ X $. Moreover, it is nonempty by construction. Now, notice that if $ k \in \{ 1 , \ldots , K_s' \} \setminus B_t $, we have $ Y_{ s , k }' \cap X_t = \varnothing $, and using this fact at the second step below, we have
\begin{align*}
\bigsqcup_{ k = 1 }^{ K_t } Y_{ t , k } &= \bigsqcup_{ k = 1 }^{ K_t } Y_{ s , b ( t , k ) }' \cap X_t  \\
&= \bigsqcup_{ k = 1 }^{ K_s' } Y_{ s , k }' \cap X_t \\
&= X_s' \cap X_t \\
&= X_t.
\end{align*}
Thus, condition (\ref{defnSystemOfFiniteReturnTimeMaps(d)}) holds. For all $ k \in \{ 1 , \ldots , K_s' \} \setminus B_t $, since $ \psi^{ -1 } ( Y_{ s , a ( s , n ) }' \cap Z )$ is $ h $-invariant and $ Y_{ s , k }' \cap \psi^{ -1 } ( Y_{ s , a ( s , n ) }' \cap Z ) = \varnothing $, we have $ h^{ J_{ s , k }' } ( Y_{ s , k }' ) \cap X_t = \varnothing $. Thus, by the same logic as above, we have 
\begin{align*}
\bigsqcup_{ k = 1 }^{ K_t } h^{ J_{ t , k } } ( Y_{ t , k } ) &= X_t.
\end{align*}
Thus, condition (\ref{defnSystemOfFiniteReturnTimeMaps(e)}) holds. 

It still remains to verify that condition (\ref{defnSystemOfFiniteReturnTimeMaps(f)}) holds. Let $ x \in X $. There is precisely one $ s \in \{ 1 , \ldots , T' \} $, one $ k \in \{ 1 , \ldots , K_s' \}$, and one $ j \in \{ 0 , \ldots , J_{ s , k }' - 1 \} $ such that $ x \in h^j ( Y_{ s , k } ) $. Observe that 
\begin{equation}\label{eq: Intersection of Y_r , a ( r , n ) and Z is Z }
\bigsqcup_{ r = 1 }^{ T' } \bigsqcup_{ m = 1 }^{ N_r } Y_{ r , a ( r , m ) }' \cap Z = Z .
\end{equation}
Thus, there is exactly one $ n \in \{ 1 , \ldots , N_s \} $ such that $ x \in \psi^{ -1 } ( Y_{ s , a ( s , n ) }' \cap Z ) $. Let $ t = \sum_{ r = 1 }^{ s - 1 } N_r + n $. We can now see that there is exactly one $ k \in B_t $ such that $ h^{ -j } ( x ) \in Y_{ t , k } $. Since $ J_{ t , k } = J_{ s , a ( s , n ) } $, we have $ j \in \{ 0 , \ldots , J_{ t , k } - 1 \}$. Thus, condition (\ref{defnSystemOfFiniteReturnTimeMaps(f)}) holds.

We now show that $ \sS $ satisfies the conclusions of the lemma. Let $ t \in \{ 1 , \ldots , T \}$ and let $ s \in \{ 1 , \ldots , T' \} $ and $ n \in \{ 1 , \ldots , N_s \} $ satisfy $ t = \sum_{ r = 1 }^{ s - 1 } N_r + n $. Since $ Y_{ t , 1 } = Y_{ s , a ( s , n ) }' \cap X_t $, it follows that $ Y_{ t , 1 } \cap Z \neq \varnothing $. Let $ k \in \{ 2 , \ldots , K_t \} $. Then there is no $ n \in \{ 1 , \ldots , N_s \} $ such that $ Y_{ s , b ( t , k ) }' = Y_{ s , a ( s , n ) }' $. Thus, by (\ref{eq: Intersection of Y_r , a ( r , n ) and Z is Z }), we have $ Y_{ t , k } \cap Z = \varnothing $. Thus, $ \sS $ satisfies conclusion (\ref{lemmaOnlyOneYtkIntersectsZ(a)}) of this lemma. Since $ X_t \subset X_s' $ and since $ \sS' $ satisfies Lemma \ref{lemmaIteratesAreContainedInThePartition}(\ref{lemmaIteratesAreContainedInThePartition(b)}), $ \sS $ satisfies conclusion (\ref{lemmaOnlyOneYtkIntersectsZ(b)}) of this lemma.

It is clear that $ \bigsqcup_{ t = 1 }^{ T' } X_t' \subset \bigsqcup_{ t = 1 }^T X_t $. Since $ \bigsqcup_{ t = 1 }^T \bigsqcup_{ n = 1 }^{ N_t } Y_{ t , a ( t , n )  }' \cap Z = Z $, we have 
\[
 \bigsqcup_{ t = 1 }^T \bigsqcup_{ n = 1 }^{ N_t } \psi^{ -1 } ( Y_{ t , a ( t , n ) } \cap Z ) = X .
\]
 Hence, 
\begin{equation}\label{eq: Sum of X_t' = Sum of X_t }
\bigsqcup_{ t = 1 }^{ T' } X_t' = \bigsqcup_{ t = 1 }^T X_t.
\end{equation}
This shows that $ \sS $ satisfies conclusion (\ref{lemmaOnlyOneYtkIntersectsZ(d)}) of this lemma. Now, note that by construction, for every $ t \in \{ 1 , \ldots , T \} $ and every $ k \in \{ 1 , \ldots , K_t \} $,  there is a $ s \in \{ 1 , \ldots , T' \} $ and a $ l \in \{ 1 , \ldots , K_s' \} $ such that $ Y_{ t , k } \subset Y_{ s , l }' $. Thus, because of (\ref{eq: Sum of X_t' = Sum of X_t }), we can apply Lemma \ref{lemmaConditionForFinerPartitions} to see  $ \sP_1 ( \sS ) $ is finer than $ \sP_1 ( \sS' ) $, which is finer than $ \sP $ since $ \sS' $ satisfies Lemma \ref{lemmaIteratesAreContainedInThePartition}(\ref{lemmaIteratesAreContainedInThePartition(c)}). By Lemma \ref{lemmaPartition1FinerIffPartition2Finer}, we see that $ \sP_2 ( \sS' ) $ is finer than $ \sP_2 ( \sS ) $, which is also finer than $ \sP $ by  Lemma \ref{lemmaIteratesAreContainedInThePartition}(\ref{lemmaIteratesAreContainedInThePartition(c)}). Thus, $ \sS $ satisfies conclusion (\ref{lemmaOnlyOneYtkIntersectsZ(c)}) of the lemma. This completes the proof.
\end{proof}

In the following lemma, we add a new property for $ \sS $ on top of the properties of $ \sS $ from Lemma \ref{lemmaOnlyOneYtkIntersectsZ}.
 
\begin{lemma}\label{lemmaAllJtkLargerThanNExceptJt1}
Let $ ( X , h ) $ be a fiberwise essentially minimal zero-dimensional system, let $ \sP $ be a partition of $ X $, let $ N , N' \in \bZ_{ \geq 0 } $ and let $ Z $ and $ \psi $ be as in Definition \ref{defnFiberwiseEssentiallyMinimal}. Then there is a system $ \systembasic $ of finite first return time maps subordinate to $ \sP $ such that:
\begin{enumerate}[(a)]
\item \label{lemmaAllJtkLargerThanNExceptJt1(a)} For all $ t \in \{ 1 , \ldots , T \} $, we have $ \psi ( X_t ) \subset Y_{ t , 1 } $.
\item \label{lemmaAllJtkLargerThanNExceptJt1(b)} For all $ t \in \{ 1 , \ldots , T \} $ and all $ n \in \{ 0 , \ldots , N \} $, $ h^n ( X_t ) $ is contained in an element of $ \sP $.
\item \label{lemmaAllJtkLargerThanNExceptJt1(c)} The partitions $ \sP_1 ( \sS ) $ and $ \sP_2 ( \sS ) $ are finer than $ \sP $.
\item \label{lemmaAllJtkLargerThanNExceptJt1(d)} For all $ t \in \{ 1 , \ldots , T \} $ and all $ k \in \{ 2 , \ldots , K_t \} $, we have $ J_{ t , k } > N' $. 
\end{enumerate}
\end{lemma}

\begin{proof}
Fix $ N \in \bZ_{ \geq 0 } $ for the remainder if the proof. We now prove this lemma by induction, with the base case $ N' = 0 $ proved by Lemma \ref{lemmaOnlyOneYtkIntersectsZ}. So let $ N' \in \bZ_{ > 0 } $ and suppose that $ \systemarg{ \prime } $ satisfies the conclusions of this lemma with $ N' - 1 $ in place of $ N' $. Let $ t \in \{ 1 , \ldots , T' \}$ and let $ B_{ t } $ be the set of all $ k \in \{ 2 , \ldots , K_t' \} $ such that $ J_{ t , k }' = N' $. Define $ X_t^{ \prime \prime } = X_t' \setminus \bigsqcup_{ k \in B_{ t } } Y_{ t ,  k }' $. 

Let $ z \in Z $. Since $ \sS' $ satisfies conclusion (\ref{lemmaAllJtkLargerThanNExceptJt1(a)}) of this lemma, there is a $ t \in \{ 1 , \ldots , T' \} $ such that $ z \in Y_{ t , 1 }' $. Thus, $ z \in X_t\dprime $. Let $ s \in \{ 1 , \ldots , T' \} $ satisfy $ s \neq t $. Then $ \psi^{ -1 } ( z ) \cap X_s' = \varnothing $, and so since $ X_s\dprime \subset X_s' $, we conclude that $ \psi^{ -1 } ( z ) \cap X_s\dprime = \varnothing $. Thus, by applying Lemma \ref{lemmaWhereTheBasesMustBe} with $ T^{ \prime \prime} $ in place of $ T $ and with $ X_1\dprime , \ldots , X_{ T\dprime }\dprime $ in place of $ X_1 , \ldots , X_T $, there is a system $ \systemarg{ \prime \prime } $ of finite first return time maps subordinate to $ \sP $. 

By applying Proposition \ref{propReturnTimeMapsGiveFinerPartitions} with $ \sS\dprime $ in place of $ \sS $ and with $ \sP $ in place of $ \sP' $, we may assume that $ \sP_1 ( \sS\dprime ) $ and $ \sP_2 ( \sS\dprime ) $ are finer than $ \sP $. By applying Lemma \ref{lemmaOnlyOneYtkIntersectsZ}, we may additionally assume that $ \psi ( X_t\dprime ) \subset Y_{ t , 1 }\dprime $ for all $ t \in \{ 1 , \ldots , T \} $. 

For each $ t \in \{ 1 , \ldots , T\dprime \} $, let $ C_t $ be the set of all $ k \in \{ 1 , \ldots , K_t\dprime \} $ such that $ J_{ t , k }\dprime = N' $. Let $ D_1 = \{ a ( 1 ) , \ldots , a ( L_1 ) \} $ be the set of all $ t \in \{ 1 , \ldots , T\dprime \} $ such that $ C_t = \varnothing $. Let $ D_2 = \{ a ( L_1 + 1 ) , \ldots , a ( L_2 ) \} $ be the set of all $ t \in \{ 1 , \ldots , T\dprime \} $ such that $ 1 \in C_t $. Let $ D_3 = \{ a ( L_2 + 1 ) , \ldots , a ( L_3 ) \} $ be the set of all $ t \in \{ 1 , \ldots , T\dprime \} $ such that $ C_t \neq \varnothing $ but $ 1 \notin C_t $. 

Let $ t \in \{ 1 , \ldots , L_1 \} $. Set $ X_t = X_{ a ( t ) }\dprime $ and $ K_t = K_{ a ( t ) }\dprime $, and for each $ k \in \{ 1 , \ldots , K_t \} $, we set $Y_{ t , k } = Y_{ a ( t ) , k }\dprime $ and $ J_{ t , k } = J_{ a ( t ) , k }\dprime $. It is clear that for all $ k \in \{ 2 , \ldots , K_t \} $, we have 
\begin{equation}\label{eq:NewSystemConstruction1-Jtk>N'}
J_{ t , k } > N' .
\end{equation}

Let $ l \in \{ L_1 + 1 , \ldots , L_2 \} $. Then set $ \widetilde{ X }_l = \left( X_{ a ( l ) }\dprime \setminus \bigsqcup_{ k \in C_{ a ( l ) } } Y_{ a ( l ) , k }\dprime \right) \sqcup  Y_{ a ( l ) , 1 }  $.  Define
\[
\sP^{ ( l ) } = \{ U \cap \psi^{ -1 } ( X_{ a ( l ) }\dprime \cap Z ) \setdiv U \in \sP \text{ and } U \cap \psi^{ -1 } ( X_{ a ( l ) }\dprime \cap Z ) \neq \varnothing \}.
\]
Apply Lemma \ref{lemmaOnlyOneYtkIntersectsZ} with $ ( \psi^{ -1 } ( X_l\dprime \cap Z ) , h|_{ \psi^{ -1 } ( X_l\dprime \cap Z ) } ) $ in place of $ ( X , h ) $, $ X_l\dprime \cap Z $ in place of $ Z $, $ \psi|_{ \psi ( X_l\dprime \cap Z ) } $ in place of $ \psi $, and $ \sP^{ ( l ) } $ in place of $ \sP $ to get a system $ \systemarg{ ( l ) \prime } $ of finite first return time maps subordinate to $ \sP^{ ( t ) } $ satisfying the conclusions of the lemma. Let $ t \in \{ 1 , \ldots , T^{ ( l ) \prime } \} $ be such that $ J_{ t , 1 }^{ ( l )\prime } = N' $; in this case, define $ X_t^{ ( l ) \prime \prime } = X_t^{ ( l ) \prime } $. Let $ t \in \{ 1 , \ldots , T^{ ( l ) } \} $ be such that $ J_{ t , 1 }^{ ( l ) } \neq N' $; in this case, let $ C ( t , l ) $ be the set of all $ k \in \{ 1 , \ldots , K_t^{ ( l )\prime } \} $ such that $ J_{ t , k }^{ ( l )\prime } = N' $ and define $ X_t^{ ( l ) \prime \prime } \setminus \left( \bigsqcup_{ k \in C ( t , l ) } Y_{ t , k }^{ ( l )\prime } \right) $. Then apply Lemma \ref{lemmaWhereTheBasesMustBe} to get a system $ \systemarg{ ( l ) \prime \prime } $ of finite first return time maps subordinate to $ \sP^{ ( l ) } $. Since $ \psi ( X_t^{ ( l )\prime } ) \subset X_t^{ ( l )\prime } $, we can apply Lemma \ref{lemmaOnlyOneYtkIntersectsZ} with $ ( \psi^{ -1 } ( X_l\dprime \cap Z ) , h|_{ \psi^{ -1 } ( X_l\dprime \cap Z ) } ) $ in place of $ ( X , h ) $, $ X_l\dprime \cap Z $ in place of $ Z $, $ \psi|_{ \psi ( X_l\dprime \cap Z ) } $ in place of $ \psi $, $ \sP^{ ( l ) } $ in place of $ \sP $, and $ \sS^{ ( l ) \prime \prime } $ in place of $ \sS' $ to get a system $ \systemarg{ ( l ) } $ of finite first return time maps subordinate to $ \sP^{ ( l ) } $ satisfying the conclusions of the lemma. For each $ t \in \{ 1 , \ldots , T^{ ( l ) } \} $, by taking the union of all subsets of $ X_t^{ ( l ) } $ with a return time of $ N' $, we are free to assume that and all $ k \in \{ 2 , \ldots , K_t^{ ( l ) } \} $, we have $ J_{ t , k }^{ ( l ) } > N' $; since $ h^j ( Y_{ t , 1 }^{ ( l ) } ) $ is a subset of $ h^j ( Y_{ s , 1 }\dprime ) $ for all $ j \in \{ 0 , \ldots , N - 1 \} $ and since $ \sP_1 ( \sS\dprime ) $ and $ \sP_2 ( \sS\dprime ) $ are finer than $ \sP $ as well, $ \sS^{ ( l ) } $ will still satisfy conclusion (\ref{lemmaAllJtkLargerThanNExceptJt1(c)}) of Lemma \ref{lemmaIteratesAreContainedInThePartition} (as well as the other conclusions). Let $ s \in \{ 1 , \ldots , T^{ ( l ) } \} $ and set $ t = L_1 + s + \sum_{ r = L_2 + 1 }^{ l - 1 } T^{ ( r ) } $, set $ X_t = X_s^{ ( l ) } $, set $ K_t = K_s^{ ( l ) } $, and for each $ k \in \{ 1 , \ldots , K_t \} $, set $ Y_{ t , k } = Y_{ s , k }^{ ( l ) } $ and set $ J_{ t , k } = J_{ s , k }^{ ( l ) } $. By construction, for all $ k \in \{ 2 , \ldots , K_t \} $, we have
\begin{equation}\label{eq:NewSystemConstruction2-Jtk>N'}
J_{ t , k } > N' . 
\end{equation}

Let $ l \in \{ L_2 + 1 , \ldots , L_3 \} $. Then set $ \widetilde{ X }_l = X_{ a ( l ) }\dprime \setminus \bigsqcup_{ k \in C_{ a ( l ) } } Y_{ a ( l ) , k }\dprime $.  Apply Lemma \ref{lemmaIteratesAreContainedInThePartition} with $ ( \psi^{ -1 } ( X_l\dprime \cap Z ) , h|_{ \psi^{ -1 } ( X_l\dprime \cap Z ) } ) $ in place of $ ( X , h ) $, $ X_l\dprime \cap Z $ in place of $ Z $, $ \psi|_{ \psi ( X_l\dprime \cap Z ) } $ in place of $ \psi $, and $ \sP^{ ( l ) } $ in place of $ \sP $ to get a system $ \systemarg{ ( l ) } $ of finite first return time maps subordinate to $ \sP^{ ( t ) } $ satisfying the conclusions of the lemma. Let $ s \in \{ 1 , \ldots , T^{ ( l ) } \} $ and set $ t = L_2 + s + \sum_{ r = L_1 + 1 }^{ l - 1 } T^{ ( r ) } $, set $ X_t = X_s^{ ( l ) } $, set $ K_t = K_s^{ ( l ) } $, and for each $ k \in \{ 1 , \ldots , K_t \} $, set $ Y_{ t , k } = Y_{ s , k }^{ ( l ) } $ and set $ J_{ t , k } = J_{ s , k }^{ ( l ) } $. For each $ k \in \{ 1 , \ldots , K_t \} $, there is some $ k' \in ( \{ 1 , \ldots , K_{ a ( l ) }\dprime \} \setminus C_{ a ( l ) } ) $ such that $ Y_{ t , k } \subset Y_{ a ( l ) , k' }\dprime $. But since $ J_{ a ( l ) , k' }\dprime > N' $ and since $ X_t \subset X_{ a ( l ) }\dprime $, we have 
\begin{equation}\label{eq:NewSystemConstruction3-Jtk>N'}
J_{ t , k } > N' .
\end{equation} 

Set $ T = L_1 + \sum_{ l = L_1 + 1 }^{ L_3 } T^{ ( l ) } $. We now show that $ \systembasic $ is indeed a system of finite first return time maps subordinate to $ \sP $ by verifying the conditions of Definition \ref{defnSystemOfFiniteReturnTimeMaps}. That conditions (\ref{defnSystemOfFiniteReturnTimeMaps(a)}) and (\ref{defnSystemOfFiniteReturnTimeMaps(c)}) hold is clear. For condition (\ref{defnSystemOfFiniteReturnTimeMaps(b)}), let $ t \in \{ 1 , \ldots , T \} $. It is clear that $ X_t $ is a compact open subset of $ X $. For $ t \in \{ 1 , \ldots , L_1 \} $, we have $ X_t \subset X_{ a ( t ) }\dprime $; for $ t \in \{ L_1 + 1 , \ldots , T \} $, let $ L \in \{ L_1 + 1 , \ldots , L_3 \} $ and let $ s \in \{ 1 , \ldots , T^{ ( L ) } \} $ satisfy $ t = L_2 + \sum_{ l = L_1 + 1 }^{ L - 1 } T^{ ( l ) } + s $, and then we have $ X_t \subset X_{ a ( L ) }\dprime $. Thus, in both cases, $ X_t' $ is contained in an element of $ \sP $. This verifies condition (\ref{defnSystemOfFiniteReturnTimeMaps(b)}). 

Conditions (\ref{defnSystemOfFiniteReturnTimeMaps(d)}) and (\ref{defnSystemOfFiniteReturnTimeMaps(e)}) are immediate from the definitions. For condition (\ref{defnSystemOfFiniteReturnTimeMaps(f)}), we have 
\begin{align*}
\bigsqcup_{ t = 1 }^{ T } \bigsqcup_{ k = 1 }^{ K_t } \bigsqcup_{ j = 0 }^{ J_{ t , k } - 1 } h^j ( Y_{ t , k } ) &= \left( \bigsqcup_{ t = 1 }^{ L_1 } \bigsqcup_{ k = 1 }^{ K_t } \bigsqcup_{ j = 0 }^{ J_{ t , k } - 1 } h^j ( Y_{ t , k } ) \right) \sqcup \left( \bigsqcup_{ t = L_1 + 1 }^{ L_3 } \bigsqcup_{ k = 1 }^{ K_t } \bigsqcup_{ j = 0 }^{ J_{ t , k } - 1 } h^j ( Y_{ t , k } ) \right)  \\
&= \left( \bigsqcup_{ t \in D_1 } \bigsqcup_{ k = 1 }^{ K_t\dprime } \bigsqcup_{ j = 0 }^{ J_{ t , k }\dprime - 1 } h^j ( Y_{ t , k }\dprime ) \right) \sqcup \left( \bigsqcup_{ t \in D_2 \cup D_3 } \bigsqcup_{ k = 1 }^{ K_t\dprime } \bigsqcup_{ j = 0 }^{ J_{ t , k }\dprime - 1 } h^j ( Y_{ t , k }\dprime ) \right) \\
&= \bigsqcup_{ t = 1 }^{ T\dprime } \bigsqcup_{ k = 1 }^{ K_t\dprime } \bigsqcup_{ j = 0 }^{ J_{ t , k }\dprime - 1 } h^j ( Y_{ t , k }\dprime )  \\
&= X.
\end{align*}

We now show that $ \sS $ satisfies the conclusions of the lemma. Conclusions (\ref{lemmaAllJtkLargerThanNExceptJt1(a)}) and (\ref{lemmaAllJtkLargerThanNExceptJt1(c)}) are satisfied by our use of Lemma \ref{lemmaIteratesAreContainedInThePartition} throughout the proof. Conclusion (\ref{lemmaAllJtkLargerThanNExceptJt1(b)}) is satisfied by the fact that $ \sS\dprime $ satisfies this condition and each base of $ \sS $ is contained in a base of $ \sS\dprime $. Conclusion (\ref{lemmaAllJtkLargerThanNExceptJt1(d)}) is satisfied by equations (\ref{eq:NewSystemConstruction1-Jtk>N'}), (\ref{eq:NewSystemConstruction2-Jtk>N'}), and (\ref{eq:NewSystemConstruction3-Jtk>N'}). This completes the proof of the lemma.
\end{proof}

In our final lemma, we improve conclusion (\ref{lemmaAllJtkLargerThanNExceptJt1(d)}) of Lemma \ref{lemmaAllJtkLargerThanNExceptJt1} at the expense of weakening conclusion (\ref{lemmaAllJtkLargerThanNExceptJt1(a)}). Luckily, this weakened conclusion (\ref{lemmaAllJtkLargerThanNExceptJt1(a)}) is still sufficient for our use of this system in the proof of Theorem \ref{thmMainTheorem}.

\begin{lemma}\label{lemmaTheLastLemma}
Let $ ( X , h ) $ be a fiberwise essentially minimal zero-dimensional system, let $ \sP $ be a partition of $ X $, let $ N \in \bZ_{ > 0 } $. Then there is a system $ \systembasic $ of finite first return time maps subordinate to $ \sP $ such that, setting $ \widehat{ X }_t = X_t \setminus ( Y_{ t , 1 } \cap h^{ J_{ t , 1 } } ( Y_{ t , 1 } ) ) $ for each $ t \in \{ 1 , \ldots , T \} $, we have:
\begin{enumerate}[(a)]
\item \label{lemmaTheLastLemma(a)} For each $ t \in \{ 1 , \ldots , T \} $ and each $ z \in \psi ( X_t ) $, $ Y_{ t , 1 } $ intersects the minimal set of $ ( \psi^{ -1 } ( z ) , h|_{ \psi^{ -1 } ( z ) } ) $.
\item \label{lemmaTheLastLemma(b)} For all $ t \in \{ 1 , \ldots , T \} $ and all $ n \in \{ 0 , \ldots , N - 1 \} $, $ h^n ( X_t ) $ is contained in an element of $ \sP $.
\item \label{lemmaTheLastLemma(c)} The partitions $ \sP_1 ( \sS ) $ and $ \sP_2 ( \sS ) $ are finer than $ \sP $.
\item \label{lemmaTheLastLemma(d)} The sets $ \widehat{ X }_t, h ( \widehat{ X }_t ) , \ldots , h^{ N } ( \widehat{ X }_t ) $ are pairwise disjoint.
\end{enumerate}
\end{lemma}

\begin{proof}
Apply Lemma \ref{lemmaAllJtkLargerThanNExceptJt1} with $ N $ in place of $ N' $ to get a system $ \systemarg{ \prime } $ of finite first return time maps subordinate to $ \sP $ that satisfies the conclusions of the lemma. Let $ t \in \{ 1 , \ldots , T' \} $. Define 
\[
B_t = h^{ J_{ t , 1 }' } ( Y_{ t , 1 }' ) \cap \left ( \bigsqcup_{ k = 2 }^{ K_t' } Y_{ t , k }' \right)
\]
and for each $ n \in \{ 1 , \ldots , N \} $ define 
\[
C_{ t , n } = B_t \cap \left( \bigcup_{ k = 2 }^{ K_t' } h^n ( Y_{ t , 1 }' \cap h^{ J_{ t , k }' } ( Y_{ t , k }' ) ) \right). 
\]
Set $ T\dprime = T' $ and define
\[
X_t\dprime = X_t' \setminus \left( \bigcup_{ n = 1 }^N  \bigcup_{ m = 1 }^n h^{ -m } ( C_{ t , n } ) \cap X_t' \right).
\]

Let $ z \in Z $. Since $ \sS' $ satisfies the conclusions of Lemma \ref{lemmaAllJtkLargerThanNExceptJt1}, there is precisely one $ t \in \{ 1 , \ldots , T' \} $ such that $ z \in X_t' $. If $ z \in X_t\dprime $, then clearly $ X_t\dprime $ intersects the minimal set of $( \psi^{ -1 } ( z ) , h|_{ \psi^{ -1 } ( z ) } ) $. Suppose $ z \notin X_t\dprime $. Then there is some $ n \in \{ 1 , \ldots , N \} $ such that $ z \in \bigcup_{ m = 1 }^n h^{ -m } ( C_{ t , n } ) \cap X_t' $, which means that there is some $ m \in \{ 1 , \ldots , n \} $ such that 
\begin{equation}\label{eq:hn(z) is in Xtdprime}
h^m ( z ) \in X_t\dprime \cap h^{ J_{ t , 1 }' } ( Y_{ t , 1 }' ).
\end{equation}
Since $ h^m ( z ) $ is in the minimal set of $ ( \psi^{ -1 } ( z ) , h|_{ \psi^{ -1 } ( z ) } ) $, we see $ X_t\dprime $ intersects the minimal set of $ ( \psi^{ -1 } ( z ) , h|_{ \psi^{ -1 } ( z ) } ) $. By Lemma \ref{lemmaWhereTheBasesMustBe}, there is a system $ \systemarg{ \prime \prime } $ of finite first return time maps subordinate to $ \sP $. By using Proposition \ref{propReturnTimeMapsGiveFinerPartitions}, we are free to assume that $ \sP_1 ( \sS\dprime ) $ and $ \sP_2 ( \sS\dprime ) $ are finer than $ \sP $.

Let $ t \in \{ 1 , \ldots, T' \} $, suppose that $ J_{ t , 1 }' \leq N $, and suppose that $ \big\{ x \in X_t\dprime \setdiv \mbox{$\lambda_{ X_t\dprime } ( x ) = J_{ t , 1 }' $} \big\} $ is nonempty. We claim that 
\[
\big\{ x \in X_t\dprime \setdiv \mbox{$\lambda_{ X_t\dprime } ( x ) = J_{ t , 1 }' $} \big\} =  Y_{ t , 1 }' \cap X_t\dprime.
\] 
Let $ x \in X_t\dprime $ and suppose $ \lambda_{ X_t\dprime } ( x ) = J_{ t , 1 }' $. Then since $ X_t\dprime \subset X_t' $, we have $ \lambda_{ X_t' } ( x ) \leq J_{ t , k }' $. Thus, by Lemma \ref{lemmaAllJtkLargerThanNExceptJt1}(\ref{lemmaAllJtkLargerThanNExceptJt1(b)}), we have $ x \in Y_{ t , 1 }' $. Thus, $ \big\{ x \in X_t\dprime \setdiv \mbox{$ \lambda_{ X_t\dprime } ( x ) = J_{ t , 1 }' $} \big\} \subset  Y_{ t , 1 }' \cap X_t\dprime $.  So suppose $ x \in X_t\dprime \cap Y_{ t , 1 }' $ and suppose $ h^{ J_{ t , 1 }' } ( x ) \notin X_t\dprime $. This means that there is some $ n \in \{ 1 , \ldots , N \} $ and some $ m \in \{ 1 , \ldots , n \} $ such that $ h^{ J_{ t , 1 }' } ( x ) \in h^{ -m } ( C_{ t , n } ) $. But then notice that $ h^j ( h^{ J_{ t , 1 }' } ( x ) ) \notin X_t\dprime $ for all $ j \in \{ 0 , \ldots , n - m \} $, and so since $ x \in X_t\dprime $, this must mean that $ n - m \leq J_{ t , 1 }' $. But since $ h^{ J_{ t , 1 }' + m } ( x )  \in C_{ t , n } $, we have $ h^{ J_{ t , 1 }' + m - n } ( x ) \in h^{ -n } ( C_{ t , n } ) \subset X_t' $, a contradiction to $ \lambda_{ X_t' } ( x ) = J_{ t , 1 }' $. Thus, $ \big\{ x \in X_t\dprime \setdiv \mbox{$\lambda_{ X_t\dprime } ( x ) = J_{ t , 1 }' $} \big\} =  Y_{ t , 1 }' \cap X_t\dprime $, and so we are free to assume that 
\[
\big\{ x \in X_t\dprime \setdiv \mbox{$\lambda_{ X_t\dprime } ( x ) = J_{ t , 1 }' $} \big\} = Y_{ t , 1 }\dprime,
\] 
by combining all $ Y_{ t , k }\dprime $ with $ J_{ t , k }\dprime = J_{ t , 1 }' $, since this assumption does not contradict the fact that $ \sP_1 ( \sS\dprime ) $ and $ \sP_2 ( \sS\dprime ) $ are finer than $ \sP $, since we have $ J_{ t , 1 }\dprime = J_{ t , 1 }' $ and $ h^j ( Y_{ t , 1 }\dprime ) \subset h^j ( Y_{ t , 1 }' )$ for all $ j \in \{ 0 , \ldots , J_{ t , 1 }' \} $. Notice that by (\ref{eq:hn(z) is in Xtdprime}), for all $ z \in \psi ( X_t\dprime ) $, $ Y_{ t , 1 }\dprime $ intersects the minimal set of $ ( \psi^{ -1 } ( z ) , h|_{ \psi^{ -1 } ( z ) } ) $. Moreover, for all $ k \in \{ 2 , \ldots , K_t\dprime \} $, we have $ Y_{ t , k }\dprime \subset \bigsqcup_{ l = 2 }^{ K_t' } Y_{ t , k }' $, and so $ J_{ t , k }\dprime >  N $. 

Now, note that by (\ref{eq:hn(z) is in Xtdprime}), we have
\[
Z \subset \bigcup_{ j = -N }^0 \bigsqcup_{ t = 1 }^{ T\dprime } h^j ( X_t\dprime ) .
\]
Also, observe that if $ t \in \{ 1 , \ldots , T\dprime \} $, $ k \in \{ 1 , \ldots , K_t\dprime \} $, and $ z \in Z \cap \bigcup_{ j = -N }^0 h^j ( Y_{ t , k }\dprime ) $, then $ Y_{ t , k }\dprime $ intersects the minimal set of $ ( \psi^{ -1 } ( z ) , h|_{ \psi^{ -1 } ( z ) } ) $.

Let $ \{ a ( 1 ) , \ldots , a ( L ) \} $ be the set of all $ t \in \{ 1 , \ldots , T\dprime \} $ such that $ Z \cap \bigcup_{ j = -N }^0 h^j ( Y_{ t , 1 }\dprime ) \neq \varnothing $. For each $ t \in \{ 1 , \ldots , L \} $, define 
\[
X_t = X_{ a ( t ) }\dprime \cap \psi^{ -1 } \left( Z \cap \bigcup_{ j = -N }^0 h^j ( Y_{ a ( t ) , 1 }\dprime ) \right).
\]
Let $ \{ b ( t , 1 ) , \ldots , b ( t , K_l ) \} $ be the set of all $ k \in \{ 1 , \ldots , K_{ a ( t ) }\dprime \} $ such that $ Y_{ a ( t ) , k }\dprime \cap X_t \neq \varnothing $, making the choice 
\begin{equation}\label{eq:Condition(a)1}
b ( t , 1 ) = 1 .
\end{equation}
For each $ k \in \{ 1 , \ldots , K_t \} $, define $ Y_{ t , k } = Y_{ a ( t ) , b ( t , k ) }\dprime \cap X_t $ and define $ J_{ t , k } = J_{ a ( t ) , b ( t , k ) }\dprime $. Since for all $ z \in \psi ( X_{ a ( t ) }\dprime ) $, $ Y_{ a ( t ) , 1 }\dprime $ intersects the minimal set of $ ( \psi^{ -1 } ( z ) , h|_{ \psi^{ -1 } ( z ) } ) $, by the above definitions we see that for all $ z \in \psi ( X_t ) $, $ Y_{ t , 1 } $ intersects the minimal set of $ ( \psi^{ -1 } ( z ) , h|_{ \psi^{ -1 } ( z ) } ) $.

Set 
\begin{equation}\label{eq:widetildeX definition}
\widetilde{ X } = X \setminus \left ( \bigsqcup_{ t = 1 }^L \bigsqcup_{ k = 1 }^{ K_t } \bigsqcup_{ j = 0 }^{ J_{ t , l } - 1 } h^j ( Y_{ t , k } )  \right).
\end{equation}
If $ \widetilde{ X } $ is empty, set $ T = L $. Otherwise, set $ \widetilde{ Z } = Z \cap \widetilde{ X } $ and set $ \widetilde{ \psi } = \psi|_{ \widetilde{ X } } $. Notice that if $ x \in \widetilde{ X } $, then $ \psi ( x ) \notin \bigsqcup_{ t = 1 }^L \bigcup_{ j = -N }^0 h^j ( Y_{ a ( t ) , 1 }\dprime ) $, and therefore $ \psi ( x ) \in \widetilde{ Z } $. Thus, $ ( \widetilde{ X }, h|_{ \widetilde{ X } } ) $ is a fiberwise essentially minimal zero-dimensional system where $ \widetilde{ \psi } : \widetilde{ X } \to \widetilde{ Z } $ satisfies the requirements of Definition \ref{defnFiberwiseEssentiallyMinimal}. Notice that for each $ z \in \widetilde{ Z } $, there is a $ t \in \{ 1 , \ldots , T\dprime \} $ and a $ k \in \{ 2 , \ldots , K_t\dprime \} $ such that 
\begin{equation}\label{eq:z in widetildeX}
z \in \left( \widetilde{ X } \cap \bigcup_{ j = -N }^0 h^j ( Y_{ t , k }\dprime ) \right) .
\end{equation} 

Let $ \{ a' ( 1 ) , \ldots , a' ( \widetilde{ T } ) \} $ be the set of all $ t \in \{ 1 , \ldots , T\dprime \} $ such that $ \widetilde{ Z } \cap  \bigcup_{ j = -N }^0 h^j ( X_t\dprime ) = \varnothing $. For each $ t \in \{ 1 , \ldots , \widetilde{ T } \} $, set $ \widetilde{ X }_t = \widetilde{ X } \cap \left( X_{ a' ( t ) }\dprime \setminus Y_{ a' ( t ) , 1 }\dprime \right) $. Then from (\ref{eq:z in widetildeX}), we see that for every $ z \in \widetilde{ Z } $, there is a $ t \in \{ 1 , \ldots , \widetilde{ T } \}$ such that $ \widetilde{ X }_t $  intersects the minimal set of $ ( \widetilde{\psi}^{ -1 } ( z ) , h|_{ \widetilde{\psi}^{ -1 } ( z ) } ) $. Thus, by Lemma \ref{lemmaWhereTheBasesMustBe}, there is a system 
\[
\widetilde{ \sS } = ( \widetilde{ T } , ( \widetilde{ X }_t ) , ( \widetilde{ K }_t ) , ( \widetilde{ Y }_{ t , k } ) , ( \widetilde{ J }_{ t , k } ) ) , 
\]
of finite first return time map subordinate to 
\[ 
\widetilde { \sP } = \big\{ U \cap \widetilde{ X } \setdiv \mbox{$ U \in \sP $ and $ U \cap \widetilde{ X } \neq \varnothing$} \big\}.
\]
For each $ t \in \{ 1 , \ldots , \widetilde{ T } \} $, let $ \{ c ( t , 1 ) , \ldots , c ( t , N_t ) \} $ be the set of all $ k \in \{ 1 , \ldots , \widetilde{ K }_t \} $ such that $ \widetilde{ Z } \cap \bigcup_{ j = -N }^0 h^j ( \widetilde{ Y }_{ t , k } ) \neq \varnothing $. 

Let $ s \in \{ 1 , \ldots , \widetilde{ T } \} $, let $ n \in \{ 1 , \ldots , N_s \} $ and set $ t = L + n + \sum_{ r = 1 }^ { s - 1 } N_r $. Define 
\[
X_t = \widetilde{ X }_s \cap \psi^{ - 1 } \left( \widetilde{ Z } \cap \bigcup_{ j = -N }^0 h^j ( \widetilde{ Y }_{ s , c ( s , n ) } ) \right).
\]
Let $ \{ d ( t , 1 ) , \ldots , d ( t , K_t ) \} $ be the set of all $ k \in \{ 1 , \ldots , \widetilde{ K }_s \} $ such that $ \widetilde{ Y }_{ s , k } \cap X_t \neq \varnothing $, where we make the choice 
\begin{equation}\label{eq:Condition(a)2}
d ( t , 1 ) = c ( s , n ) .
\end{equation} For each $ l \in \{ 1 , \ldots , K_t \} $, set $ Y_{ t , l } = X_t \cap \widetilde{ Y }_{ s , d ( t , l ) } $ and set $ J_{ t , l } = \widetilde{ J }_{ s , d ( t , l ) } $.

Set $ T = L + \sum_{ r = 1 }^{ \widetilde{ T } } N_r $. We now check that, for $ ( X , h ) $, $ \systembasic $ is a system of finite first return time maps subordinate to $ \sP $ by verifying the conditions of Definition \ref{defnSystemOfFiniteReturnTimeMaps}. That conditions (\ref{defnSystemOfFiniteReturnTimeMaps(a)}) and (\ref{defnSystemOfFiniteReturnTimeMaps(c)}) hold is clear. Let $ t \in \{ 1 , \ldots , L \} $. Then $ \bigcup_{ j = -N }^0 h^j ( Y_{ t , 1 }\dprime ) $ is compact and open in $ X $, and so $ Z \cap \bigcup_{ j = -N }^0 h^j ( Y_{ a ( t ) , 1 }\dprime ) $ is compact and open in $ Z $, and by the continuity of $ \psi $, $ \psi^{ -1 } (  Z \cap \bigcup_{ j = -N }^0 h^j ( Y_{ a ( t ) , 1 }\dprime ) ) $ is compact and open in $ X $, and so $ X_t $ is therefore compact and open in $ X $. It is also clear by construction that $ X_t $ is nonempty. Moreover, since $ X_t \subset X_{ a ( t ) }\dprime $, and $ X_{ a ( t ) }\dprime $ is contained in an element of $ \sP $, so is $ X_t $. Now, let $ t \in \{ L + 1 , \ldots , T \} $. By the exact same reasoning, $ X_t $ is a nonempty compact open subset of $ X $. Since $ \widetilde{ \sS } $ is subordinate to $ \widetilde { \sP }$, and since every element of $ \widetilde{ \sP } $ is contained in an element of $ \sP $, we see that $ X_t $ is contained in an element of $ \sP $. Thus, condition (\ref{defnSystemOfFiniteReturnTimeMaps(b)}) holds. It is now routine to verify conditions (d), (e), and (f) from the definitions, so we omit the computations.

We now show that $ \sS $ satisfies the conclusions of the lemma. Conclusion (\ref{lemmaTheLastLemma(a)}) follows from (\ref{eq:Condition(a)1}) and (\ref{eq:Condition(a)2}). Conclusion (\ref{lemmaTheLastLemma(b)}) follows from the fact that for each $ t \in \{ 1 , \ldots , T \} $, there is some $ s \in \{ 1 , \ldots , T\dprime \} $ such that $ X_t \subset X_s\dprime \subset X_s' $, and the fact that $ \sS' $ satisfies Lemma \ref{lemmaAllJtkLargerThanNExceptJt1}(\ref{lemmaAllJtkLargerThanNExceptJt1(b)}). Conclusion (\ref{lemmaTheLastLemma(c)}) follows from the definitions of $ \widetilde{ \sS } $ and of $ X_t $ and $ Y_{ t , k } $ for $ t \in \{ 1 , \ldots , T \} $ and $ k \in \{ 1 , \ldots , K_t \} $, along with an application of Lemmas \ref{lemmaConditionForFinerPartitions} and \ref{lemmaPartition1FinerIffPartition2Finer} along with the fact that $ \sP_1 ( \sS\dprime ) $ and $ \sP ( \sS\dprime ) $ are finer than $ \sP $. For (\ref{lemmaTheLastLemma(d)}), notice that if $ t \in \{ L + 1 , \ldots, T \} $, then $ J_{ t , k } > N $ for all $ k \in \{ 1 , \ldots , K_t \} $, and so $ \widehat{ X }_t , h ( \widehat{ X }_t ) , \ldots , h^N ( \widehat{ X }_t ) $ are clearly pairwise disjoint. So let $ t \in \{ 1 , \ldots , L \} $, let $ x \in \widehat{ X }_t $, and suppose there is some $ n \in \{ 1 , \ldots , N \} $ such that $ h^n ( x ) \in \widehat{ X }_t $. This means that $ \lambda_{ X_t } ( x ) \leq N $, and, from our work above, we know that means $ x \in Y_{ a ( t ) , 1 }' $. By definition of $ \widehat{ X }_t $, this means that there is some $ k \in \{ 2 , \ldots , K_{ a ( t ) }' \} $ such that $ x \in Y_{ a ( t ) , 1 }' \cap h^{ J_{ a ( t ) , k }' } ( Y_{ a ( t ) , k }' ) $.  We now have two cases. First, suppose that $ h^n ( x ) \in B_{ a ( t ) } $. But since $ x \in Y_{ a ( t ) , 1 }' \cap h^{ J_{ a ( t ) , k }' } ( Y_{ a ( t ) , k }' ) $, this would mean that $ h^n ( x ) \in C_{ t , n } $, which means that $ x \notin X_t\dprime $, a contradiction. The second possibility is that $ h^n ( x ) \in Y_{ a ( t ) , 1 }' \cap h^{ J_{ a ( t ) , l }' } ( Y_{ a ( t ) , l }' ) $ for some $ l \in \{ 2 , \ldots , K_{ a ( t ) }' \} $. But then since $ J_{ a ( t ) , l }' > n $, we would have $ x =  h^{ -n } ( h^n ( x ) )  \notin X_{ a ( t ) }' $, a contradiction. Thus, $ \widehat{ X }_t \cap h^n ( \widehat{ X }_t ) = \varnothing $. This proves that $ \sS $ satisfies the conclusions of the lemma and therefore proves the lemma.
\end{proof}

\begin{proof}[Proof of Theorem \ref{thmMainTheorem}]
Let $ \eps > 0 $, let $ N \in \bZ_{ > 0 } $ satisfy $ \pi / N < \eps $, and let $ \sP $ be a partition of $ X $. Following the proof of Theorem 2.1 of \cite{Putnam90}, we will show that there is a $ C^* $-subalgebra $ A $ of $ C^*( \bZ , X , h ) $ which is isomorphic to a direct sum of matrix algebras and matrix algebras over $ C( S^1 ) $ such that $ C ( \sP ) \subset A $ and such that $ A $ contains a unitary $ u' $ such that $ \| u' - u \| < \eps $. By using the semiprojectivity of circle algebras to construct a direct system, this will imply that $ C^* ( \bZ , X , h ) $ is an A$ \bT $-algebra.

Let $ Z $ and $ \psi $ be as in Definition \ref{defnFiberwiseEssentiallyMinimal}. Apply Lemma \ref{lemmaTheLastLemma} to get a system 
\[
\systembasic
\]
of finite first return time maps subordinate to $ \sP $ satisfying the conclusions of the lemma. Now, notice that Lemma \ref{lemmaTheLastLemma}(\ref{lemmaTheLastLemma(a)}) says that for each $ t \in \{ 1 , \ldots , T \} $ and each $ z \in  \psi ( X_t ) $, $ Y_{ t , 1 } $ intersects the minimal set of $ ( \psi^{ -1 } ( z ) , h|_{ \psi^{ -1 } ( z ) } ) $, and hence, $ h^{ J_{ t , 1 } } ( Y_{ t , 1 } ) $ intersects the minimal set of $ ( \psi^{ -1 } ( z ) , h|_{ \psi^{ -1 } ( z ) } ) $. Since by Proposition \ref{propPropertiesOfX_t}(\ref{propPropertiesOfX_t(b)}) we have $ \bigsqcup_{ t = 1 }^{ T } \bigcup_{ j \in \bZ } h^j ( X_t ) = X $, it follows that for every $ z \in Z $, there is a unique $ t \in \{ 1 , \ldots , T \} $ such that $ h^{ J_{ t , 1 } } ( Y_{ t , 1 } ) $ intersects the minimal set of $ ( \psi^{ -1 } ( z ) , h|_{ \psi^{ -1 } ( z ) } ) $. Thus, we can apply Lemma \ref{lemmaWhereTheBasesMustBe} with $ h^{ J_{ 1 , 1 } } ( Y_{ 1 , 1 } ), \ldots , h^{ J_{ T , 1 } } ( Y_{ T , 1 } ) $ in place of $ X_1 , \ldots , X_T $ to get a system $ \systemarg{ \prime } $ of finite first return time maps subordinate to $ \sP $ where $ T' = T $ and 
\[
X_t' = h^{ J_{ t , 1 } } ( Y_{ t , 1 } )
\]
for all $ t \in \{ 1 , \ldots , T \} $. By applying Proposition \ref{propReturnTimeMapsGiveFinerPartitions}, we may assume that $ \sP_1 ( \sS' ) $ is finer than both $ \sP_1 ( \sS ) $ and $ \sP_2 ( \sS ) $.

For each $ t \in \{ 1 , \ldots , T \} $, each $ k \in \{ 1 , \ldots , K_t \}$, and each $ i , j \in \{ 0 , \ldots , J_{ t , k } - 1 \}$, define 
\[
e_{ i , j }^{ ( t , k ) } = \chi_{ h^i ( Y_{ t , k } ) } u^{ i - j } \chi_{ h^j ( Y_{ t , k } ) } .
\]
One can check that these elements are matrix units for a finite dimensional $ C^* $-subalgebra of $ C^* ( \bZ , X , h ) $ isomorphic to $ \bigoplus_{ t = 1 }^{ T } \bigoplus_{ k = 1 }^{ K_t } M_{ J_{ t , k } }$. Notice that $ C ( \sP_1 ( \sS ) ) \subset C^* ( \bZ , X , h ) $ is equal to the set of diagonal matrices in $ A_n $. Since by $ \sP_1 ( \sS ) $ is finer than $ \sP $, we have $ C ( \sP ) \subset C ( \sP_1 ( \sS ) ) $, and so it follows that 
\[
C ( \sP ) \subset A_1 .
\]

Define an element $ v_1 \in A_1 $ by 
\[
v_1 = \sum_{ t = 1 }^{ T } \sum_{ k = 1 }^{ K_t } \left( \chi_{ Y_{ t , k } } u^{ 1 - J_{ t , k } } \chi_{ h^{ J_{ t , k } - 1 } ( Y_{ t , k } ) } + \sum_{ j = 0 }^{ J_{ t , k } - 2 } \chi_{ h^{ j + 1 } ( Y_{ t , k } ) } u \chi_{ h^j ( Y_{ t , k } ) } \right).
\]
To see what $ v_1 $ does, let $ t \in \{ 1 , \ldots , T \}$ and let $ k \in \{ 1 , \ldots , K_t \} $, and observe that for $ j \in \{ 0 , \ldots , J_{ t , k } - 2 \}$ we have
\begin{align}
v_1 \chi_{ h^j ( Y_{ t , k } ) } v_1^* &= \chi_{ h^{ j + 1 } ( Y_{ t , k } ) } u \chi_{ h^j ( Y_{ t , k } ) } u^* \chi_{ h^{ j + 1 } ( Y_{ t , k } ) } \nonumber \\
&= \chi_{ h^{ j + 1 } ( Y_{ t , k } ) } . \label{eq:v1 Moves h^j(Ytk)}
\end{align}
We also have 
\begin{align}
v_1 \chi_{ h^{ J_{ t , k } - 1 } ( Y_{ t , k } ) } v_1^* &= \chi_{ Y_{ t , k } } u^{ 1 - J_{ t , k } } \chi_{ h^{ J_{ t , k } - 1 } ( Y_{ t , k } ) } u^{ J_{ t , k } - 1 } \chi_{ Y_{ t , k } } \nonumber \\
&= \chi_{ Y_{ t , k } }. \label{eq:v1 Moves h^Jtk(Ytk)}
\end{align}
Define $ u_1 = v_1^* u $. First note the formula for $ u_1 $:
\begin{align}
u_1 &= v_1^* u \nonumber \\
&= \sum_{ t = 1 }^T \sum_{ k = 1 }^{ K_t } \left( \chi_{ h^{ J_{ t , k } - 1 } ( Y_{ t , k } ) } u^{ J_{ t , k } } \chi_{ h^{ -1 } ( Y_{ t , k } ) } + \sum_{ j = 0 }^{ J_{ t , k } - 2 } \chi_{ h^j ( Y_{ t , k } ) } \right) . \nonumber 
\end{align}
To see what $ u_1 $ does, let $ t \in \{ 1 , \ldots , T \}$ and let $ k \in \{ 1 , \ldots , K_t \}$, and observe that for $ j \in \{ 0 , \ldots , J_{ t , k }  - 2 \} $ we have
\begin{align}
u_1 \chi_{ h^j ( Y_{ t , k } ) } u_1^* &= v_1^* \chi_{ h^{ j + 1 } ( Y_{ t , k }  ) } v_1 \nonumber  \\
&= \chi_{ h^j ( Y_{ t , k } ) } , \nonumber 
\end{align}
where the last line is justified by $(\ref{eq:v1 Moves h^j(Ytk)})$. We also have 
\begin{align}
u_1 \chi_{ h^{ -1 } ( Y_{ t , k } ) } u_1^* &= v_1^* \chi_{ Y_{ t , k }  } v_1 \nonumber \\
&= \chi_{ h^{ J_{ t , k } - 1 } ( Y_{ t , k } ) } , \nonumber
\end{align}
where the last line is justified by $(\ref{eq:v1 Moves h^Jtk(Ytk)})$.

Using $ \sS' $, we similarly define $ A_2 $, $ v_2 $, and $ u_2 $. Specifically, we define $ A_2 $ to be the $ C^* $-algebra generated by the matrix units 
\[
e_{ i , j }^{ ( t , k )\prime } = \chi_{ h^i ( Y_{ t , k }' ) } u^{ i - j } \chi_{ h^j ( Y_{ t , k }' ) }
\]
for $ t \in \{ 1 , \ldots , T \} $, $ k \in \{ 1 , \ldots , K_t' \} $, and $ i , j \in \{ 0 , \ldots , J_{ t , k }' - 1 \} $, we define 
\[
v_2 = \sum_{ t = 1 }^{ T } \sum_{ k = 1 }^{ K_t' } \left( \chi_{ Y_{ t , k }' } u^{ 1 - J_{ t , k }' } \chi_{ h^{ J_{ t , k }' - 1 } ( Y_{ t , k }' ) } + \sum_{ j = 0 }^{ J_{ t , k }' - 2 } \chi_{ h^{ j + 1 } ( Y_{ t , k }' ) } u \chi_{ h^j ( Y_{ t , k }' ) } \right),
\]
and we define $ u_2 = v_2^* u $. Note that since $ \sP_1 ( \sS' ) $ is finer than $ \sP_1 ( \sS ) $ and since $ X_t' \subset X_t $ for all $ t \in \{ 1 , \ldots , T \} $, we have $ A_1 \subset A_2 $.

Now consider the unitary $ v_2 v_1^* $, which is in $ A_2 $. Before our computations, first note that if $ U $ is a compact open subset of $ X \setminus \bigsqcup_{ t = 1 }^{ T } h^{ -1 } ( X_t' ) $, then 
\begin{equation}\label{v2ActsAsuSometimes}
v_2 \chi_U v_2^* = \chi_{ h ( U ) } . 
\end{equation}
Let $ t \in \{ 1 , \ldots , T \} $.
We have
\begin{align}
v_2 v_1^* \chi_{ Y_{ t , 1 } } v_1 v_2^* &= v_2 \chi_{ h^{ J_{ t , 1 } - 1 } ( Y_{ t , 1 } ) } v_2^* &\text{by (\ref{eq:v1 Moves h^Jtk(Ytk)})} \nonumber \\
&= v_2 \chi_{ h^{ -1 } ( X_t' ) } v_2^*  \nonumber \\
&= \sum_{ k = 1 }^{ K_t' } ( v_2 \chi_{ h^{ J_{ t , k }' - 1 } ( Y_{ t , k }' ) } v_2^* )   \nonumber \\
&= \sum_{ k = 1 }^{ K_t' } \chi_{ Y_{ t , k }' } \nonumber \\
&= \chi_{ X_t' } \nonumber \\
&= \chi_{ h^{ J_{ t , 1 } } ( Y_{ t , 1 } ) } . \label{eqv2v1CommutesWithYt1}
\end{align}
Now let $ k \in \{ 2 , \ldots , K_t \} $. Since $ \sP_1( \sS' ) $ is finer than $ \sP_1 ( \sS ) $ and since $ h^{ J_{ t , 1 } } ( Y_{ t , 1 } ) = X_t' = \bigsqcup_{ l = 1 }^{ K_t' } Y_{ t , l }' $, there is a set 
\[
F_{ t , k } \subset \big\{ ( l , j ) \setdiv \mbox{$ l \in \{ 1 , \ldots , K_t' \} $, and $ j \in \{ 1 , \ldots , J_{ t , l }' - 1 \}$}\big\}
\]
such that 
\[
h^{ J_{ t , k } } ( Y_{ t , k } ) = \bigsqcup_{ ( l , j ) \in F_{ t , k } } h^{ j } ( Y_{ t , l }' ) .
\] 
\begin{align}
v_2 v_1^* \chi_{ Y_{ t , k } } v_1 v_2^* &= v_2 \chi_{ h^{ J_{ t , k } - 1 } ( Y_{ t , k } ) } v_2^* \nonumber \\
&= \sum_{ ( l , j ) \in F_{ t , k } } v_2 \chi_{ h^{ j - 1 } ( Y_{ t , l }'  ) } v_2^* \nonumber \\
&= \sum_{ ( l , j ) \in F_{ t , k } } \chi_{ h^{ j } ( Y_{ t , l }' ) } &\text{by (\ref{v2ActsAsuSometimes})} \nonumber \\
&= \chi_{ h^{ J_{ t , k } } ( Y_{ t , k } ) } . \label{eqv2v1CommutesWithYtk}
\end{align}
In particular, by (\ref{eqv2v1CommutesWithYt1}) and (\ref{eqv2v1CommutesWithYtk}), we see that \begin{equation}\label{eqv2v1CommutesWithXt}
v_2 v_1^* \chi_{ X_t } v_1 v_2^* = \chi_{ X_t } .
\end{equation}

Let $ t \in D $. Recall that $ Y_{ t , 1 }' = Y_{ t , 1 } \cap h^{ J_{ t , 1 } } ( Y_{ t , 1 } ) $ and that $ J_{ t , 1 }' = J_{ t , 1 } $. We therefore have
\begin{align}
v_2 v_1^* \chi_{ Y_{ t , 1 }' } &=  v_2 \chi_{ h^{ J_{ t , 1 } - 1 } ( Y_{ t , 1 }' ) } u^{ J_{ t , 1 } - 1 }  \nonumber \\
&= \chi_{ Y_{ t , 1 }' }. \label{eq:v2v1 identity on Yt1'}
\end{align}

As in Lemma \ref{lemmaTheLastLemma}, for each $ t \in \{ 1 , \ldots , T \} $, we set
\[
\widehat{ X }_t = X_t \setminus ( Y_{ t , 1 } \cap h^{ J_{ t , 1 } } ( Y_{ t , 1 } ) ).
\]
Additionally, we set
\[
Y = \bigsqcup_{ t = 1 }^T \widehat{ X }_t
\]
Recall that $ A_1 \subset A_2 $, and so $ \chi_Y $ and $ v_1 $ are elements of $ A_2 $. Thus, (\ref{eqv2v1CommutesWithXt}) and (\ref{eq:v2v1 identity on Yt1'}) tell us that $ \chi_{ Y } v_2 v_1^* \chi_{ Y } $ is a unitary in $ \chi_{ Y } A_2 \chi_{ Y } $.

Set $ v = \chi_{ Y } v_2 v_1^* \chi_{ Y } $.  Since $ \chi_{ Y } A_2 \chi_{ Y } $ is a finite dimensional $C^*$-algebra, $ v $ has finite spectrum. By Lemma \ref{lemmaFiniteSpectrumLogarithm}, there is a unitary $ w $ in $ \chi_Y A \chi_Y $ with $ w^N = v $ and $ \| w - \chi_Y \| \leq \pi / N < \eps $. Define 
\[
z = \sum_{ j = 0 }^{ N - 1 } \chi_{ h^j ( Y ) } u^j w^{ N - j } u^{ -j } \chi_{ h^j ( Y ) } + \chi_{ X \setminus \bigsqcup_{ j = 0 }^{ N - 1 } h^j ( Y ) } .
\]
It is easy to see that $ z $ is a unitary, since $ z = \sum_{ j = 0 }^N z_j $ for unitaries $ z_j \in \chi_{ h^j ( Y ) } C^* ( \bZ , X , h ) \chi_{ h^j ( Y ) } $ for $ j \in \{ 0 , \ldots , N - 1 \} $ and a unitary $ z_N = \chi_{ X \setminus \bigsqcup_{ j = 0 }^{ N - 1 } h^j ( Y ) } . $ We claim that $ A $, the $ C^* $-algebra generated by $ z A_1 z^* $ and $ z u_2 z^* $, has the desired properties. Specifically, we claim that 
\[
A  \cong \bigoplus_{ t = 1 }^{ T } \left(  \left( C ( S^1 ) \otimes M_{ J_{ t , 1 } } \right)\oplus \left( \bigoplus_{ k = 2 }^{ K_t } M_{ J_{ t , k } } \right) \right) , 
\]
$ A $ contains $ C ( \sP ) $, and $ A $ contains a unitary $ u' $ such that $ \| u' - u \| < \eps $.

First, we want $ C ( \sP ) \subset A $. Because $ \sP_1 ( \sS ) $ is finer than $ \sP $, we have $ C ( \sP ) \subset A_1 $, so all that is left to show is that $ z $ commutes with $ C ( \sP ) $. To see this, let $ U \in \sP $. Since for all $ t \in \{ 1 , \ldots , T \} $ and for all $ n \in \{ 0 , \ldots , N - 1 \} $, $ h^n ( X_t ) $ is contained in an element of $ \sP $, we also see $ h^n ( \widehat{ X }_t ) $ is contained in an element of $ \sP $, and so we can write $ U = \bigsqcup_{ r = 0 }^R U_r $ where $ U_0 \subset X \setminus \bigsqcup_{ j = 0 }^{ N - 1 } h^j ( Y ) $ and for all $ r \in \{ 1 , \ldots , R \} $, there are $ q_r \in \{ 1 , \ldots , T \} $ and $ m_r \in \{ 0 , \ldots , N - 1 \}$ such that $ U_r = h^{ m_r } ( \widehat{ X }_{ q_r } ) $. By (\ref{eqv2v1CommutesWithXt}) and (\ref{eq:v2v1 identity on Yt1'}), we know that $ v $ commutes with $\chi_{ h^{ -m_r } ( U_r ) } $ for every $ r \in \{ 1 , \ldots , R \} $. So by Lemma \ref{lemmaFiniteSpectrumLogarithm}, $ w $ commutes with $ \chi_{ h^{ -m_r } ( U_r ) } $ for all $ r \in \{ 1 , \ldots , R \} $ as well. We now have:
\begin{align*}
\chi_U z &= \left( \sum_{ r = 0 }^R \chi_{ U_r } \right) \left( \sum_{ j = 0 }^{ N - 1 } \sum_{ t = 1 }^T \chi_{ h^j ( Y ) } u^{ j } w^{ N - j } u^{ -j } \chi_{ h^j ( Y ) }  + \chi_{ X \setminus \bigsqcup_{ j = 0 }^{ N - 1 } h^j ( Y ) } \right) \\
&= \sum_{r=1}^R \chi_{ U_r } u^{ m_r } w^{ N - m_r } u^{ -m_r } \chi_{ U_r } + \chi_{ U_0 } .
\end{align*}
A similar computation yields the same thing for $ z \chi_U $. Thus, $ z $ commutes with $ \chi_U $ for all $ U \in \sP $, which shows that $ z $ commutes with $ C ( \sP ) $.

Now, we define $ u' = z v_1 u_2 z^* $, a unitary in $ A $. We still must show that $ \| u' - u \| < \eps $. We have
\begin{align}
\| u' - u \| &= \| z v_1 u_2 z^* - u \|  \nonumber \\
&= \| z v_1 u_2 - u z \|  \label{u'-u}
\end{align}
We will now show that $\| z v_1 u_2 - u z \| < \eps$. 

Let $ E $ be a compact open subset of $ X \setminus h^{ -1 } ( Y ) $. Then $ u_2 \chi_E = u_1 \chi_E $ and so 
\begin{equation}\label{eq:v1u2chiE}
v_1 u_2 \chi_E = u \chi_E .
\end{equation}

Let $ j \in \{ 0 , \ldots , N - 2 \} $. We have
\begin{align*}
\chi_{ h^{ j + 1 } ( Y ) } z v_1 u_2 \chi_{ h^j ( Y ) } &=  \chi_{ h^{ j + 1 } ( Y ) } z u \chi_{ h^j ( Y ) } &\text{by (\ref{eq:v1u2chiE})} \\
&= \chi_{ h^{ j + 1 } ( Y ) } u^{ j + 1 } w^{ N - ( j + 1 ) } u^j \chi_{ h^j ( Y ) }
\end{align*}
and
\begin{align*}
\chi_{ h^{ j + 1 } ( Y ) } u z \chi_{ h^j ( Y ) } &= \chi_{ h^{ j + 1 } ( Y ) } u^{ j + 1 } w^{ N - j } u^j \chi_{ h^j ( Y ) }
\end{align*}
so since $ \| w - \chi_Y \| < \eps $, we have
\begin{equation}\label{eq:hjlessthaneps}
\| \chi_{ h^{ j + 1 } ( Y ) } ( z v_1 u_2 - u z ) \chi_{ h^j ( Y ) } \| < \eps.
\end{equation}
Similarly, we have
\begin{align*}
\chi_{ h^N ( Y ) } z v_1 u_2 \chi_{ h^{ N - 1 } ( Y ) } &=  \chi_{ h^N ( Y ) } z u \chi_{ h^{ N - 1 } ( Y ) } &\text{by (\ref{eq:v1u2chiE})} \\
&=  \chi_{ h^N ( Y ) } u \chi_{ h^{ N - 1 } ( Y ) }
\end{align*}
and 
\begin{align*}
\chi_{ h^N ( Y ) } u z \chi_{ h^j ( Y ) } &= \chi_{ h^N ( Y ) } u^N w u^{ N - 1 } \chi_{ h^{ N - 1 } ( Y ) }
\end{align*}
and again since $ \| w - \chi_Y \| < \eps $, we have
\begin{equation}\label{eq:hNlessthaneps}
\| \chi_{ h^N ( Y ) } ( z v_1 u_2 - u z ) \chi_{ h^{ N - 1 } ( Y ) } \| < \eps.
\end{equation}

Let $ j \in \{ 0 , \ldots , N - 1 \} $ and let $ p $ be any projection orthogonal to $ h^{ j + 1 } ( Y ) $. Then
\begin{align*}
p z v_1 u_2 \chi_{ h^j ( Y ) } &= p z u \chi_{ h^j ( Y ) } &\text{by (\ref{eq:v1u2chiE})} \\
&= p z \chi_{ h^{ j + 1 } ( Y ) } u \\
&= p \chi_{ h^{ j + 1 } ( Y ) } u^{ j + 1 } w^{ N - ( j + 1 ) } u^{ - ( j + 1 ) } \chi_{ h^{ j + 1 } ( Y ) } u \\
&= 0
\end{align*}
and
\begin{align*}
p u z \chi_{ h^j ( Y ) } &= p u \chi_{ h^j ( Y ) } u^j w^{ N - j } u^{ -j } \chi_{ h^j ( Y ) } \\
&= p \chi_{ h^{ j + 1 } ( Y ) } u^{ j + 1 } w^{ N - j } u^{ -j } \chi_{ h^j ( Y ) } \\
&= 0.
\end{align*}
Thus, we have
\begin{equation}\label{eq:orthogtochij}
p ( z v_1 u_2 - u z ) \chi_{ h^j ( Y ) } = 0 
\end{equation}

Now observe
\begin{align*}
z v_1 u_2 \chi_{ X \setminus \bigsqcup_{ j = -1 }^{ N - 1 } h^j ( Y ) } &= z u \chi_{ X \setminus \bigsqcup_{ j = -1 }^{ N - 1 } h^j ( Y ) } &\text{by (\ref{eq:v1u2chiE})} \\
&= z \chi_{ X \setminus \bigsqcup_{ j = 0 }^{ N } h^j ( Y ) } u \\
&= \chi_{ X \setminus \bigsqcup_{ j = 0 }^{ N } h^j ( Y ) } u 
\end{align*}
and
\begin{align*}
u z \chi_{ X \setminus \bigsqcup_{ j = -1 }^{ N - 1 } h^j ( Y ) } &= u \chi_{ X \setminus \bigsqcup_{ j = -1 }^{ N - 1 } h^j ( Y ) } \\
&= \chi_{ X \setminus \bigsqcup_{ j = 0 }^{ N } h^j ( Y ) } u 
\end{align*}
and therefore 
\begin{equation}\label{eq:Xsetminus-1toN}
( z v_1 u_2 - u z )\chi_{ X \setminus \bigsqcup_{ j = -1 }^{ N - 1 } h^j ( Y ) } = 0 .
\end{equation}

Let $ t \in \{ 1 , \ldots , T \} $. Let $ \widehat{ Y }_t = Y_{ t , 1 } \cap \left( \bigsqcup_{ k = 2 }^{ K_t } h^{ J_{ t , k } } ( Y_{ t , k } ) \right) $ and let $ \widehat{ Y }_t' = h^{ J_{ t , 1 } } ( Y_{ t , 1 } ) \cap \left( \bigsqcup_{ k = 2 }^{ K_t } Y_{ t , k } \right) $. Notice that $ h^{ J_{ t , 1 } - 1 } ( \widehat{ Y }_t ) = \bigsqcup_{ k = 2 }^{ K_t' } h^{ J_{ t , k }' - 1 } ( Y_{ t , k }' ) $ and $ h^{ -1 } ( \widehat{ Y }_t' ) = \bigsqcup_{ k = 2 }^{ K_t' } h^{ -1 } ( Y_{ t , k }' ) $, and so $ u_2 \chi_{ h^{ -1 } ( \widehat{ Y }_t' ) } = \chi_{ h^{ J_{ t , 1 } - 1 } ( \widehat{ Y }_t ) } u_2 $. 
\begin{align*}
z v_1 u_2 \chi_{ h^{ - 1 } ( \widehat{ Y }_t' ) } &= z v_1 \chi_{ h^{ J_{ t , 1 } - 1 }  ( \widehat{ Y }_t ) } u_2  \\
&= z \chi_{ \widehat{ Y }_t } v_1 u_2  \\
&= v_2 v_1^* \chi_{ \widehat{ Y }_t } v_1 u_2  \\
&= v_2 \chi_{ h^{ J_{ t , 1 } - 1 } ( \widehat{ Y }_t ) } u_2  \\
&= v_2 u_2 \chi_{ h^{ - 1 } ( \widehat{ Y }_t' ) }  \\
&= u \chi_{ h^{ - 1 } ( \widehat{ Y }_t' ) }  \\
&= u z \chi_{ h^{ - 1 } ( \widehat{ Y }_t' ) }
\end{align*}
and so
\begin{equation}\label{eq:h-1widehatYt'}
( z v_1 u_2 - u z ) \chi_{ h^{ - 1 } ( \widehat{ Y }_t' ) } = 0
\end{equation}

Let $ t \in \{ 1 , \ldots , T \} $, let $ k \in \{ 1 , \ldots , K_t \} $, and let $ l \in \{ 2 , \ldots , K_t \} $. Define $ Y_{ t , k , l } = Y_{ t , k } \cap h^{ J_{ t , l } } ( Y_{ t , l } ) $. We have
\begin{align*}
z v_1 u_2 \chi_{ h^{ - 1 } ( Y_{ t , k , l } ) } &= z v_1 \chi_{ h^{ - 1 } ( Y_{ t , k , l } ) } u_2 \\
&= z \chi_{ h^{ - J_{ t , l } } ( Y_{ t , k , l } ) } v_1 u_2 \\
&= v_2 v_1^* \chi_{ h^{ - J_{ t , l } } ( Y_{ t , k , l } ) } v_1 u_2 \\
&= v_2 \chi_{ h^{ - 1 } ( Y_{ t , k , l } ) } u_2 \\
&= v_2 u_2 \chi_{ h^{ - 1 } ( Y_{ t , k , l } ) } \\
&= u \chi_{ h^{ - 1 } ( Y_{ t , k , l } ) } \\
&= u z \chi_{ h^{ - 1 } ( Y_{ t , k , l } ) }
\end{align*}
Combining the above with (\ref{eq:h-1widehatYt'}), we get
\begin{equation}\label{eq:h-1Y}
( z v_1 u_2 - u z ) \chi_{ h^{ - 1 } ( Y ) } = 0
\end{equation}

We now apply Lemma \ref{lemmaChrisProjectionCutdownLemma} with $ M = N + 3 $, $ a = z v_1 u_2 - u z $, $ p_n = \chi_{ h^n ( Y ) } $ for all $ n \in \{ 1 , \ldots , N \}$, $ q_n = \chi_{ h^{ n - 1 } ( Y ) } $ for all $ n \in \{ 1 , \ldots , N + 1 \} $, $ p_{ N + 1 } = \chi_Y $, $ p_{ N + 2 } = q_{ N + 2 } = \chi_{ h^{ -1 } ( Y ) } $, and $ p_{ N + 3 } = q_{ N + 3 } = \chi_{ X \setminus \bigsqcup_{ j = -1 }^{ N - 1 } h^j ( Y ) }$. By (\ref{eq:hjlessthaneps}), we have $ \| p_n a q_n \| < \eps $ for all $ n \in \{ 1 , \ldots , N - 1 \} $. By (\ref{eq:hNlessthaneps}), we have $ \| p_N a q_N \| < \eps $. By (\ref{eq:orthogtochij}), for $ n \in \{ 1 , \ldots , N + 1 \} $, we have $ p_m a q_n = 0 $ for all $ m \in \{ 1 , \ldots , M \}$ such that $ m \neq n $. By (\ref{eq:h-1Y}), we have $ p_m a q_{ N + 2 } = 0 $ for all $ m \in \{ 1 , \ldots , M \} $.  By (\ref{eq:Xsetminus-1toN}), we have $ p_m a q_{ N + 3 } = 0 $ for all $ m \in \{ 1 , \ldots , M \} $. Thus, by Lemma \ref{lemmaChrisProjectionCutdownLemma}, we have $ \| z v_1 u_2 - u z \| < \eps $. By (\ref{u'-u}), we have $ \| u' - u \| < \eps $ as desired.

We will now show that 
\[
A  \cong \bigoplus_{ t = 1 }^{ T } \left(  \left( C ( S^1 ) \otimes M_{ J_{ t , 1 } } \right)\oplus \left( \bigoplus_{ k = 2 }^{ K_t } M_{ J_{ t , k } } \right) \right) , 
\]
The $ C^* $-algebra $ \widehat{ A } $ generated by $ A_1 $ and $ u_2 $ is unitarily equivalent to $ A $ (via $ z $). Thus, $ A \cong \widehat{ A } $, and we will therefore work with $ \widehat{ A } $ for the remainder of the proof. For convenience of notation during the rest of the proof, set 
\[
\widehat { p } = \chi_{ X \setminus \bigsqcup_{ t = 1 }^T \bigsqcup_{ j = 0 }^{ J_{ t , 1 } - 1 } h^j ( Y_{ t , 1 } ) } .
\]
Let $$ \widehat{u} = \sum_{ t = 1 }^{ T' } \sum_{ j = 0 }^{ J_{ t , 1 } - 1 } e_{ j , J_{ t , 1 } - 1 }^{ ( t , 1 )} u_2 e_{ J_{ t , 1 } - 1 , j }^{ ( t , 1 ) } + \widehat { p } . $$ We claim $ \widehat{ u } $ is a unitary in $ \widehat{ A } $. First, for $ t \in \{ 1 , \ldots , T \} $, and $ j \in \{ 0 , \ldots , J_{ t , 1 } - 1 \}$, observe that
\begin{align}
u_2e_{ J_{ t , 1 } - 1, j }^{ ( t , 1 )} e_{ j, J_{ t , 1 } - 1 }^{ ( t , 1 )} u_2^* &= u_2 \chi_{ h^{ -1 } ( X_t' ) } u_2^* \nonumber \\
&= v_2^* \chi_{ X_t' } v_2 \nonumber \\
&= \chi_{ h^{ -1 }( X_t' ) } \nonumber \\
&= \chi_{ h^{ J_{ t , 1 } - 1 }( Y_{ t , 1 } ) } . \label{eqWidehatUIsUnitary}
\end{align}
Now, observe that
\begin{align*}
\widehat{ u } \widehat{ u }^* &= \sum_{ t = 1 }^{ T } \sum_{ j = 0 }^{ J_{ t , 1 } - 1 } e_{ j , J_{ t , 1 } - 1 }^{ ( t , 1 ) } \left( u_2 e_{ J_{ t , 1 } - 1, j }^{ ( t , 1 ) } e_{ j , J_{ t , 1 } - 1 }^{ ( t , 1 ) } u_2^* \right) e_{ J_{ t , 1 } - 1, j }^{ ( t , 1 ) } + \widehat { p } \\
&= \sum_{ t = 1 }^{ T } \sum_{ j = 0 }^{ J_{ t , 1 } - 1 } e_{ j , J_{ t , 1 } - 1 }^{ ( t , 1 ) } \chi_{ h^{ J_{ t , 1 } - 1 }(Y_{ t , 1 } ) } e_{ J_{ t , 1 } - 1, j }^{ ( t , 1 ) } + \widehat { p } &\text{by (\ref{eqWidehatUIsUnitary})}\\
&= \sum_{ t = 1 }^{ T } \sum_{ j = 0 }^{ J_{ t , 1 } - 1 } \chi_{ h^j ( Y_{ t , 1 } ) } + \widehat { p } \\
&= 1 .
\end{align*}
A similar computation shows $ \widehat{ u }^* \widehat{ u } = 1 $. Thus, $ \widehat{ u } $ is a unitary.

We claim that $ A_1 $ and $ \widehat{ u } $ commute. To see this, it is clear that we only need to check commutativity with matrix units of the form $ e_{ i , j }^{ ( t , 1 ) } $ for $ t \in \{ 1 , \ldots , T \} $. But with this in mind, we have
\begin{align*}
\widehat{ u } e_{ i , j }^{ ( t , 1 ) } \widehat{ u }^* &= e_{ i , J_{ t , 1 } - 1 }^{ ( t , 1 ) } u_2 e_{ J_{ t , 1 } - 1 , J_{ t , 1 } - 1 }^{ ( t , 1 ) } u_2^* e_{ J_{ t , 1 } - 1 , j }^{ ( t , 1 )} \\
&= e_{ i , J_{ t , 1 } - 1 }^{ ( t , 1 ) } e_{ J_{ t , 1 } - 1 , J_{ t , 1 } - 1 }^{ ( t , 1 ) } e_{ J_{ t , 1 } - 1 , j }^{ ( t , 1 ) } &\text{by (\ref{eqWidehatUIsUnitary})} \\
&= e_{ i , j }^{ ( t , 1 ) } .
\end{align*}
Thus, $ A_1 $ and $ \widehat{ u } $ commute.

We claim that $ A_1 $ and $ \widehat{ u } $ generate $ \widehat{ A } $. To see this, notice that
\[
\left( \sum_{ t = 1 }^T  e_{ J_{ t , 1 } - 1 , 0 }^{ ( t , 1 ) } \right) \widehat{ u } \left( \sum_{ t = 1 }^T  e_{ 0 , J_{ t , 1 } - 1 }^{ ( t , 1 ) } \right) = \sum_{ t = 1 }^T e_{ J_{ t , 1 } - 1 , J_{ t , 1 } - 1 }^{ ( t , 1 ) } u_2 e_{ J_{ t , 1 } - 1 , J_{ t , 1 } - 1 }^{ ( t , 1 ) } ,
\]
which, when added to
\begin{align*}
\chi_{ X \setminus \bigsqcup_{ t = 1 }^T  h^{ J_{ t , 1 } - 1 } ( Y_{ t , k } ) } &= \chi_{ X \setminus \bigsqcup_{ t = 1 }^T h^{ -1 } ( X_t' ) } \\
&= \sum_{ t = 1 }^T \sum_{ k = 1 }^{ K_t' } \sum_{ j = 0 }^{ J_{ t , k }' - 2 } \chi_{ h^j ( Y_{ t , k }' ) }
\end{align*}
yields $ u_2 $.

Let $ t \in \{ 1 , \ldots , T \} $, let $ k \in \{ 2 , \ldots , K_t \} $, and let $ i , j \in \{ 0 , \ldots , J_{ t , k } - 1 \} $. We have
\begin{align*}
e_{ i , j }^{ ( t , k ) } \widehat{ u } &= \chi_{ h^i ( Y_{ t , k } ) } u^{ j - i } \chi_{ h^j ( Y_{ t , k } ) } \widehat{ p } \\
&= \chi_{ h^i ( Y_{ t , k } ) } u^{ j - i } \chi_{ h^j ( Y_{ t , k } ) } \\
&= e_{ i , j }^{ ( t , k ) }
\end{align*}
and similarly $ \widehat{ u } e_{ i , j }^{ ( t , k ) } = e_{ i , j }^{ ( t , k ) } $. Thus, setting $ p_{ t , k } = \sum_{ i = 0 }^{ J_{ t , k } - 1 } e_{ i , i }^{ ( t , k ) } $, we have 
\begin{equation}\label{eqFinalIsom2} 
p_{ t , k } \widehat{ A } p_{ t , k } \cong M_{ J_{ t , k } } . 
\end{equation}

Fix $ t \in \{ 1 , \ldots , T \} $ and set $ p_t = \sum_{ j = 0 }^{ J_{ t , 1 } - 1 } \chi_{ h^j ( Y_{ t , 1 } ) } = \sum_{ j = 0 }^{ J_{ t , 1 } - 1 } e_{ j , j }^{ ( t , 1 ) } $. We now claim that $ p_t u_2 p_t $ and $ p_t \widehat{ u } p_t $ are unitaries in $ p_t C^*( \bZ , X , h ) p_t $. To show this, we show that $ p_t $ commutes with $ u_2 $ and $ \widehat{ u } $. 

It is obvious that $ p_t $ commutes with $ \widehat{ u } $, since $ \widehat{ u } $ commutes with $ A_1 $. So to show that $ p_t $ commutes with $ u_2 $, we first claim that, for each $ j \in \{ -J_{ t , 1 } , \ldots , -2 \} $, we have
\begin{equation} \label{eq: h^j ( X_s' ) Is A Subset Of Something}
h^j ( X_t' ) \subset \bigsqcup_{ k = 1 }^{ K_t' } \bigsqcup_{ j' = 0 }^{ J_{ t , k }' - 2 } h^{ j' }( Y_{ t , k }' ) .
\end{equation}
To see this, note that 
\[
\bigsqcup_{ k = 1 }^{ K_s' } \bigsqcup_{ j' = 0 }^{ J_{ t , k }' - 2 } h^{ j' } ( Y_{ t , k }' ) = \bigsqcup_{ k = 1 }^{ K_s' } \bigsqcup_{ j' = 0 }^{ J_{ t , k }' - 1 } h^{ j' } ( Y_{ t , k }' ) \setminus h^{ -1 } ( X_t' ) = \bigcup_{ j' \in \bZ } h^j( X_t' ) \setminus h^{ -1 }( X_t' ) ,
\]
and then note that $ Y_{ t , 1 } = h^{ -J_{ t , 1 } } ( X_t' ) $, and so since $ Y_{ t , 1 } , h ( Y_{ t , 1 } ) , \ldots , h^{ J_{ t , 1 } - 1 } ( Y_{ t , 1 } ) $ are pairwise disjoint, it follows that $ h^j ( X_t' ) \cap h^{ -1 } ( X_t' ) = \varnothing $ for all $ j \in  \{ -J_{ t , 1 } , \ldots , -2 \} $. Thus, the claim follows. Now, 
\begin{align*}
u_2 p_t &= \left( \sum_{ s = 1 }^{ T' } \sum_{ k = 1 }^{ K_t' } \left( \chi_{ h^{ J_{ s , k }' - 1 }( Y_{ s , k }' ) } u^{ J_{ s , k }' } \chi_{ h^{ -1 } ( Y_{ s , k }' ) } + \sum_{ j = 0 }^{ J_{ s , k }' - 2 } \chi_{ h^j ( Y_{ s , k }' ) } \right) \right)  \left( \sum_{ j = 0 }^{ J_{ t , 1 } - 1 } \chi_{ h^j ( Y_{ t , 1 } ) } \right) \\
&= \left( \sum_{ s = 1 }^{ T' } \sum_{ k = 1 }^{ K_s' } \left( \chi_{ h^{ J_{ s , k }' - 1 } ( Y_{ s , k }' ) } u^{ J_{ s , k }' } \chi_{ h^{ -1 } ( Y_{ s , k }' ) } + \sum_{ j = 0 }^{ J_{ s , k }' - 2 } \chi_{ h^j ( Y_{ s , k }' ) } \right) \right)  \left( \sum_{ j = -J_{ t , 1 } }^{ -1 } \chi_{ h^j ( X_t' ) } \right) \\
&= \left( \sum_{ k = 1 }^{ K_t' } \chi_{ h^{ J_{ t , k }' - 1 } ( Y_{ t , k }' ) } u^{ J_{ t , k }' } \chi_{ h^{ -1 } ( Y_{ t , k }' ) } \right) \chi_{ h^{ -1 } ( X_s' ) } + \sum_{ j = -J_{ t , 1 } }^{ -2 } \chi_{ h^j ( X_t' ) } &\text{by (\ref{eq: h^j ( X_s' ) Is A Subset Of Something})} \\
&= \left( \sum_{ k = 1 }^{ K_t' } \chi_{ h^{ J_{ t , k }' - 1 } ( Y_{ t , k }' ) } u^{ J_{ t , k }' } \chi_{ h^{ -1 } ( Y_{ t , k }' ) } \right) \sum_{ k = 1 }^{ K_t' } \chi_{ h^{ -1 } ( Y_{ t , k }' ) } + \sum_{ j = -J_{ t , 1 } }^{ -2 } \chi_{ h^j ( X_t' ) } \\
&= \sum_{ k = 1 }^{ K_t' } \chi_{ h^{ J_{ t , k }' - 1 } ( Y_{ t , k }' ) } u^{ J_{ t , k }' } \chi_{ h^{ -1 } ( Y_{ t , k }' ) } + \sum_{ j = -J_{ t , 1 } }^{ -2 } \chi_{ h^j ( X_t' ) } .
\end{align*}
Similarly, 
\begin{align*}
p_t u_2 &= \sum_{ k = 1 }^{ K_s' } \chi_{ h^{ J_{ s , k }' - 1 } ( Y_{ s , k }' ) } u^{ J_{ s , k }' } \chi_{ h^{ -1 } ( Y_{ s , k }' ) } + \sum_{ j = -J_{ t , 1 } }^{ -2 } \chi_{ h^j ( X_s' ) } .
\end{align*}
Thus, $ p_t $ commutes with $ u_t $, and so $ p_t u_2 p_t $ is a unitary in $ p_t C^*( \bZ , X , h ) p_t $. 

Now notice that $ \bigcup_{ j \in \bZ } h^j ( X_t' ) = \bigsqcup_{ k = 1 }^{ K_t' } \bigsqcup_{ j = 0 }^{ J_{ t , k }' - 1 } h^j ( Y_{ t , k }' ) $ is an $ h $-invariant compact open subset of $ X $. Set $ r_t = \chi_{ \bigcup_{ j \in \bZ } h^j ( X_t' ) } $, a projection that therefore commutes with $ u $, which means that it is central in $ C^* ( \bZ , X , h ) $. We claim that $ [ r_t \widehat{ u } r_t ] = J_{ t , 1 } [ r_t u_2 r_t ]$ in $ K_1 ( r_t C^* ( \bZ , X , h ) r_t ) $. For each $ j \in \{ 0, \ldots,  J_{ t , 1 } - 1  \} $, set $ D_j = \{ 0, \ldots,  J_{ t , 1 } - 1  \} \setminus \{ j \}$ and set 
\[
w_j = e_{ j,  J_{ t , 1 } - 1 }^{ ( t , 1 ) }  u_2 e_{ J_{ t , 1 } - 1 , j }^{ ( t , 1 ) } + \sum_{ i \in D_j } e_{ i , i } + ( r_t - p_t ) .
\]
Note that $ w_{ J_{ t , 1 } - 1 } = r_t u_2 r_t $ . We have
\begin{align} 
\prod_{ j = 0 }^{ J_{ t , 1 } - 1 } w_j &=  \prod_{ j = 0 }^{ J_{ t , 1 } - 1 } \left( e_{ j,  J_{ t , 1 } - 1 }^{ ( t , 1 )} u_2 e_{  J_{ t , 1 } - 1 , j }^{ ( t , 1 )} + \sum_{ i \in D_j } e_{ i, i }^{ ( t , 1 )} + ( r_s - p_s ) \right) \nonumber \\
&=  \sum_{ j = 0 }^{ J_{ t , 1 } - 1 }  e_{ j,  J_{ t , 1 } - 1  }^{ ( t , 1 )} u_2 e_{  J_{ t , 1 } - 1 , j }^{ ( t , 1 )} + ( r_t - p_t ) \nonumber \\
&= r_t \widehat{ u } r_t . \label{eqProductOfUnitaries}
\end{align}
Let $ j \in \{ 0 , \ldots ,  J_{ t , 1 } - 2 \}$. Define $ D_j' = \{ 0, \ldots,  J_{ t , 1 } - 2 \} \setminus \{ j \} $, 
\[
\widehat{ p }_j = \sum_{ i \in D_j' } e_{ i ,i }^{ ( t , 1 ) } + ( r_t - p_t ) ,
\] and 
\[
w_j' = e_{ j , J_{ t , 1 } - 1 }^{ ( t , 1 ) } + e_{ J_{ t , 1 } - 1 , j }^{ ( t , 1 ) } + \widehat{ p }_j .
\]
We have
\[ w_j' r_t u_2 r_t = \left( e_{ j , J_{ t , 1 } - 1 }^{ ( t , 1 )} + e_{ J_{ t , 1 } - 1 , j }^{ ( t , 1 )} + \widehat{ p }_j  \right) \left(\sum_{ k = 1 }^{ K_t' } \left( \chi_{ h^{ J_{ t , k }' - 1 } ( Y_{ t , k }' ) } u^{ J_{ t , k }' } \chi_{ h^{ -1 } ( Y_{ t , k }' ) } + \sum_{ j = 0 }^{ J_{ t , k }' - 2 } \chi_{ h^j ( Y_{ t , k }' ) } \right) \right) .
\]
We simplify the right hand side of the above equation by breaking the simplification into a few steps. First, recall that $ h^{ -1 } ( X_t' )  = h^{ J_{ t , 1 } - 1 } ( Y_{ t , 1 } ) $. With this in mind, we have
\begin{align}
e_{ j , J_{ t , 1 } - 1 }^{ ( t , 1 ) } r_t u_2 r_t &= \sum_{ k = 1 }^{ K_t' } \chi_{ h^j ( Y_{ t , 1 } ) } u^{ j - ( J_{ t , 1 } - 1 ) + J_{ t , k }' } \chi_{ h^{ -1 } ( Y_{ t , k }' ) } \nonumber \\
&= \sum_{ k = 1 }^{ K_t' }\chi_{ h^{ j - J_{ t , 1 } } ( X_t' ) } u^{ j - ( J_{ t , 1 } - 1 ) + J_{ t , k }' } \chi_{ h^{ -1 } ( Y_{ t , k }') } \nonumber \\
&= \sum_{ k = 1 }^{ K_t' } \chi_{ h^{ j - J_{ t , 1 } + J_{ t , k }' } ( Y_{ t , k }' ) }  u^{ j - ( J_{ t , 1 } - 1 ) + J_{ t , k }' } \chi_{ h^{ -1 } ( Y_{ t , k }' ) } . \label{eqU2CompLHS1}
\end{align}
Since 
\[
h^j ( Y_{ t , 1 } ) \subset \bigsqcup_{ k = 1 }^{  K_t'  } \bigsqcup_{ j' = 0 }^{ J_{ t , k }' - 2 } h^{ j' } ( Y_{ t , k }' ) ,
\]
we have 
\begin{align}
e_{ J_{ t , 1 } - 1 , j }^{ ( t , 1 ) } r_t u_2 r_t &= \chi_{ h^{ J_{ t , 1 } - 1 }( Y_{ t , 1 } ) } u^{ J_{ t , 1 } - 1 - j } \chi_{ h^j ( Y_{ t , 1 } ) } \left( \sum_{ k = 1 }^{ K_t' } \sum_{ j = 0 }^{ J_{ t , k }' - 2 } \chi_{ h^j ( Y_{ t , k }' ) } \right) \nonumber \\
&=  \chi_{h^{ J_{ t , 1 } - 1 }(Y_{ t , 1 })} u^{ J_{ t , 1 } - 1 - j } \chi_{h^j(Y_{ t , 1 })}  \nonumber \\
&= e_{ J_{ t , 1 } - 1 , j }^{ ( t , 1 )}. \label{eqU2CompLHS2}
\end{align}
Note that $\widehat{ p }_j = \chi_{ \widehat{ E }_j } $ where 
\begin{align}
\widehat{ E }_j &= \left( \bigsqcup_{i \in D_j'} h^i(Y_{ t , 1 }) \right) \sqcup \left( \left( \bigsqcup_{ k = 1 }^{ K_t' } \bigsqcup_{i=0}^{J_{ t , k }'-1} h^i(Y_{ t , k }') \right) \setminus \left( \bigsqcup_{j=0}^{ J_{ t , 1 } - 1 } h^j(Y_{ t , 1 }) \right) \right) \nonumber \\
&=  \left( \bigsqcup_{ k = 1 }^{ K_t' } \bigsqcup_{i=0}^{J_{ t , k }'-1} h^i(Y_{ t , k }') \right) \setminus \left( h^{ J_{ t , 1 } - 1 }(Y_{ t , 1 }) \sqcup h^j(Y_{ t , 1 }) \right) \nonumber \\
&=  \left( \bigsqcup_{ k = 1 }^{ K_t' } \bigsqcup_{i=0}^{J_{ t , k }'-1} h^i(Y_{ t , k }') \right) \setminus \left( h^{-1}( X_t' ) \sqcup h^j(Y_{ t , 1 }) \right) \nonumber \\
&=  \left( \bigsqcup_{ k = 1 }^{ K_t' } \bigsqcup_{i=0}^{J_{ t , k }'-2} h^i(Y_{ t , k }') \right) \setminus h^j(Y_{ t , 1 }). \label{eqEjComputation}
\end{align}
Thus, $\widehat{ E }_j \subset \bigsqcup_{ k = 1 }^{ K_t' } \bigsqcup_{ j = 0 }^{J_{ t , k }' - 2 } h^j ( Y_{ t , k }' ) $, and so we have 
\begin{align}
\widehat{ p }_j r_t u_2 r_t &= \widehat{ p }_j \left(  \sum_{ k = 1 }^{ K_t' } \sum_{j=0}^{ J_{ t , k }' - 2 } \chi_{ h^j ( Y_{ t , k }' ) } \right)  \nonumber \\
 &= \widehat{ p }_j . \label{eqU2CompLHS3}
\end{align}
From (\ref{eqU2CompLHS1}), (\ref{eqU2CompLHS2}), and (\ref{eqU2CompLHS3}), we have 
\[
w_j' r_t u_2 r_t =  \sum_{ k = 1 }^{  K_t'  }\chi_{h^{ j - J_{ t , 1 } + J_{ t , k }' } ( Y_{ t , k }' ) } u^{ j -(J_{ t , 1 } - 1) + J_{ t , k }' } \chi_{ h^{ -1 } ( Y_{ t , k }' ) } + e_{ J_{ t , 1 } - 1 , j }^{ ( t , 1 ) } + \widehat{ p }_j .
\]
Now, we compute the following: 
\[
w_j' r_t u_2 r_t w_j' = \left( \sum_{ k = 1 }^{ K_t' } u^{ j - ( J_{ t , 1 } - 1 ) +J_{ t , k }'} \chi_{h^{-1}(Y_{ t , k }')} + e_{ J_{ t , 1 } - 1 , j }^{ ( t , 1 )} + \widehat{p}_j\right) \left(e_{ j , J_{ t , 1 } - 1 }^{ ( t , 1 )} + e_{ J_{ t , 1 } - 1 , j }^{ ( t , 1 )} + \widehat{ p }_j \right). 
\]
From (\ref{eqEjComputation}), it is clear that 
\[\widehat{ p }_j e_{ j , J_{ t , 1 } - 1 }^{ ( t , 1 )} =0.
\]
Since $ j \neq  J_{ t , 1 } - 1 $, we have $ h^j ( Y_{ t , 1 } ) \cap h^{ -1 } ( X_t' ) = \varnothing $, so 
\[
\sum_{ k = 1 }^{  K_t'  } \chi_{ h^{ j - J_{ t , 1 } + J_{ t , k }' } ( Y_{ t , k }' ) } u^{ j - ( J_{ t , 1 } - 1) + J_{ t , k }' } \chi_{ h^{ -1 } ( Y_{ t , k }' ) } e_{ j , J_{ t , 1 } - 1 }^{ ( t , 1 )} = 0 .
\]
Thus, we have
\begin{align}
w_j' r_t u_2 r_t e_{ j , J_{ t , 1 } - 1 }^{ ( t , 1 )} &= e_{ J_{ t , 1 } - 1 , j }^{ ( t , 1 )} e_{ j ,  J_{ t , 1 } - 1 }^{ ( t , 1 )} \nonumber \\
&= e_{ j , j } . \label{eqU2CompRHS1}
\end{align}
Since $ J_{ t , 1 } - 1 \notin D_j $, and from (\ref{eqEjComputation}), we can easily see that 
\[
\left( e_{ J_{ t , 1 } - 1 ,j }^{ ( t , 1 )} + \widehat{ p }_j \right) e_{ J_{ t , 1 } - 1 , j }^{ ( t , 1 )} = 0 .
\]
Thus, we have
\begin{align*}
w_j' r_t u_2 r_t  e_{ J_{ t , 1 } - 1 , j }^{ ( t , 1 )} &= \sum_{ k = 1 }^{ K_s' } \chi_{ h^{ j - J_{ t , 1 } + J_{ t , k }' } ( Y_{ t , k }' ) } u^{ j - ( J_{ t , 1 } - 1 ) + J_{ t , k }' } \chi_{ h^{ -1 } ( Y_{ t , k }' ) }  u^{ J_{ t , 1 } - 1 - j } \chi_{ h^j ( Y_{ t , 1 } ) } \nonumber \\
&= \sum_{ k = 1 }^{ K_s' } \chi_{ h^{ j - J_{ t , 1 } + J_{ t , k }' } ( Y_{ t , k }' ) } u^{ j - ( J_{ t , 1 } - 1 ) + J_{ t , k }' } \chi_{ h^{ -1 } ( Y_{ t , k }' ) } u^{ J_{ t , 1 } - 1 - j } \chi_{ h^{ j - J_{ t , 1 } }( Y_{ t , k }' ) } \nonumber \\
&= \sum_{ k = 1 }^{ K_s' } \chi_{ h^{ j - J_{ t , 1 } + J_{ t , k }' } ( Y_{ t , k }' ) } u^{ J_{ t , k }' } \chi_{ h^{ j - J_{ t , 1 } }( Y_{ t , k }' ) } . 
\end{align*}
But then notice that 
\begin{align*}
e_{ j , J_{ t , 1 } - 1 }^{ ( t , 1 )} u_2 e_{ J_{ t , 1 } - 1 , j }^{ ( t , 1 )} &= \chi_{ h^j ( Y_{ t , 1 } ) } u^{ J_{ t , 1 } - 1 - j } \chi_{ h^{ J_{ t , 1 } - 1 }  ( Y_{ t , 1 } ) } u_2 \chi_{ h^{ J_{ t , 1 } - 1 } ( Y_{ t , 1 } ) } u^{ j - ( J_{ t , 1 } - 1 ) } \chi_{ h^j ( Y_{ t , 1 } ) } \\
&= \sum_{ k = 1 }^{ K_s' } \chi_{ h^{ j - J_{ t , 1 } + J_{ t , k }' } ( Y_{ t , k }' ) } u^{ J_{ t , 1 } - 1 - j } \chi_{ h^{ J_{ t , k }' - 1 } ( Y_{ t , k }' ) } u^{ J_{ t , k }' } \chi_{ h^{ -1 } ( Y_{ t , k }' ) } u^{ j - ( J_{ t , 1 } - 1 ) } \chi_{ h^{ j - J_{ t , 1 } } ( Y_{ t , k }' ) } \\
&= \sum_{ k = 1 }^{ K_s' } \chi_{ h^{ j - J_{ t , 1 } + J_{ t , k }' } ( Y_{ t , k }' ) } u^{ j - ( J_{ t , 1 } - 1 ) + J_{ t , k }' } \chi_{ h^{ j - J_{ t , 1 } } ( Y_{ t , k }' ) } .
\end{align*}
Thus, 
\begin{equation}\label{eqU2CompRHS2}
w_j' r_t u_2 r_t  e_{ J_{ t , 1 } - 1 , j }^{ ( t , 1 )} = e_{ j , J_{ t , 1 } - 1 }^{ ( t , 1 )}  u_2 e_{ J_{ t , 1 } - 1  , j }^{ ( t , 1 )} .
\end{equation}
Finally, it is immediately clear that 
\begin{align}
w_j' r_t u_2 r_t \widehat{ p }_j &= \widehat{ p }_j \widehat{ p }_j  \nonumber \\
&= \widehat{ p }_j. \label{eqU2CompRHS3}
\end{align}
So by (\ref{eqU2CompRHS1}), (\ref{eqU2CompRHS2}), and  (\ref{eqU2CompRHS3}), we see
\begin{align*}
w_j' r_t u_2 r_t w_j' &= e_{ j , j }^{ ( t , 1 )} + \sum_{ k = 1 }^{ K_s' } \chi_{ h^{ j - J_{ t , 1 } + J_{ t , k }' } ( Y_{ t , k }' ) } u^{ J_{ t , k }' } \chi_{ h^{ j - J_{ t , 1 } } ( Y_{ t , k }' ) } + \widehat{ p }_j \\
&= \sum_{ k = 1 }^{ K_s' } \chi_{ h^{ j - J_{ t , 1 } + J_{ t , k }' } ( Y_{ t , k }' ) } u^{ J_{ t , k }' } \chi_{ h^{ j - J_{ t , 1 } } ( Y_{ t , k }' ) } + \sum_{ i \in D_j } e_{ i , i }^{ ( t , 1 )} + ( r_t - p_t ) \\
&= e_{ j , J_{ t , 1 } - 1 } u_2 e_{ J_{ t , 1 } - 1 , j } + \sum_{ i \in D_j } e_{ i , i }^{ ( t , 1 )} +  ( r_t - p_t ) \\
&= w_j .
\end{align*}
Now note that $ w_j' $ is a unitary in $ r_t A_2 r_t $, and since $ r_t A_2 r_t $ is a finite-dimensional $ C^* $-subalgebra of $ r_t C^* ( \bZ , X , h ) r_t $, $ w_j' $ has trivial $ K_1 $-class. Thus, in $ K_1 ( r_t C^*( \bZ , X , h ) r_t ) $, we have $ [ w_j ] = [ r_t u_2 r_t ] $, so by (\ref{eqProductOfUnitaries}), we have 
\begin{equation}\label{eqRsK1Classes}
[ r_t \widehat{u} r_t ] = J_{ t , 1 } [ r_t u_2 r_t ].
\end{equation} 

We now show that $ [ r_t u_2 r_t ] \neq 0 $. First note that $ r_t v_2 r_t  \in r_t A_2 r_t $, and since $ r_t A_2 r_t$ is finite-dimensional, we have $ [ r_t v_2 r_2 ] = 0 $. From Lemma \ref{lemmaK1ElementsFromInvariantProjections}, we have $ [ r_t u r_t ] \neq 0 $. Thus, 
\begin{align}
[ r_t u_2 r_t ] &= [ r_t v_2^* u r_t ] \nonumber \\
&= -[ r_t v_2 r_t ] + [ r_t u r_t ] \nonumber \\
&\neq 0.\label{eqRsU2K1ClassNonzero}
\end{align}
By Lemma \ref{lemmaTorsionFreeK1}, $ r_t C^*( \bZ , X , h ) r_t $ has torsion-free $ K_1 $, so (\ref{eqRsK1Classes}) and (\ref{eqRsU2K1ClassNonzero}) give us 
\begin{equation}\label{eqRsWidehatUK1ClassNonzero}
[ r_t \widehat{u} r_t ] \neq 0 .
\end{equation}
A very straightforward computation shows 
\[
r_t \widehat{ u } r_t = p_t \widehat{ u } p_t + ( r_t - p_t ) 
\]
This fact combined with (\ref{eqRsWidehatUK1ClassNonzero}) and Lemma \ref{lemmaK1Corners} gives us $ [ p_t \widehat{u} p_t ] \neq 0 $ in $K_1( p_t C^*(\bZ,X,h) p_t )$. Thus, $\spec( p_t \widehat{u} p_t ) = S^1 $. So because of this, because $ p_t \widehat{ u } p_t $ commutes with $ e_{ i , j }^{ ( t , 1 )} $ for all $ i , j \in \{ 0 , \ldots , J_{ t , 1 } - 1 \}$, and because $ p_t \widehat{ u } p_t $ and $ ( e_{ i , j }^{ ( t , 1 )} )_{ 0 \leq i , j \leq  J_{ t , 1 } - 1 } $ generate $ p_t \widehat{ A } p_t $, by Lemma \ref{lemmaMnOfS1Isomorphism}, we have
\begin{equation}\label{eqFinalIsom1} 
p_t \widehat{ A } p_t \cong C( S^1 ) \otimes M_{ J_{ t , 1 } }. 
\end{equation}

Altogether, from (\ref{eqFinalIsom2}) and (\ref{eqFinalIsom1}), we get 
\[
\widehat{ A }  \cong \bigoplus_{ t = 1 }^{ T } \left(  \left( C ( S^1 ) \otimes M_{ J_{ t , 1 } } \right)\oplus \left( \bigoplus_{ k = 2 }^{ K_t } M_{ J_{ t , k } } \right) \right) , 
\]
finishing the proof of the theorem.
\end{proof}

We now illustrate the proof of Theorem \ref{thmMainTheorem} with various examples. We adopt the notation of the proof of Theorem \ref{thmMainTheorem} in these examples. In Example \ref{exmpCantor}, we give a example of a well-known minimal zero-dimensional system to help the reader get an idea about what's happening before moving onto something more complicated. In Example \ref{exmpIntegers}, we give an example of an essentially minimal zero-dimensional system that contains a periodic point. Finally, in Example \ref{exmpFinalExample}, we give an example of a fiberwise essentially minimal zero-dimensional system.

\begin{exmp}\label{exmpCantor}
Let $ X = \{ 0 , 1 \}^\bN $ be the Cantor set and let $ h $ be the 2-odometer action, which is the homeomorphism that takes an element of $ X $ and sends it to the element of $ X $ which is cofinal after the first occurrence of 0 in the sequence, with 1 replacing this first 0, and with 0 replacing all preceding 1's; additionally, $ h $ sends the sequence of all 1's to the sequence of all 0's. For examples of how $ h $ behaves,
\begin{align*}
h ( ( 0 , 1 , 0 , 1 , 1 , \ldots ) ) &= ( 1 , 1 , 0 , 1 , 1 , \ldots ), \\
h ( ( 1 , 1 , 0 , 1 , 1 , \ldots ) ) &= ( 0 , 0 , 1 , 1 , 1 , \ldots ), \\
h ( ( 0 , 0 , 1 , 1 , 1 , \ldots ) ) &= ( 1 , 0 , 1 , 1 , 1 , \ldots ).
\end{align*}
This is a classic example of a minimal zero-dimensional system.

Let $ \sP $ be a partition of $ X $ and let $ N \in \bZ_{ > 0 } $. For each $ n \in \{ 0 , \ldots , N \} $, let $ U_n \in \sP $ be the set containing $ h^n ( x ) $ where $ x = ( 0 , 0 , 0 , \ldots ) $. 

Let $ n \in \{ 0 , \ldots , N \} $ and write $ h^n ( x ) = ( x_1 , x_2 , x_3 , \ldots ) $. There is an $ L_n \in \bZ_{ >0 } $ and a compact open set $ V_n $ such that $ V_n \subset U_n $ and
\[
V_n = \big\{ ( y_1 , y_2 , y_3 , \ldots ) \setdiv \mbox{$ x_l = y_l $ for all $ l \in \{ 1 , \ldots , L_n \} $} \big\}.
\]

Let $ V = \bigcap_{ n = 0 }^N h^{ -n } ( V_n ) $. Then $ V $ is a compact open subset of $ X $ containing $ x $. Let $ X_1 $ be a compact open subset of $ V $ such that there is an $ L \in \bZ_{ > N } $ such that
\[
X_1 = \big\{ ( y_1 , y_2 , y_3 , \ldots ) \setdiv \mbox{$ x_l = 0 $ for all $ l \in \{ 1 , \ldots , L \} $} \big\}.
\]
Set $ T = 1 $, $ K_1 = 1 $, $ Y_{ t , 1 } = X_1 $, and $ J_{ t , 1 } = 2^L $. One can check that $ \systembasic $ is a system of finite first return time maps subordinate to $ \sP $ that satisfies the conclusions of Lemma \ref{lemmaTheLastLemma}.

Since $ h^{ J_{ t , 1 } } ( Y_{ t , 1 } ) = X_1 = Y_{ t , 1 } $, we can set $ \systemarg{ \prime } $ equal to $ \sS $. Then $ A_1 = A_2 \cong M_{ 2^L } $. Since $ v_1 = v_2 $, $ v_2 v_1^* = 1 $, and so we can take $ w = 1 $ and therefore $ z = 1 $. This should make sense, as $ v_1 u_2 = v_1 v_2^* u = u $, and so the $ C^* $-algebra $ A $ generated by $ A_1 $ and $ u_2 $ actually contains $ u $, and so we certainly don't need to do any trickery with $ z $ to obtain an approximation of $ u $. Following the rest of the proof, it is easy to see the approximating circle algebra we construct is isomorphic to 
\[
C ( S^1 ) \otimes M_{ 2^L } 
\]
\end{exmp}

\begin{exmp}\label{exmpIntegers}
Let $ X = \bZ \cup \{ \infty \} $ be the one-point compactification of the integers and let $ h $ be the shift homeomorphism; specifically, $ h ( x ) = x + 1 $ for all $ x \in \bZ $ and $ h ( \infty ) = \infty $. This is a classic example of an essentially minimal zero-dimensional system whose minimal set is $ \{ \infty \} $ (see Example \ref{exmpsSystemsOfFiniteFirstReturnTimeMaps}(\ref{exmpsSystemsOfFiniteFirstReturnTimeMaps(c)})). Note that the minimal set is finite, and so $ ( X , h ) $ indeed has a periodic point.

Let $ \sP $ be a partition of $ X $ and let $ N \in \bZ_{ > 0 } $. Let $ U $ be the element of $ \sP $ that contains $ \infty $ and let $ X_1 = \{ \infty \} \cup \left( \left( ( - \infty , a ] \cup [ b , \infty ) \right) \cap \bZ \right ) $ be a compact open subset of $ U $ such that $ b - a > N $. Set $ T = 1 $, set $ K_1 = 2 $, set $ Y_{ 1 , 1 } = \{ \infty \} \cup \left( \left( ( - \infty , a - 1 ] \cup [ b , \infty ) \right) \cap \bZ \right ) $, set $ Y_{ 1 , 2 } = \{ a \} $, set $ J_{ 1 , 1 } = 1 $, and set $ J_{ 1 , 2 } =  b - a $. It can be verified that $ \systembasic $ is a system of finite first return time maps subordinate to $ \sP $ that satisfies the conclusions of Lemma \ref{lemmaTheLastLemma}.

Define $ T' = 1 $ and $ X_1' = h^{ J_{ 1 , 1 } } ( Y_{ 1 , 1 } ) =  \{ \infty \} \cup \left( \left( ( - \infty , a ] \cup [ b + 1 , \infty ) \right) \cap \bZ \right ) $. Set $ K_1' = 2 $, set $ Y_{ 1 , 1 }' =  \{ \infty \} \cup \left( \left( ( - \infty , a - 1 ] \cup [ b + 1 , \infty ) \right) \cap \bZ \right ) $, set $ Y_{ 1 , 2 }' = \{ a \} $, set $ J_{ 1 , 1 }' = 1 $, and set $ J_{ 1 , 2 }' = b - a + 1 $. Then it can be verified that $ \systemarg{ \prime } $ is a system of finite first return time maps subordinate to $ \sP $ and indeed $ \sP_1 ( \sS' ) $ is finer than both $ \sP_1 ( \sS ) $ and $ \sP_2 ( \sS ) $.

We can see that $ A_1 \cong \bC \oplus M_{ b - a } $ and $ A_2 \cong \bC \oplus M_{ b - a + 1 } $. We compute $ Y = \{ a \} \sqcup \{ b \} $ and so $ \chi_Y A_2 \chi_Y \cong M_2 $. We see that $ \chi_Y v_2v_1^* \chi_Y = \begin{psmallmatrix} 0 & 1 \\ 1 & 0 \end{psmallmatrix} $. Thus, we can take $ w $ to be $  \begin{psmallmatrix} 0 & 1 \\ 1 & 0 \end{psmallmatrix} $, then we can take $ w = \begin{psmallmatrix} \cos ( \theta ) e^{ i \theta }  & \sin ( \theta ) e^{ - i \theta }  \\ \sin ( \theta ) e^{ - i \theta }   & \cos ( \theta ) e^{ i \theta }  \end{psmallmatrix} $ where $ \theta = 2\pi/N $ (this was discussed in Example \ref{exmpBergs}).

To get an idea of what $ z A_1 z^* $ looks like, we compute the following for each $ j \in \{ 0 , \ldots , N - 1 \} $.
\begin{align*}
z \chi_{ h^j ( \{ a \} ) } z^* &= \sin^2 \left( \frac{ \pi j }{ 2 N } \right) \chi_{ \{ a \} } + \cos \left( \frac{ \pi j }{ 2 N } \right) \sin \left( \frac{ \pi j }{ 2 N } \right) \left( \chi_{ \{ b \} } u^{ J_{ 1 , 2 } } \chi_{ \{ a \} } + \chi_{ \{ a \} } u^{ - J_{ 1 , 2 } } \chi_{ \{ b \} } \right) \\
&\hspace{0.4cm}+ \cos^2 \left( \frac{ \pi j }{ 2 N } \right) \chi_{ \{ b \} } 
\end{align*}
So at $ j = 0 $, $ z \chi_{ h^j ( \{ a \} ) } z^* $ is $ \{ b \} $, and then as $ j $ increases to $ N $, $ z \chi_{ h^j ( \{ a \} ) } z ^* $ slowly becomes $ \chi_{ h^N ( \{ a \} ) } $, with an error less than $ \eps $ at each step. This computation helps us illustrate how $ z v_1 u_2 z^* $ approximates $ u $; one of the problems of $ v_1 u_2 $ by itself is that it sends $ h^{ J_{ 1 , 2 } - 1 } ( Y_{ 1 , 2 } ) $ to $ Y_{ 1 , 2 } $, which is certainly not what $ u $ does. However, after conjugating by $ z $, $ Y_{ 1 , 2 } $ becomes $ h^{ J_{ 1 , 2 } } ( Y_{ 1 , 2 } ) $. The cost of this is that we slightly mess up what $ v_1 u_2 $ does to $ h^j ( Y_{ 1 , 2 } ) $ for $ j \in \{ 0 , \ldots , N - 1 \} $; before conjugating by $ z $, $ v_1 u_2 $ actually does act as $ u $ on $ h^j ( Y_{ 1 , 2 } ) $; however, with the fix we implemented on $ h^{ J_{ 1 , 2 } } ( Y_{ 1 , 2 } ) $, we must alter this a bit in order to ensure that we are conjugating everything in $ A_1 $ by the same unitary $ z $. Luckily, with enough room (at least $ N $ spaces of movement), this error is less than $ \eps $.

By following the rest of the proof, we eventually see that our approximating circle algebra isomorphic to
\[
( C ( S^1 ) \otimes M_{ J_{ 1 , 1 } } ) \oplus M_{ J_{ 1 , 2 } } = C ( S^1 ) \oplus M_{ b - a }.
\]
\end{exmp}

\begin{exmp}\label{exmpFinalExample}
Let $ X' = \{ 0 , 1 \}^\bN $ be the cantor set, let $ h' $ be the 2-odometer action, and let $ Z = \bZ \cup \{ \infty \} $ be the one-point compactification of the integers. Let $ X = ( X' \times Z ) / ( X' \times \{ \infty \} ) $, let $ \pi : X' \times Z \to X $ be the quotient map, and let $ h = \pi ( h' \times \id ) $ under the quotient map (see Example \ref{exmpsFiberwiseEssentiallyMinimal}(\ref{exmpsFiberwiseEssentiallyMinimal(c)})). 

Let $ \sP $ be a partition of $ X $. Let $ U $ be the element of $ \sP $ that contains $ \pi ( X' \times \{ \infty \} ) $. Then there is a set $ V \subset Z $ of the form $ \left( \left( ( - \infty , a ] \cup [ b , \infty ) \right) \cap \bZ \right ) $ such that $ X_1 = \pi ( X' \times V ) $ is a subset of $ U $. 

Let $ k \in [ a + 1 , b - 1 ] \cap \bZ $, set $ x = ( 0 , 0 , 0 , \ldots ) \in X' $, let $ n \in \{ 0 , \ldots , N \} $, and write $ h^n ( x ) = ( x_1 , x_2 , x_3 , \ldots ) $. There is an $ L_{ k , n } \in \bZ_{ > N } $ and a compact open set $ V_{ k , n } $ such that $ \pi ( V_{ k , n } \times \{ k \} ) $ is contained an in element of $ \sP $ and 
\[
V_{ k , n } = \big\{ ( y_1 , y_2 , y_3 , \ldots ) \setdiv \mbox{$ x_l = y_l $ for all $ l \in \{ 1 , \ldots , L_{ k , n } \} $} \big\}.
\]
Let $ V = \bigcap_{ k = a + 1 }^{ b - 1 } \bigcap_{ n = 0 }^N h^{ -n } ( V_{ k , n } ) $. For each $ k \in [ a + 1 , b - 1 ] \cap \bZ $, set $ X_{ k - a + 1 } = \pi ( \{ k \} \times V ) $. Note that this defines $ X_2 , X_3 , \ldots , X_{ b - a + 2 } $.

Set $ L = \max_{ k , n } L_{ k , n } $ and set $ T = b - a + 2 $. Set $ K_1 = 1 $, set $ Y_{ 1 , 1 } = X_1 $, and set $ J_{ 1 , 1 } = 1 $. For each $ k \in \{ 2 , \ldots , T \} $, set $ K_t = 1 $, set $ Y_{ t , 1 } = X_t $, and set $ J_{ t , 1 } = 2^L $. Then $ \systembasic $ is a system of finite first return time maps subordinate to $ \sP $ that satisfies the conclusions of Lemma \ref{lemmaTheLastLemma}.

The rest of the computations are similar to Example \ref{exmpCantor}. We compute $ A_1 \cong A_2 \cong \bC \oplus \left( \bigoplus_{ t = 2 }^T M_{ 2^L } \right) $, $ z = 1 $, and eventually we see that the approximating circle algebra we construct is isomorphic to 
\[
C( S^1 ) \oplus \left( \bigoplus_{ t = 2 }^T ( C( S^1 ) \otimes M_{ 2^L } ) \right) .
\]
\end{exmp}

\bibliographystyle{plain}
\bibliography{researchbib}

\begin{thebibliography}{10}

\bibitem{Bratteli72}
O.~{Bratteli}.
\newblock {Inductive limits of finite dimensional {$C^*$}-algebras}.
\newblock {\em Trans. Amer. Math. Soc.}, 171:195--234, 1972.

\bibitem{DadarlatGong97}
M.~{Dadarlat} and G.~{Gong}.
\newblock {A classification result for approximately homogeneous
  {$C^*$}-algebras of real rank zero}.
\newblock {\em Geom. Funct. Anal.}, 7:646--711, 1997.

\bibitem{Elliott76}
G.~{Elliott}.
\newblock {On the Classification of Inductive Limits of Sequence of Semisimple
  Finite-Dimensional Algebras}.
\newblock {\em J. Algebra}, 38:29--44, 1976.

\bibitem{Elliott93}
G.~{Elliott}.
\newblock {On the classification of {$ C^* $}-algebras of real rank zero}.
\newblock {\em J. Reine Angew. Math.}, 443:179--219, 1993.

\bibitem{ElliottGong96}
G.~{Elliott} and G.~{Gong}.
\newblock {On the classification of {$ C^* $}-algebras of real rank zero, II}.
\newblock {\em Ann. Math.}, 144(3):497--610, 1996.

\bibitem{ElliotGongLinZiu16}
G.~{Elliott}, G.~{Gong}, H.~{Lin}, and Z.~{Niu}.
\newblock {On the classification of simple amenable {$ C^* $}-algebras with
  finite decomposition rank, II}.
\newblock 2016.

\bibitem{GaHiSa17}
E.~{Gardella}, I.~{Hirshberg}, and L.~{Santiago}.
\newblock {Rokhlin dimension: duality, tracial properties, and crossed
  products}.
\newblock {\em Ergod. Th. \& Dynam. Sys.}, 41.

\bibitem{HermanPutnamSkau92}
R.H. {Herman}, I.~{Putnam}, and C.~{Skau}.
\newblock {Ordered Bratteli diagrams, dimension groups, and topological
  dynamics}.
\newblock {\em Int. J. Math.}, 3:827--864, 1992.

\bibitem{Pasnicu05}
C.~{Pasnicu}.
\newblock Real rank zero and continuous fields of {$C^*$}-algebras.
\newblock {\em Bull. Math. Soc. Sc. Math. Roumanie}, 48(96)(3):319--325, 2005.

\bibitem{Phillips99}
N.~C. {Phillips}.
\newblock A classification theorem for nuclear purely infinite simple
  {$C^*$}-algebras.
\newblock {\em Doc. Math}, 5:49--114, 2000.

\bibitem{PimsnerVoiculescu80}
M.~{Pimsner} and D.~{Voiculescu}.
\newblock {Exact sequences for {$ K $}-groups and Ext-groups of certain
  cross-product {$ C^* $}-algebras}.
\newblock {\em J. Op. Th.}, 4:93--118, 1980.

\bibitem{Putnam89}
I.~{Putnam}.
\newblock {The {$C^*$}-algebras associated to minimal homeomorphisms of the
  Cantor set}.
\newblock {\em Pac. J. Math.}, 136(2):329--353, 1989.

\bibitem{Putnam90}
I.~{Putnam}.
\newblock {On the topological stable rank of certain transformation group {$
  C^* $}-algebras}.
\newblock {\em Ergod. Th. \& Dynam. Sys.}, 10:197--207, 1990.

\bibitem{WeggeOlsen93Book}
N.E. Wegge-Olsen.
\newblock {\em {$ K $}-Theory and {$C^*$}-algebras: A Friendly Approach}.
\newblock Oxford Science Publications. Oxford University Press, 1993.

\end{thebibliography}

\end{document}